\documentclass[12pt,leqno,a4paper]{amsart}
\usepackage{amssymb,enumerate}
\newcommand{\FF}{{\mathbb{F}}}
\newcommand{\ZZ}{{\mathbb{Z}}}

\newcommand{\fA}{{\mathfrak{A}}}   
\newcommand{\fS}{{\mathfrak{S}}}   

\newcommand{\bB}{{{B}}}
\newcommand{\bC}{{{C}}}
\newcommand{\bG}{{{G}}}
\newcommand{\bH}{{{H}}}
\newcommand{\bL}{{{L}}}
\newcommand{\bM}{{{M}}}
\newcommand{\bP}{{{P}}}
\newcommand{\bQ}{{{Q}}}
\newcommand{\bT}{{{T}}}
\newcommand{\bU}{{{U}}}
\newcommand{\bX}{{{X}}}

\newcommand{\cE}{{\mathcal{E}}}
\newcommand{\cO}{{\mathcal{O}}}
\newcommand{\cU}{{\mathcal{U}}}

\newcommand{\Ph}[1]{\Phi_#1}

\newcommand{\CF}{{\operatorname{CF}}}
\newcommand{\Ind}{{\operatorname{Ind}}}
\newcommand{\Infl}{{\operatorname{Infl}}}
\newcommand{\Irr}{{\operatorname{Irr}}}
\newcommand{\Res}{{\operatorname{Res}}}
\newcommand{\ad}{{\operatorname{ad}}}
\newcommand{\SC}{{\operatorname{sc}}}      
\newcommand{\Aut}{{\operatorname{Aut}}}
\newcommand{\Br}{{\operatorname{Br}}}
\newcommand{\GL}{{\operatorname{GL}}}
\newcommand{\PGL}{{\operatorname{PGL}}}
\newcommand{\SL}{{\operatorname{SL}}}
\newcommand{\OO}{{\operatorname{O}}}

\newcommand{\GU}{{\operatorname{GU}}}
\newcommand{\SU}{{\operatorname{SU}}}
\newcommand{\PGU}{{\operatorname{PGU}}}

\newcommand\RLG{{R_\bL^\bG}}
\newcommand\RLPG{{R_{\bL\subset\bP}^\bG}}
\newcommand\RLM{{R_\bL^\bM}}
\newcommand\RMG{{R_\bM^\bG}}
\newcommand\RLH{{R_\bL^\bH}}
\newcommand\RHG{{R_\bH^\bG}}
\newcommand\RTG{{R_\bT^\bG}}

\newcommand\tbG{{\tilde\bG}}
\newcommand\tbL{{\tilde\bL}}
\newcommand\tbZ{{\tilde{\mathbf{Z}}}}

\newcommand{\tw}[1]{{}^#1\!}
\newcommand{\pl}{\!+\!}
\newcommand{\cd}{\!\cdot\!}
\newcommand{\blangle}{\big\langle}
\newcommand{\brangle}{\big\rangle}

\def\pmod#1{~({\rm mod}~#1)}

\newcommand{\thrd}{\frac{1}{3}}
\newcommand{\chrs}{\text{ chars}}

\let\la=\lambda
\let\ti=\times

\newtheorem{thm}{Theorem}[section]
\newtheorem{lem}[thm]{Lemma}

\newtheorem{cor}[thm]{Corollary}

\newtheorem{prop}[thm]{Proposition}

\theoremstyle{definition}
\newtheorem{exmp}[thm]{Example}
\newtheorem{defn}[thm]{Definition}

\theoremstyle{remark}
\newtheorem{rem}[thm]{Remark}

\begin{document}

\title{Quasi-isolated blocks and\\ Brauer's height zero conjecture}

\date{\today}

\author{Radha Kessar}
\address{Institute of Mathematics, University of Aberdeen,
  Fraser Noble Building, Aberdeen, AB243UE, UK}
\email{r.kessar@abdn.ac.uk}
\author{Gunter Malle}
\address{FB Mathematik, TU Kaiserslautern,
  Postfach 3049, 67653 Kaisers\-lautern, Germany.}
\email{malle@mathematik.uni-kl.de}

\thanks{The project was supported by EPSRC grant EP/I033637/1}

\begin{abstract}
This paper has two main results. Firstly, we complete the parametrisation of
all $p$-blocks of finite quasi-simple groups by finding the so-called
quasi-isolated blocks of exceptional groups of Lie type for bad primes.
This relies on the explicit decomposition of Lusztig induction from suitable
Levi subgroups. Our second major result is the proof of one
direction of Brauer's long-standing height zero conjecture on blocks of
finite groups, using the reduction by Berger and Kn\"orr to the quasi-simple
situation. We also use our result on blocks to verify a conjecture of Malle
and Navarro on nilpotent blocks for all quasi-simple groups.
\end{abstract}

\maketitle

\pagestyle{myheadings}
\markboth{R. Kessar and G. Malle}{Quasi-isolated blocks and the height zero conjecture}

\section{Main results} \label{sec:main}

Brauer's famous height zero conjecture \cite{Br55} from 1955 states that
a $p$-block $B$ of a finite group has an abelian defect group if and only if
every ordinary irreducible character in $B$ has height zero.

Here we are concerned with one direction of this conjecture:
\begin{itemize}
\item[(HZC1)] If a $p$-block $B$ of a finite group has abelian defect groups,
 then every ordinary irreducible character of $B$ has height zero.
\end{itemize}
One of the main aims of this paper is the proof of the following result:

\begin{thm}   \label{thm:main}
 The 'if part' (HZC1) of Brauer's height zero conjecture holds for all finite
 groups.
\end{thm}

Our proof relies on the crucial paper of Berger and Kn\"orr \cite{BK} where
they show that this direction of the conjecture holds for all groups, provided
that it holds for all quasi-simple groups. An alternative proof of this
reduction was later given by Murai \cite{Mu}.

Many particular cases of (HZC1) had been considered before. Olsson \cite{Ol90}
showed the claim for the covering groups of alternating groups. The case of
unipotent blocks of groups of Lie type was treated by Brou\'e--Malle--Michel
\cite{BMM} and Brou\'e--Michel \cite{BM93}. Cabanes--Enguehard \cite{CE93}
then showed (HZC1) for most blocks of finite reductive groups. In addition to
these results we use the theorem of Blau--Ellers \cite{BE} that this direction
of the conjecture holds for all central quotients of special linear and special
unitary groups.
\par
Let us mention a few other important partial results: Gluck and Wolf \cite{GlWo}
proved both directions of the height zero conjecture for $p$-solvable groups.
Fong--Harris showed (HZC1) for principal 2-blocks,  Navarro--Tiep
\cite{NT10} recently proved both directions for 2-blocks of maximal defect and
Kessar--Koshitani--Linckelmann \cite{KKL} proved (HZC1) for 2-blocks whose
defect groups are  elementary abelian of order $8$.  Our
paper is independent of the latter results.

As our second main result and as a crucial ingredient to the proof of
Theorem~\ref{thm:main} we complete the parametrisation of the $\ell$-blocks of
finite quasi-simple groups, where $\ell$ is a prime number. (See
Remark~\ref{rem:history} for historic comments on this problem.) The only
case that remains to be considered is the one of quasi-isolated blocks of
exceptional groups of Lie type when $\ell$ is bad, that is, non-unipotent
blocks  parametrised by non-identity semisimple elements whose centraliser
in the dual group is not contained in any proper Levi subgroup. This is the
case which we solve here.

Although our determination of quasi-isolated blocks proceeds in a case-by-case
manner, the result on blocks and their defect groups can be phrased in the
following general, generic form, which also appeared for the blocks considered
in the earlier work of Cabanes and Enguehard. Throughout this introduction,
$\bG$ denotes a simple simply connected algebraic group over an algebraic
closure of a finite field $\FF_p$ with Steinberg endomorphism $F:\bG \to\bG$.
See Sections~2--6 for further notation and the proofs.

\begin{thm}[Parametrization of Blocks]   \label{thm:mainblocks}
 Assume that $\bG$ is simple simply connected of exceptional Lie type in
 characteristic~$p$ and $\ell\ne p$ a bad prime for $\bG$.
 Let $1\ne s\in \bG^{*F}$ be a quasi-isolated $\ell'$-element.
 \begin{enumerate}[\rm(a)]
  \item There is a natural bijection
   $$b_{\bG^F}(\bL,\la)\longleftrightarrow (\bL,\la)$$
   between $\ell$-blocks of $\bG^F$ in $\cE_\ell(\bG^F,s)$ and pairs
   $(\bL,\la)$ up to $\bG^F$-conjugation, where
   \begin{enumerate}
    \item[\rm(1)] $\bL$ is an $e$-split Levi subgroup of $\bG$, with
     $e=e_\ell(q)$,
    \item[\rm(2)] $\la\in\cE(\bL^F,s)$ is $e$-cuspidal, and
    \item[\rm(3)] $\la$ is of quasi-central $\ell$-defect.
   \end{enumerate}
  \item There is a defect group $P\leq N_\bG(\bL,\la)^F$ of
   $b_{\bG^F}(\bL,\la)$ with a normal series
   $$Z(\bL)_\ell^F \unlhd D:=C_P(Z(\bL)_\ell^F) \unlhd P,$$
   with quotients $P/D$ isomorphic to a  Sylow $\ell$-subgroup
   of $W_{\bG^F}(\bL,\la)$ and $D/Z(\bL)_\ell^F$ isomorphic to a Sylow
   $\ell$-subgroup of $\bL^F/Z(\bL)_\ell^F[\bL,\bL]^F$.
  \item Here, $b_{\bG^F}(\bL,\la)$ has abelian defect if and only if
   $W_{\bG^F}(\bL,\la)$ is an $\ell'$-group.
  \item Further, when $\ell\ne2$ then $D=Z(\bL)_\ell^F$ in {\rm(b)} and
   $P$ is a Sylow $\ell$-subgroup of the extension of $Z(\bL)_\ell^F$ by
   $W_{\bG^F}(\bL,\la)$.
 \end{enumerate}
\end{thm}

In \cite{BR}, Bonnaf\'e and Rouquier proved that every $\ell$-block of a
finite reductive group in characteristic different from $\ell$ is Morita
equivalent, via Lusztig induction, to a quasi-isolated block of some Levi
subgroup. This comparison theorem provides a crucial reduction in the proof
of Theorem~\ref{thm:main}. But note that it is not known in general whether
Morita equivalences preserve abelianess of defect groups. In our context,
relying on previous results, mainly of Cabanes--Enguehard, we prove the
following result:

\begin{thm}[Preservation of Abelian Defect Groups]   \label{thm:all-ab-intro}
 Let $\bG$ be simple, simply connected in characteristic~$p$ and $\ell\ne p$
 a prime. Let $\bM$ be an $F$-stable Levi subgroup of $\bG$, and
 let $b$ and $c$ be Bonnaf\'e--Rouquier corresponding $\ell$-blocks of
 $\bG^F$ and $\bM^F$ respectively (see Definition~\ref{dfn:bon-rou-corr}).
 Let $Z$ be a central $\ell$-subgroup of $\bG^F$ and let $\bar b$ and
 $\bar c$ be the images of $b$ and $c$ in $\bG^F/Z$ and $\bM^F/Z$ respectively.
 If either $\bar b$ or $\bar c$ has abelian defect groups, then the defect
 groups of $\bar b$ and $\bar c$ are isomorphic.
\end{thm}

The above result should ideally follow from general properties of the
bimodules inducing Bonnaf\'e--Rouquier Morita equivalences, but our proof is
different. In fact, one expects that if $b$ and $c$ are Bonnaf\'e--Rouquier
correspondents, then any defect group of $c$ is a defect group of $b$ ---
this is known to hold in many cases.

In order to prove Theorem~\ref{thm:mainblocks} we apply a criterion of Cabanes
and Enguehard (see Proposition~\ref{prop:Cabanescrux} below) which allows one
to determine the blocks if Lusztig induction from suitable Levi subgroups can
be shown to satisfy a generalised Harish-Chandra theory. The following
result is not only a crucial ingredient for our proofs but of independent
interest:

\begin{thm}[$e$-Harish-Chandra Theory]   \label{thm:BMM}
 Assume that $\bG$ is simple simply-connected of exceptional Lie type in
 characteristic~$p$ and $\ell\ne p$ a bad prime for $\bG$. Let $s\in \bG^{*F}$
 be a quasi-isolated $\ell'$-element. Then with $e=e_\ell(q)$ we have:
 \begin{enumerate}[\rm(a)]
  \item The sets $\cE(\bG^F,(\bL,\lambda))$, where $(\bL,\lambda)$ runs
   over a set of representatives of the $\bG^F$-conjugacy classes of
   $e$-cuspidal pairs of $\bG$ below $\cE(\bG^F,s)$, partition $\cE(\bG^F,s)$.
  \item $\bG^F$ satisfies an $e$-Harish-Chandra theory above each $e$-cuspidal
   pair $(\bL,\la)$ of $\bG^F$ below $\cE(\bG^F,s)$ (see
   Definition~\ref{def:d-HC} below).
 \end{enumerate}
\end{thm}

The case when $s=1$, that is, the case of unipotent characters, was the main
result in \cite[Thm.~3.2]{BMM} (where there was no restriction on the type
of $\bG$, but $\ell$ was assumed to be large enough).
\medskip

Finally, we use the previous results to characterise blocks of quasi-simple
groups all of whose height zero characters have the same degree, thus completing
a programme begun by Malle--Navarro \cite{MN}, and continued by Gramain
\cite{Gr10} for the case of spin-blocks of alternating groups:

\begin{thm}[Characterization of Nilpotent Blocks]   \label{thm:nil}
 Let $S$ be a finite quasi-simple group and $p$ a prime. Assume that $B$ is
 a $p$-block of $S$ all of whose height zero characters have the same
 degree. Then the defect group of $B$ is abelian and thus $B$ is nilpotent.
\end{thm}

The paper is organised as follows. In Section~\ref{sec:back} we collect
various results on groups of Lie type, Lusztig induction, blocks and Brauer
pairs and we state our main criteria for block distribution and the structure
of defect groups. In Sections~\ref{sec:F4}--\ref{sec:E8} we determine the
decomposition of Lusztig induction from suitable Levi subgroups in the Lusztig
series belonging to quasi-isolated elements of exceptional groups of rank at
least~4 and the block distribution in these series.
Section~\ref{sec:def} is devoted to showing Theorem~\ref{thm:all-ab-intro}.
The remaining steps of the proof of (HZC1) are given in
Section~\ref{sec:proof}, see Theorem~\ref{thm:BrauerHZC}. Finally,
in Section~\ref{sec:nilblocks} we prove Theorem~\ref{thm:nil}.

\section{Background results and methods} \label{sec:back}

Throughout this paper, $\ell$ denotes a prime number.

\subsection{Blocks and Brauer pairs}
Let $G$ be a finite group and let $(K,\cO,k)$ be a splitting modular system
for $G$, i.e., $\cO$ is a complete discrete valuation ring with residue field
$k$ of characteristic $\ell$ and field of fractions $K$ such that $k$ and $K$
are splitting fields for all groups involved in $G$. Let $\CF(G, K)$ denote
the set of $K$-valued class functions on $G$ and let $\Irr(G)$
denote the subset of $\CF(G, K)$ consisting of irreducible characters of $G$.
Let $\langle\ ,\ \rangle_G$ denote the standard inner product on $\CF(G, K)$.

By an $\ell$-block of $G$ we will mean a primitive idempotent of $Z(kG)$.
By idempotent lifting, the canonical surjection of $\cO$ onto
$k$ induces a bijection between the set of primitive idempotents of
$Z(\cO G)$ and primitive idempotents of $Z(kG)$, and this induces an
orthogonal decomposition of the set of $K$-valued class functions
$\CF(G, K)$ on $G$ with respect to the standard inner product.
For $f \in \CF(G, K)$ and $b$ an $\ell$-block of $G$, the projection of $f$
onto the component of $b$ in $\CF(G, K)$ is denoted by $b.f$, and we write
$b=b_G(f)$ if $f =b.f$.
This defines a partition $\Irr(G) = \coprod_b\Irr(b)$, where
$\Irr(b)= \{\chi\in\Irr(G) \mid b.\chi=\chi\}$.

A \emph {Brauer pair} of $G$ (or \emph{$G$-Brauer pair}) with respect to $\ell$
is a pair $(Q, c)$, such that $Q$ is an $\ell$-subgroup of $G$ and $c$ is
an $\ell$-block of $C_G(Q)$. The set of $G$-Brauer pairs has a structure of
a $G$-poset such that the following properties hold: If $(Q,c)$ and $(R,d)$
are Brauer pairs with $(R,d) \subseteq (Q,c)$, then $R \leq Q$, and for any
Brauer pair $(Q,c)$ and any subgroup $R$ of $Q$, there is a unique Brauer pair
$(R,d)$ such that $(R,d) \subseteq (Q,c)$.
In particular, for each Brauer pair $(Q,c)$, there exists a unique
$\ell$-block, say $b$ of $G$ such that $(\{1\},b)\subseteq(Q,c)$, and in this
case we say that $(Q,c)$ is a \emph{$b$-Brauer pair} or that $(Q,c)$ is
\emph{associated to $b$}. A Brauer pair $(Q, c)$ is a $b$-Brauer pair if
and only if $\Br_Q(b)c=c$, if and only if $\Br_Q(b)c \ne 0$, where
${\Br}_Q: (kG)^Q \to kC_G(Q)$ denotes the Brauer homomorphism.
\par
For an $\ell$-block $b$ of $G$, the subset of the set of Brauer pairs of $G$
associated to $b$ is closed under inclusion and under the action of $G$.
For any Brauer pair $(Q,c)$, $Z(Q)$ is contained in every defect group of
$c$ and $(Q,c)$ is said to be \emph{centric} (or \emph{self-centralising})
if $Z(Q)$ is a defect group of $c$. A Brauer pair $(Q,c)$ is maximal if and
only if $(Q,c)$ is centric and $N_G(Q,c)/QC_G(Q)$ is an $\ell'$-group, where
$N_G(Q,c)$ denotes the stabiliser in $G$ of $(Q,c)$.
Further, $(Q,c)$ is maximal if and only if $Q$ is a defect group of
the unique $\ell$-block of $G$ to which $(Q,c)$ is associated.
$G$ acts transitively on the subset of maximal $b$-Brauer pairs.

If $(Q,c)$ and $(R,d)$ are Brauer pairs with $(R,d) \subseteq (Q,c)$,
and such that $R$ is normal in $Q$, then we write $(R,d) \unlhd (Q,c)$.

For a more detailed exposition on Brauer pairs, we refer the reader to the
monographs \cite[\S 40]{The95}, \cite[Part IV]{AKO}, or to the original
article of Alperin and Brou\'e \cite{AB79} --- in the latter reference
Brauer pairs are referred to as subpairs.
Here we recall a few stray facts which will be used in the sequel.

Let $R$ be an $\ell$-subgroup of $G$ and let $H$ be a subgroup of $G$ such that
$RC_G(R)\leq H \leq N_G(R)$. Every central idempotent of $kH$ is in
$kC_G(R) =kC_H(R)$ (see \cite[Part IV, Lemma 3.17]{AKO}).
Now let $(R, d)$ be a $G$-Brauer pair and suppose that
$RC_G(R)\leq H \leq N_G(R,d)$. Then, $d$ is an $\ell$-block of $H$. Further,
for any subgroup $Q$ of $H$ containing $R$, $C_G(Q)=C_H(Q)$, the $H$-Brauer
pairs with first component $Q$ are the $G$-Brauer pairs with first component
$Q$ and for any block $c$ of $C_H(Q)=C_G(Q)$,
$(\{1\},d)\subseteq(Q,c)$ as $H$-Brauer pairs if and only if
$(R, d) \subseteq(Q,c)$ as $H$-Brauer pairs, if and only if
$(R, d)\subseteq(Q,c)$ as $G$-Brauer pairs
(see \cite[Part IV, Lemma 3.18]{AKO}).
We will use these facts without further comment.

We will need a few facts about covering blocks. For $\tilde G $ a finite
group containing $G$ as normal subgroup, $\tilde b$ an $\ell$-block of
$\tilde G$ and $b$ an $\ell$-block of $G$, $\tilde b$ is said to cover $b$ if
$\tilde b b \ne 0$.

\begin{lem}   \label{lem:local-cliff}
 Let $b$ be an $\ell$-block of $G$ and let $(A,u)\subseteq (D,v)\subseteq(P,w)$
 be $b$-Brauer pairs such that $D $ is maximal with respect to
 $D\leq AC_G(A)$ and $P$ is maximal with respect to $P \leq N_G(A, u)$.
 Let $\tilde G$ be a finite group with $G \unlhd\tilde G$. Then:
 \begin{enumerate}[\rm(a)]
  \item  $D$ is a defect group of the block $u$ of $AC_G(A)$ and $P$ is a
   defect group of the block $u$ of $N_G(A,u)$. Further, $D =P \cap AC_G(A)$
   and $P/D$ is isomorphic to a Sylow $\ell$-subgroup of $N_G(A,u)/AC_G(A)$.
  \item  Let $\tilde b$ be an $\ell$-block of $\tilde G$ and $(A,\tilde u)$
   a $\tilde b$-Brauer pair. If $\tilde u$ covers $u$, then $\tilde b$ covers
   $b$.
  \item  There exists an $\ell$-block $\tilde b$ of $\tilde G$ covering $b$,
   and $\tilde b$-Brauer pairs $(A, \tilde u) \unlhd (\tilde P,y)$ such that
   $ \tilde P \leq N_{\tilde G}(A,u)$, $\tilde u$ covers $u$,
   $\tilde P \cap G =P$ and $\tilde P/P$ is isomorphic to a Sylow
   $\ell$-subgroup of $N_{\tilde G}(A,u)/N_G(A, u)$.
 \end{enumerate}
\end{lem}

\begin{proof}
By \cite[Part IV, Lemma 3.18]{AKO}, $(D,v)$ is a maximal
$AC_G(A)$-Brauer pair, and is associated to $u$ so $D$ is a defect group
of $u$. Similarly, $P$ is a defect group of $u$ as block of $N_G(A, u)$.
Consider the normal inclusion $AC_G(A) \unlhd N_G(A,u)$. As $u$ is the
only block of $kN_G(A,u)$ covering the block $u$ of $kC_G(A)$, by covering
block theory, $ D =P \cap A C_G(A)$ and $P/D$ is isomorphic to a Sylow
$\ell$-subgroup of $N_G(A,u)/AC_G(A)$
(see \cite[Ch.~5, Thm.~5.16]{NTs}). This proves~(a).

Let $\tilde b$ and $\tilde u$ be as in~(b) and suppose that $\tilde u$
covers $u$. Then, $\Br_A(\tilde b)\tilde u = \tilde u$, $\Br_A(b)u =u$
and $\tilde u u \ne 0$.
Since $\Br_A$ is an algebra homomorphism, and $\tilde u$ is central in
$C_{\tilde G}(A)$, $\Br_A(\tilde b b) \tilde u u =\tilde u u \ne 0$,
and it follows that $\Br_A(\tilde b b) \ne 0$, whence $\tilde b b \ne 0$,
proving~(b).

For~(c), consider the normal inclusion $N_G(A, u) \unlhd N_{\tilde G}(A,u)$.
By (a), $u$ is a block of $N_G(A, u)$ with defect group $P$. So, again by
\cite[Ch.~5, Thm.~5.16] {NTs}, there exists a block $u'$ of
$N_{\tilde G}(A,u)$ covering $u$ such that $u'$ has a defect group
$\tilde P \leq N_{\tilde G}(A, u)$
with $\tilde P \cap G=\tilde P \cap N_G(A,u) = P$ and $\tilde P/P$
isomorphic to a Sylow $\ell$-subgroup of $N_{\tilde G}(A, u)/N_G(A, u)$.
Now, $\tilde P$ being a defect group of $u'$ implies that
$\Br_{\tilde P}(u')\ne 0$. Also, $u$ is the unique block of $N_G(A,u)$
covered by $u'$, hence $u' u =u' $. So,
$$\Br_{\tilde P}(u')\Br_{\tilde P}(u) =\Br_{\tilde P}(u')\ne 0 ,$$
whence $\Br_{\tilde P}(u) \ne 0$.

Now consider the normal inclusion $C_G(A) \unlhd C_{\tilde G}(A)$. Let
$\cU$ be the set of $\ell$-blocks of $C_{\tilde G}(A)$ covering $u$.
Since $\tilde P$ normalises $C_{\tilde G}(A)$ and stabilises $u$,
$\tilde P$ acts by conjugation on $\cU$. In particular,
$\sum_{f \in \cU} f \in (k\tilde G)^{\tilde P}$.
Also, $ u (\sum_{f \in \cU} f) = u $. So,
$$\Br_{\tilde P} (u) \Br_{\tilde P}(\sum_{f \in \cU} f)=
  \Br_{\tilde P}(u) \ne 0.$$
Since $\tilde P$ permutes the elements of $\cU$, the above equation yields
that there is an element say $\tilde u$, of $\cU$ such that
$\tilde u \in (k\tilde G)^{\tilde P}$ and $\Br_{\tilde P}(\tilde u)\ne 0$.
Consequently, there exists a $\tilde G$-Brauer pair $(\tilde P, y)$ such that
$(A, \tilde u) \leq (\tilde P, y)$. Let $\tilde b$ be the
unique $\ell$-block of $\tilde G$ such that $(A, \tilde u)$ is a
$\tilde b$-Brauer pair. Since $\tilde u$ covers $u$, (b) gives that
$\tilde b$ covers $b$. This proves (c).
\end{proof}

Let $\chi \in \Irr(G)$ and let $\theta$ be a linear character of $G$. Then
$\theta \otimes \chi$ is an irreducible character of $G$ and the map
$\chi \mapsto\theta \otimes \chi $ is a permutation on $\Irr(G)$ which
respects $\ell$-blocks: for any $\ell$-block $b$ of $G$, the set
$\{\theta\otimes\chi \mid \chi\in\Irr(b)\}$ is the set of irreducible
characters of an $\ell$-block of $G$, which we will denote by
$\theta \otimes b$.
Denoting also by $\theta$ the restriction of $\theta$ to any subgroup of $G$,
the map $(Q,f)\mapsto(Q,\theta\otimes f)$ is a $G$-poset isomorphism between
the set of $b$-Brauer pairs and the set of $\theta \otimes b$-Brauer pairs.

\begin{lem}   \label{lem:cliff-ab}
 Let $\tilde G$ be a finite group such that $G \unlhd \tilde G$,  $b$ an
 $\ell$-block of $G$ and $\tilde b$ an $\ell$-block of $\tilde G$ covering $b$.
 Suppose that $\tilde G/G$ is abelian. Then:
 \begin{enumerate}[\rm(a)]
  \item  Any $\ell$-block of $\tilde G$ covering $b$ is of the form
   $\theta\otimes\tilde b$, where $\theta$ is a linear character of
   $\tilde G/G$.
  \item  Assume that $b$ has a defect group $Z \leq Z(\tilde G)$.
   Suppose that the unique character $\chi\in\Irr(b)$ containing $Z$ in its
   kernel extends to its stabiliser $I$ in $\tilde G$.
   Then, $\tilde b$ is nilpotent, and if $D$ is a defect group of
   $\tilde b$, then $D \leq I$, $D \cap G=Z$ and $D /Z $ is isomorphic to
   the Sylow $\ell$-subgroup of $I/G$. Moreover, there are $|I: G|_{\ell'}$
   $\ell$-blocks of $\tilde G$ covering $b$.
 \end{enumerate}
\end{lem}

\begin{proof}
Let $ b'$ be an $\ell$-block of $\tilde G$ covering $b$ and let
$\chi \in \Irr(b)$. Then there exists $\eta\in\Irr(\tilde b)$ and
$\eta'\in \Irr(b')$ such that $\eta$ and $\eta'$ both cover $b$
(see \cite[Ch.~5, Lemma 5.8(ii)]{NTs}).
But since $\tilde G/G $ is abelian, $\eta' =\theta \otimes \eta$ for some
linear character $\theta$ of $\tilde G/ G$. This proves (a).

Suppose that the hypotheses of (b) hold. Then $I$ is the stabiliser in
$\tilde G$ of $b$. Induction induces a bijection between the set of
$\ell$-blocks of $I$ covering $b$ and the set of $\ell$-blocks of $\tilde G$
covering $b$;
corresponding blocks under this bijection are source algebra equivalent,
(see for instance \cite[Sec.~2]{Ke}) and in particular the correspondence
preserves the nilpotency of blocks and corresponding blocks have common
defect groups. Hence, we may assume that $ I=\tilde G$.

Since $Z$ is a central subgroup of $\tilde G$, the canonical surjection
of $\tilde G$ onto $\tilde G/Z$ induces a bijection between the $\ell$-blocks
of $\tilde G$ and the $\ell $-blocks of $\tilde G/Z$
(see \cite[Ch.~5, Thm.~8.11]{NTs}),
and also between the $\ell$-blocks of $G$ and the $\ell$-blocks of $ G/Z$;
for any block $ d$ of $\tilde G$ (respectively $G$) denote by $\bar d$
the corresponding block of $\tilde G/Z$ (respectively $G/Z$).
Then a block $b'$ of $\tilde G$ covers $b$ if and only if $\bar b'$ covers
$\bar b$, $b'$ is nilpotent if and only if $\bar b'$ is nilpotent,
and the defect groups of $\bar b'$ are of the form $D/Z$, where $D$ is a
defect group of $ b'$. Further, $\chi $ extends to an irreducible character
of $\tilde G$, so $\chi$ considered as an element of $\tilde G/Z$ extends to
an irreducible character of $\tilde G/Z$.

Thus, we may assume that $Z=1$. If $\tilde G/G$ is an $\ell'$-group, the
claim is immediate. Thus we may also assume that $\tilde G/G$ has $\ell$-power
order. Then there is a unique block $\tilde b$ lying above $b$. By assumption,
$\chi$ extends to $\tilde G$, so a defect group $D$ of $\tilde b$ is isomorphic
to $\tilde G/G$ and satisfies $D\cap G=1$. In particular, $\tilde b$ is
nilpotent.
\end{proof}

\subsection{Lusztig series and $\ell$-blocks}
We set up the following notation. Let $\bG$ be a connected reductive algebraic
group over an algebraic closure of a finite field $\FF_p$ with a
Steinberg endomorphism $F:\bG\rightarrow\bG$, and $\bG^F$ the finite group
of fixed points. We are interested in the $\ell$-blocks
of $\bG^F$, where $\ell$ is a prime number different from the defining
characteristic~$p$ of $\bG$. We first recall several useful results. \par
Let $\bT$ be an $F$-stable maximal torus of $\bG$, and $\bG^*$ a group in
duality with $\bG$ with respect to $\bT$, with corresponding Steinberg
endomorphism again denoted by $F$
(see \cite[13.10]{DM91}). We denote by $q$ the absolute value of the
eigenvalues of $F$ on the character group of $\bT$, an integral power
of~$\sqrt{p}$. By the fundamental results of Lusztig, the set of complex
irreducible characters of $\bG^F$ is a disjoint union of rational Lusztig
series $\cE(\bG^F,s)$, where $s$ runs over semisimple elements of $\bG^{*F}$
up to conjugation. Lusztig series are compatible with block theory in the
following sense (see \cite[Thm.~9.12]{CE}):

\begin{thm}[Brou\'e--Michel, Hiss] \label{thm:BrMiHi}
 Let $s\in \bG^{*F}$ be a semisimple $\ell'$-element. Then:
 \begin{enumerate}[\rm(a)]
  \item The set
   $$\cE_\ell(\bG^F,s):=\bigcup_{t\in C_{\bG^*}(s)_\ell^F}\cE(\bG^F,st)$$
   is a union of $\ell$-blocks (where $t$ runs over the $\ell$-elements in
   $C_{\bG^*}(s)^F$ up to conjugation).
  \item Any $\ell$-block in $\cE_\ell(\bG^F,s)$ contains a character from
   $\cE(\bG^F,s)$.
 \end{enumerate}
\end{thm}

Thus, to parametrise the $\ell$-blocks of $\bG^F$, it suffices to decompose
$\cE(\bG^F,s)$ into $\ell$-blocks, for all $\ell'$-elements $s\in\bG^{*F}$.

We'll also use the following notation for the union of Lusztig series
corresponding to $\ell'$-elements:
$$\cE(\bG^F,\ell'):=
   \bigcup_{\text{$\ell'$-elements }s\in \bG^{*F}}\cE(\bG^F,s).$$

\subsection{Quasi-central defect and defect groups}
In  this subsection we  develop  some results which will allow us to identify
the defect groups of blocks.

\begin{defn}   \label{def:quasicentraldef}
 Let $\zeta \in \cE(\bG^F, \ell')$. We say that $\zeta $ is \emph{of
 central $\ell$-defect} if $|\bG^F|_\ell=\zeta(1)_\ell |Z(\bG)^F|_\ell$ and
 that $\zeta$ is \emph{of quasi-central $\ell$-defect} if some (and hence any)
 character of $[\bG, \bG]^F$ covered by $\zeta$ is of central $\ell$-defect.
\end{defn}

The above definition makes sense since if $\zeta \in\cE(\bG^F,\ell')$, then
any character of $[\bG,\bG]^F$ covered by $\zeta$ is in
$\cE([\bG, \bG]^F,\ell')$.
The following are some properties of characters of quasi-central and central
$\ell$-defect. They rely on Lusztig's result \cite[Prop.~10]{Lu88} on the
restriction of irreducible characters under regular embeddings being
multiplicity free.

\begin{prop}   \label{prop:quasicentraldefpro}
 Let $\zeta \in \cE(\bG^F, \ell')$,  $A= Z(\bG)_\ell^F$, and
 $A_0 = Z([\bG, \bG])_\ell^F$. Then:
 \begin{enumerate}[\rm(a)]
  \item $\zeta $ is of quasi-central $\ell$-defect if and only if
   $|[\bG, \bG]^F|_\ell=\zeta(1)_\ell |Z ([\bG, \bG])^F|_\ell$.
  \item If $\zeta$ is of central $\ell$-defect, then $\zeta$ is of
   quasi-central $\ell$-defect.
  \item $\zeta$ is of central $\ell$-defect if and only if $A$ is a
   defect group of $b_{\bG^F}(\zeta)$.
 \end{enumerate}
 Suppose that $\zeta$ is of quasi-central $\ell $-defect. Then:
 \begin{enumerate}[\rm(a)]
  \item[\rm(d)] $b_{\bG^F}(\zeta)$ is nilpotent.
  \item [\rm(e)]Any defect group $D$ of $b_{\bG^F}(\zeta)$ contains $A$
   with $D/A$ isomorphic to a Sylow $\ell$-subgroup of $\bG^F/ A[\bG, \bG]^F$
   and $D \cap [\bG, \bG]^F =A_0$.
  \item [\rm(f)] $\cE(\bG^F, \ell') \cap \Irr(b_{\bG^F}(\zeta))= \{\zeta\}$.
  \item [\rm(g)] $\zeta$ is of central $\ell$-defect if and only if
   $\bG^F/ A[\bG, \bG]^F$ is an $\ell'$-group.
 \end{enumerate}
\end{prop}

\begin{proof}
Let $\zeta_0$ be an irreducible constituent of the restriction of $\zeta$
to $[\bG,\bG]^F$ and let $I$ be the stabiliser in $\bG^F$ of $\zeta_0$.
Since $\zeta_0 \in \cE([\bG,\bG]^F,\ell')$, the index of $I$ in $\bG^F$
is prime to $\ell$. On the other hand, by
\cite[Prop.~10]{Lu88}, $\zeta_0$ extends to an irreducible character of $I$.
Thus, $\zeta_0(1)_\ell = \zeta(1)_\ell$, proving~(a) and~(b).
Since $A$ is a central $\ell$-subgroup of $\bG$ and $\zeta\in\cE(\bG^F,\ell')$,
$A$ is in the kernel of  $\zeta$, from which~(c) is immediate.

Assume till the end of the proof that $\zeta$ is of quasi-central
$\ell$-defect.
By~(c), $b_{[\bG,\bG]^F}(\zeta_0)$ has defect group $A_0$.
Further, since $\bG= Z(\bG)[\bG, \bG]$ and
$A_0\leq Z([\bG, \bG])^F$, $A_0$ is central in $\bG^F $.
So the hypotheses of Lemma~\ref{lem:cliff-ab}(b) are satisfied for the normal
subgroup $[\bG, \bG]^F$ of $\bG^F$ and the blocks
$b_{[\bG, \bG]^F}(\zeta_0)$ and $b_{\bG^F}(\zeta)$, proving~(d) and~(e). 
The index of $I$ in $\bG^F$ is prime to $\ell$, so again by
Lemma~\ref{lem:cliff-ab}(b) there are $|I : [\bG, \bG]^F|_{\ell'}$
$\ell $-blocks of $\bG^F$ covering $b_{[\bG, \bG]^F}(\zeta_0)$. Also, there
are $|I : [\bG,\bG]^F|_{\ell'}$ elements of $\cE(\bG^F, \ell')$ covering
$\zeta_0$.
Now (f) follows from Theorem~\ref{thm:BrMiHi} and~(g) follows from~(c) and~(e).
\end{proof}

The next results will be our main tools for the identification of defect groups.
We first derive an easy upper bound for the orders of defect groups:

\begin{lem}   \label{lem:defectbound}
 Let $s\in \bG^{*F}$ be a semisimple $\ell'$-element.
 \begin{enumerate}[\rm (a)]
  \item The defect groups of any $\ell$-block in $\cE_\ell(\bG^F,s)$ have order
   at most $|C_{\bG^{*F}}(s)|_\ell$.
  \item There exists an $\ell$-block in $\cE_\ell(\bG^F,s)$ whose defect groups
   have order $|C_{\bG^{*F}}(s)|_\ell$.
 \end{enumerate}
 In particular, if $\cE_\ell(\bG^F,s)$ is a single $\ell$-block, then the defect
 groups of this block have order  $|C_{\bG^{*F}}(s)|_\ell$.
\end{lem}

\begin{proof}
Let $t$ be an $\ell$-element of $C_{\bG^{*F}}(s)$ and let
$\chi\in\cE(\bG^F,st)$. By the Jordan decomposition of characters, there
exists an irreducible (unipotent) character $\psi$ of $C_{\bG^{*F}}(st)$
such that
$$\frac{|\bG^F|_\ell}{\chi(1)_\ell}
  =\frac{|C_{\bG^{*F}}(st)|_\ell} {\psi(1)_\ell}. $$
In particular,
$$\frac{|\bG^F|_\ell}{\chi(1)_\ell} \leq |C_{\bG^{*F}}(st)|_\ell \leq
  |C_{\bG^{*F}}(s)|_\ell. $$
This proves the first part. If $\chi \in
\cE(\bG^F,s)$ corresponds to the trivial character of $C_{\bG^{*F}}(s)$,
then by the above formula, the $\ell$-defect of $\chi$ is
$|C_{\bG^{*F}}(s)|_\ell $, hence the block containing $\chi$ has defect at
least $|C_{\bG^{*F}}(s)|_\ell $. This proves ~(b).
\end{proof}

\begin{prop}   \label{prop:defect:re}
 Let $\bL\le\bG$ be an $F$-stable Levi subgroup and $A=Z(\bL)_\ell^F$,
 $A_0= Z([\bL,\bL])_\ell^F$. Suppose that
 $$ \bL=C_\bG^\circ(A), \qquad \bL^F=C_{\bG^F}(A).$$
 For $\lambda\in\cE(\bL^F,\ell')$ of quasi-central $\ell$-defect let
 $u=b_{\bL^F}(\la)$ and let $b$ be the block of $\bG^F$ such that $(A,u)$ is
 a $b$-Brauer pair. Then
 \begin{enumerate}[\rm(a)]
  \item $N_{\bG^F}(A)= N_{\bG^F}(\bL)$, $N_{\bG^F}(A, u)= N_{\bG^F}(\bL, \la)$,
   and $N_{\bG^F}(A,u)/C_{\bG^F}(A) =W_{\bG^F}(\bL, \la)$, where
   $W_{\bG^F}(\bL, \la):=N_{\bG^F}(\bL,\la)/\bL^F$.
 \end{enumerate}
 Let $(A,u)\subseteq (D,v) \subseteq (P,w)$ be $b$-Brauer pairs such that
 $D$ is maximal with respect to $D \leq C_{\bG^F}(A)$ and $P$ is maximal with
 respect to $P \leq N_{\bG^F}(\bL, \la)$. Then:
 \begin{enumerate}[\rm(a)]
  \item[\rm(b)]  $D/A$ is isomorphic to a Sylow $\ell$-subgroup of
   $\bL^F/A[\bL,\bL]^F$, $P \cap \bL^F =D$ and $P/D$ is isomorphic to a Sylow
   $\ell$-subgroup of $W_{\bG^F}(\bL,\la)$.
 \end{enumerate}
 Let $s\in {\bG^*}^F$ be an $\ell'$-element such that
 $\Irr(b)\subseteq \cE_\ell(\bG^F,s)$. Then:
 \begin{enumerate}[\rm(a)]
  \item[\rm(c)] If
   $$|C_{\bG^{*F}}(s)|_\ell = |Z^\circ(\bL)_\ell^F|\cdot|A_0|\cdot
     |W_{\bG^F}(\bL,\la)|_\ell,$$
   then  $P$  is a defect group of $b$.
  \item[\rm(d)] If
   $A$ is characteristic in $P$, then  $P$ is a defect group of $b$.
  \item[\rm(e)]  If the defect groups of $b$ are abelian,
   then  $\ell$ does not divide $|W_{\bG^F}(\bL, \la)| $.
  \item[\rm(f)]  If $A=D$ and $\ell$ does not divide $|W_{\bG^F}(\bL, \la)|$,
   then $A$ is a defect group of $b$.
  \item[\rm(g)] If $Z^\circ (\bL)^F \cap [\bL, \bL]^F$ is an $\ell'$-group,
   then $A=D$.
 \end{enumerate}
\end{prop}

\begin{proof}
Since $\bL= C_\bG^\circ (A)$, $N_{\bG^F}(A) \leq N_{\bG^F}(\bL)$
and since $A= Z(\bL)_\ell^F$, $N_{\bG^F}(\bL) \leq N_{\bG^F}(A)$. Thus,
$N_{\bG^F}(A) = N_{\bG^F}(\bL)$.
By Proposition~\ref{prop:quasicentraldefpro}(f), $\la$ is the unique element
of $\cE(\bL^F,\ell')\cap\Irr(u)$.
Since conjugation by elements of $N_{\bG^F}(\bL)$ stabilises
$\cE(\bL^F,\ell')$,  we get that
$$N_{\bG^F}(A, u) =N_{\bG^F}(\bL, u)= N_{\bG^F}(\bL, \la).$$
This proves~(a). From this, (b) follows by Lemma~\ref{lem:local-cliff} and
Proposition~\ref{prop:quasicentraldefpro}(e).

By~(b) we have
$$|P|= \frac{ |A|\cdot|\bL^F|_\ell \cdot|W_{\bG^F}(\bL, \la)|_\ell}
            { |A [\bL, \bL]^F|_\ell }.$$
Now, $|\bL^F| = |Z^\circ (\bL)^F|\cdot|[\bL,\bL]^F|$ and as
pointed out in the proof of Proposition~\ref{prop:quasicentraldefpro},
$A_0 =A\cap [\bL, \bL]^F$. Hence
$$|P| = |Z^\circ(\bL)_\ell^F|\cdot|A_0|\cdot|W_{\bG^F}(\bL,\la)|_\ell$$ and~(c)
follows from Lemma~\ref{lem:defectbound}(a).

Suppose that $A$ is characteristic in $P$. Let
$(P,w)\subseteq(R,f)\subseteq(S,j)$ be $b$-Brauer pairs with $(S,j)$
maximal and $R=N_S(P)$. Since $R$ normalises $A$,
$P\leq R\leq N_{\bG^F}(A,u)=N_{\bG^F}(\bL,\la)$.
So, by maximality of $P$, $R=P$, whence $S=P$, proving~(d).

Part~(e) is immediate from part~(b).

If $A=D$ and  $W_{\bG^F}(\bL, \la)$ is an $\ell'$-group, then $P=A$,
which means that if $(S, j)$ is any maximal Brauer pair containing $(A, u)$,
then $N_S(A) = A$. But this implies that $(A, u)$ is maximal, proving~(f).

If $Z^\circ(\bL)^F \cap [\bL,\bL]^F$ is an $\ell'$-group, then from
the equality $|\bL^F|= |Z^\circ (\bL)^F|\cdot|[\bL, \bL]^F|$ it follows that
$\bL^F/ Z^\circ (\bL)^F[\bL, \bL]^F$  and hence $\bL^F/ A[\bL,\bL]^F$
is an $\ell'$-group. But by part~(a), $D/A$ is isomorphic to a Sylow
$\ell$-subgroup of $\bL^F/A[\bL,\bL]^F$. This proves~(g).
\end{proof}

\subsection{Lusztig induction and $e$-Harish-Chandra theory}
It is known that the $\ell$-blocks of $\bG^F$ are in close
relation with Lusztig induction. For any $F$-stable Levi subgroup $\bL$ of
a (not necessarily $F$-stable) parabolic subgroup $\bP$ of $\bG$ Lusztig
defines linear maps
$$\RLPG:\ZZ\Irr(\bL^F)\longrightarrow\ZZ\Irr(\bG^F),$$
$$^*\RLPG:\ZZ\Irr(\bG^F)\longrightarrow\ZZ\Irr(\bL^F),$$
adjoint to each other with respect to the standard scalar
product on complex characters. This Lusztig induction enjoys the following
important properties:

\begin{thm}   \label{thm:RLG}
 Let $\bL$ be an $F$-stable Levi subgroup of a parabolic subgroup $\bP$ of
 $\bG$.
\begin{itemize}
 \item[(a)] If $\bM\le\bL$ is an $F$-stable Levi subgroup of a parabolic
   subgroup $\bQ$ of $\bP$ then
  $$R_{\bM\subset\bQ}^\bG =\RLPG\circ R_{\bM\subset\bQ\cap\bL}^\bL.$$
 \item[(b)] Let $\bL^*$ be an $F$-stable Levi subgroup of $\bG^*$ in duality
  with $\bL$. For any semisimple element $s\in\bL^*$, $\RLPG$
  restricts to a linear map
  $$\RLPG:\ZZ\cE(\bL^F,s)\longrightarrow\ZZ\cE(\bG^F,s).$$
 \item[(c)] If $\bP$ is $F$-stable, then
  $$\RLPG=\Ind_{\bP^F}^{\bG^F}\circ\Infl_{\bL^F}^{\bP^F}.$$
 \item[(d)] The Mackey formula holds for $\RLPG$ except possibly if $\bG^F$
  has a simple component $\tw2E_6(2)$, $E_7(2)$ or $E_8(2)$.
\end{itemize}
\end{thm}

\begin{proof}
See \cite[11.5]{DM91} for (a), part (b) is immediate from this and the
definition of the Lusztig series, for~(c) see \cite[\S 11]{DM91}, and the
recent paper of Bonnaf\'e--Michel \cite{BMi} for (d).
\end{proof}

Note that, as a formal consequence of the validity of the Mackey formula,
Lusztig induction is independent of the choice of parabolic subgroup $\bP$
containing $\bL$ (except possibly in the groups excluded in
Theorem~\ref{thm:RLG}(d)). We will henceforth just write $\RLG$.
\vskip 1pc

An $F$-stable torus $\bT$ of $\bG$ is called an \emph{$e$-torus} if it splits
completely over $\FF_{q^e}$ but does not split over any smaller field.
Equivalently, there is $a\ge0$ such that $|\bT^{F^k}|=\Phi_e(q^k)^a$ for
all $k\ge1$, where $\Phi_e$ denotes the $e$th cyclotomic polynomial. The
centralisers of $e$-tori of $\bG$ are called \emph{$e$-split Levi subgroups}.
(Note that these are indeed Levi subgroups, which are $F$-stable.)
A character $\chi\in\Irr(\bG^F)$ is called \emph{$e$-cuspidal} if
$^*\RLG(\chi)=0$ for every $e$-split proper Levi subgroup $\bL$ of $\bG$.
A pair $(\bL,\lambda)$ consisting of an $e$-split Levi subgroup $\bL$ and an
$e$-cuspidal character $\lambda\in\Irr(\bL^F)$ is then called an
\emph{$e$-cuspidal pair}. Given an $e$-cuspidal pair $(\bL,\lambda)$, we write
$$\cE(\bG^F,(\bL,\lambda)):=\{\chi\in\Irr(\bG^F)\mid
  \langle ^*\RLG(\chi),\lambda\rangle\ne0\}$$
for the set of constituents of $\RLG(\lambda)$. This is called the
\emph{$e$-Harish-Chandra series of $\bG^F$ above $(\bL,\la)$}.

\begin{defn}   \label{def:d-HC}
 We say that \emph{$\RLG$ satisfies an $e$-Harish-Chandra theory above the
 $e$-cuspidal pair $(\bL,\la)$} if there exists a collection of isometries
 $$I^\bM_{(\bL,\la)}:\ZZ\Irr(W_{\bM^F}(\bL,\la))
                     \rightarrow \ZZ\cE(\bM^F,(\bL,\la))\,,$$
 where $\bM$ runs over the set of all $e$-split Levi subgroups of $\bG$
 containing $\bL$, such that
 \begin{enumerate}[\rm(1)]
  \item for all $\bM$ we have
   $$\RMG \circ I^\bM_{(\bL,\la)} =
     I^\bG_{(\bL,\la)}\circ\Ind_{W_{\bM^F}(\bL,\la)}^{W_{\bG^F}(\bL,\la)}\,;$$
  \item the collection $\left(I^\bM_{(\bL,\la)} \right)_{\bM,(\bL,\la)}$ is
   stable under the conjugation action by $W_{\bG^F}$; and
  \item $I^\bL_{(\bL,\la)}$ maps the trivial character of the trivial
   group $W_{\bL^F}(\bL,\la)$ to $\la$.
 \end{enumerate}
\end{defn}

The following is shown in \cite[Prop.~3.15 and Thm.~3.11]{BMM}:

\begin{prop}   \label{prop:HC-coroll}
 Assume that $\RLG$ satisfies an $e$-Harish-Chandra theory above $(\bL,\la)$.
 Then for any $e$-split Levi subgroup $\bL\le\bH\le\bG$ and any
 $\chi\in\Irr(\bH^F)$ with $\langle \RLH(\la),\chi\rangle\ne 0$ we have:
 \begin{enumerate}[\rm(a)]
  \item  $$^*\RLH(\chi)=\langle\,{}^*\RLH(\chi),
   \la\rangle_{\bL^F}\sum_{g\in N_{\bH^F}(\bL)/N_{\bH^F}(\bL,\la)}{}^g\la.$$
  \item  Every constituent $\psi$ of $\RHG(\chi)$ is a constituent
   of $\RLG(\la)$.

 \end{enumerate}
\end{prop}

\subsection{$\ell$-adapted Levi subgroups and Cabanes' criterion}
The results in this section are adaptations and extensions of a powerful
criterion of Cabanes, formulated in
\cite[Prop.~3]{En00}, which provides a strong relation between the explicit
decomposition of the Lusztig induction functor $\RLG$ for suitable Levi
subgroups $\bL$ of $\bG$ and the subdivision of $\cE_\ell(\bG^F,s)$ into
$\ell$-blocks through the inclusion of Brauer pairs.

For an $\ell$-element $z$ of $\bG^F$, we write $d^{z,\bG^F}$ for the
generalised decomposition map which sends a $K$-valued class function $f$ of
$\bG^F$ to the class function $d^{z,\bG^F}(f)$ on $C_{\bG^F}(z)$  by the rule 
$d^{z,\bG^F}(f)(zy)= f(zy)$ if $y\in C_{\bG^F} (z)$ has order prime to $\ell$
and $d^{z,\bG^F}(f)(zy)= 0$ if the order of $y\in C_{\bG^F}(z)$ is
divisible by $\ell$. The map $d^{1,\bG^F}$ is the usual decomposition map. 
Note that if $A$ is an abelian $\ell$-subgroup of $\bG^F$ contained in a
maximal torus of $\bG$, then $C_\bG(A)/C_\bG^\circ(A)$ is an $\ell$-group
(see \cite[Prop.~14.20]{MT}, \cite[Prop.~2.1(i)]{CE94}).
So, $C_{\bG^F}(A)/C_\bG^\circ (A)^F$ is an $\ell$-group, and hence,
each $\ell$-block of $C_\bG^\circ (A)^F$ is covered by a unique block of
$C_{\bG^F}(A)$. We will use this fact without further comment.

\begin{lem}  \label{lem:Cabanescrux}
 Let $\bL$ be an $F$-stable Levi subgroup of $\bG$, let
 $\la\in\cE(\bL^F,\ell')$ and $\chi\in\cE(\bG^F,\ell')$. Suppose that
 $\langle\chi, \RLG(\lambda) \rangle \ne 0$ and
 $\langle {}^*\RLG(\chi), d^{1,\bL^F}(\lambda)\rangle \ne 0$.
 Then, for any $z\in Z(\bL)_\ell^F$, there exists an irreducible constituent
 $\phi$ of $\RLH(\lambda)$, where $\bH:=C_\bG^\circ(z)$, such that
 denoting by  $\tilde b$ the unique block of $C_{\bG^F}(z)$
 covering the block containing $\phi$,
 $(\langle z\rangle, \tilde b)$ is a $ b_{\bG^F}(\chi)$-Brauer pair.
\end{lem}

\begin{proof}
We have
$$d^{1, \bL^F}({}^*\RLG(\chi))= d^{z, \bL^F} ({}^*\RLG(\chi)) =
  {}^*\RLH (d^{z, \bG^F}(\chi)), $$
the first equality holding since $z$ is a central $\ell$-element of
$\bL^F$ whereas $\chi$ is in an $\ell'$-series, and the second because of
the commutation of Lusztig restriction and generalised decomposition maps
(see \cite[Thm.~21.4]{CE}). It follows that
$$\langle d^{z, \bG^F}(\chi), \RLH (\lambda) \rangle \ne 0. $$
Now the result follows by Brauer's second main theorem and the fact that the
index of $\bH=C_\bG^\circ (z)$ in $C_\bG(z)$ is a power of $\ell$.
\end{proof}

\begin{prop}   \label{prop:Cabanescrux}
 Let $\bL$ be an $F$-stable Levi subgroup of $\bG$ and let
 $\lambda\in \cE(\bL^F, \ell')$. Let $Z$ be a subgroup of $Z(\bL)_\ell^F$ and
 $\{z_1,\ldots,z_m\}$ a generating set for $Z$. Set
 $\bH_i= C_\bG^\circ(z_1,\ldots,z_i)$, $ 1\leq i\leq m$, and $ \bH_0= \bG$.
 Suppose the following:
 \begin{enumerate}[\rm(1)]
  \item For any $i$, $0\leq i\leq m-1$, and any character
   $\chi\in\Irr(\bH_{i}^F)$ with
   $\langle R_\bL^{\bH_{i}}(\la),\chi\rangle_{\bH_{i}^F}\ne0$ we have
   $\langle d^{1,\bL^F}(\la),{}^*R_\bL^{\bH_{i}}(\chi)\rangle_{\bL^F}\ne0$.
  \item The irreducible constituents of $R_\bL^{\bH_{m}}(\la)$ lie in a
   single $\ell$-block of $\bH_m^F$.
 \end{enumerate}
 Then, for all $i$, $0 \leq i \leq m$, there exists a unique block, say
 $b_i$ of $\bH_i^F$ such that all constituents of $R_\bL^{\bH_{i}}(\la)$
 lie in $b_i$. Further, letting $\tilde b_i $ be the unique block of
 $C_\bG (z_1,\ldots, z_i)^F$ covering $b_i$ for $ 1\leq i \leq m$, we have
 inclusions of Brauer pairs
 $$\left(\{1\},b_0 \right)\subseteq
    (\langle z_1\rangle,\tilde b_1)\subseteq\ldots\subseteq(Z,\tilde b_m).$$
\end{prop}

\begin{proof}
Proceed by induction on $m$. Suppose first that $m=1$, and let $b_1$ be the
block of $\bH_1^F$ in which all constituents of $R_\bL^{\bH_1}(\la)$ lie.
By (1) and Lemma \ref{lem:Cabanescrux}, for any irreducible constituent $\chi$
of $\RLG(\lambda)$, we have an inclusion of Brauer pairs
$$(\{1\}, b_{\bG^F}(\chi)) \subseteq (\langle z_1\rangle,\tilde b_1)
  =(Z,\tilde b_1).$$
Now by the uniqueness of inclusion of Brauer pairs it follows that
$b_{\bG^F}(\chi)= b_{\bG^F}(\chi')$ for any irreducible constituents
$\chi, \chi'$ of $\RLG(\lambda)$.
\par
Now suppose $m>1$. Since $\bH_m = C_{\bH_{m-1}}^\circ (z_m)$,
by the previous argument there exists a unique block $b_{m-1}$ of $\bH_{m-1}^F$
such that all constituents of $R_\bL^{\bH_{m-1}}(\la)$ lie in
$b_{m-1}$ and there is an inclusion of $\bH_{m-1}^F$-Brauer pairs
$$ (\{1\}, b_{m-1}) \subseteq (\langle z_m \rangle, b_m'),$$
where $b_m'$ is the unique block of $C_{\bH_{m-1}^F}(z_m)$ covering the block
$b_m$ of $\bH_m^F$ (note that $C_{\bH_{m-1}}(z_m)$
may be a proper subgroup of $C_\bG (z_1,\ldots,z_m)$).
This yields an inclusion of $C_\bG(z_1,\ldots,z_{m-1})^F$-Brauer pairs
$$ (\{1\},\tilde b_{m-1}) \subseteq (\langle z_m \rangle,\tilde b_m),$$
and hence we have the inclusion of $\bG^F$-Brauer pairs
$$(\langle z_1,\ldots,z_{m-1}\rangle,\tilde b_{m-1})\subseteq
  (Z,\tilde b_m).$$
The result now follows by induction since we have shown above that
all constituents of $R_\bL^{\bH_{{m-1}}}(\la)$ lie in the same block.
\end{proof}

The following gives sufficient criteria for condition~(2) of
Proposition~\ref{prop:Cabanescrux} to hold.

\begin{prop}  \label{prop:Cabanessingle}
 In the notation of Proposition~\ref{prop:Cabanescrux} condition~(2) is
 satisfied for any $\la\in\cE(\bL^F,\ell')$ if one of the following holds:
 \begin{enumerate}[\rm(1)]
  \item  $\bL=C_\bG^\circ(Z)$; or
  \item  $ \ell = 2 $ and the simple components of $C_\bG^\circ(Z)$ are of
   classical type $B$, $C$ or $D$.
 \end{enumerate}
\end{prop}

\begin{proof}
The assertion is obvious in the first case since then $\bH_m=\bL$.
In the second case, the assertion follows  by \cite[Prop. 1.5(b)]{En08}.
\end{proof}

We will make mostly use of condition~(1) above which has been checked in
many cases by Enguehard \cite{En00} for the choice $Z=Z(\bL)_\ell^F$.

Now we develop  sufficient criteria for condition~(1) of
Proposition~\ref{prop:Cabanescrux} to hold.

\begin{defn}
 Let $\bL\le\bG$ be an $e$-split Levi subgroup. We say that $\bL$ is
 \emph{$(e,\ell)$-adapted}, if there exist generators
 $Z(\bL)_\ell^F=\langle z_1,\ldots,z_m\rangle$ such that
 $C_\bG^\circ(z_1,\ldots,z_i)$ is an $e$-split Levi subgroup of $\bG$ for all
 $1\le i\le m$.
\end{defn}


\begin{prop}  \label{prop:HC-theorycrux}
 Let $e\ge1$ and let $(\bL,\la)$ be an $e$-cuspidal pair such that
 $\la\in\cE(\bL^F,\ell')$. Assume that $\RLG$ satisfies an $e$-Harish-Chandra
 theory above $\la$. Then for any $e$-split Levi subgroup $\bL\le\bH\le\bG$
 and any
 $\chi\in\Irr(\bH^F)$ such that $\langle\RLH(\la),\chi\rangle \ne 0$, we have
 $$\blangle d^{1,\bL^F} ({}^*\RLH(\chi)),{}^*\RLH(\chi)\brangle_{\bL^F}\ne 0.$$

 Further, if $\bL$ is $(e,\ell)$-adapted in $\bG$ with respect to the
 generating set $\{z_1,\ldots,z_m\}$ of $Z(\bL)_\ell^F$,
 then condition~(1) of Proposition \ref{prop:Cabanescrux} is satisfied with
 respect to $z_1,\ldots,z_m$.
\end{prop}

\begin{proof}
Let $\bL\le\bH\le\bG$ be $e$-split and $\chi\in\Irr(\bH^F)$ such that
$\langle\RLH(\la),\chi\rangle \ne 0 $. By
Proposition~\ref{prop:HC-coroll} we have
$$^*\RLH(\chi) = a \sum_{g\in N_{\bH^F}(\bL)/N_{\bH^F}(\bL,\la)}{}^g\la$$
with $a:=\langle\la,{}^*\RLH(\chi)\rangle=\langle\RLH(\la),\chi\rangle\ne 0$,
whence we see that $^*\RLH(\chi)(1)\ne0$, and thus
$$\blangle d^{1,\bL^F}({}^*\RLH(\chi)),{}^*\RLH(\chi)\brangle_{\bL^F}
  \ne 0.$$
But,
\begin{eqnarray*} \blangle d^{1,\bL^F} ({}^*\RLH(\chi)),
  {}^*\RLH(\chi) \brangle_{\bL^F}  &=&
  a \sum_{g\in N_{\bH^F}(\bL)/N_{\bH^F}(\bL,\la)}  \blangle d^{1,\bL^F}({}^g\la),
  {}^*\RLH(\chi) \brangle_{\bL^F}\\ &=&
  a|N_{\bH^F}(\bL): N_{\bH^F}(\bL,\la)|\blangle d^{1,\bL^F}(\la),
  {}^*\RLH(\chi) \brangle_{\bL^F}. \end{eqnarray*}
This proves the  first assertion. The second part follows by repeatedly
applying the first assertion to the cases $\bH=\bG$, respectively
$\bH = C_\bG^\circ(z_1,\ldots, z_i)$, $1\leq i\leq m $.
\end{proof}

The next result contains further useful criteria for condition~(1) of
Proposition \ref{prop:Cabanescrux}.

\begin{prop}  \label{prop:decomp}
 Let $\bL$ be an $F$-stable Levi subgroup of $\bG$, let
 $\lambda\in\cE(\bL^F,\ell')$ and let $b$
 be the $\ell$-block of $\bL^F$ containing $\la$. Suppose that one of the
 following holds:
 \begin{enumerate}[\rm(1)]
  \item $\bL$ is a torus;
  \item $\ell$ is good for $\bL$ and the $\ell$-block of $\bL^F$ containing
   $\la$ is nilpotent;
  \item $\la$ is of quasi-central $\ell$-defect; or
  \item $\Irr(b)\cap\cE(\bL^F,\ell')=\{\la\}$.
 \end{enumerate}
 Then for any character $\chi\in\Irr(\bG^F)$ with $\langle \RLG
 (\la),\chi\rangle_{\bG^F}\ne0$ we have
 $$\langle d^{1,\bL^F}(\la),{}^*\RLG(\chi)\rangle_{\bL^F}\ne0.$$
 Consequently, condition~(1) of Proposition~\ref{prop:Cabanescrux} holds for
 any subgroup $Z$ of $Z(\bL)_\ell^F$ and any generating set
 $\{z_1,\ldots,z_m\}$ of $Z$.
\end{prop}

\begin{proof}
(1) is a special case of~(2) and of~(3), and by
Proposition~\ref{prop:quasicentraldefpro}(f), (3) is a special case
of~(4). Also, the second assertion follows by applying the first part with
$\bG$ replaced by $\bH_i$, $1\leq i\leq m$. For any irreducible character
$\chi$ of $\bG^F$, we have that
$$\blangle d^{1,\bL^F}(\la),{}^*\RLG(\chi)\brangle_{\bL^F}
  = \blangle d^{1,\bL^F}(\la),b\cdot {}^*\RLG(\chi)\brangle_{\bL^F}
  = \blangle d^{1,\bL^F}(\la),d^{1,\bL^F}(b.{}^*\RLG(\chi))\brangle_{\bL^F}.$$
Hence, in order to prove the proposition it suffices to show that if
either~(2) or~(4) of the statement hold, then
$$\blangle d^{1,\bL^F}(\la),d^{1,\bL^F}(b.{}^*\RLG(\chi))\brangle_{\bL^F}\ne 0$$
for any $\chi \in\Irr(\bG^F)$ such that
$\langle \RLG (\la),\chi\rangle_{\bG^F}\ne0$. Indeed, for such $\chi$
we have by adjunction
$${}^*\RLG(\chi) = a \lambda + \sum _{\phi \in I} a_\phi \phi$$
for suitable $a \ne 0$, $a_\phi\in\ZZ$, where $I$ is a subset of
$\cE(\bL^F,\ell') \setminus\{\lambda\}$. So,
$$b.{}^*\RLG (\chi) = a \lambda + \sum _{\phi \in I'} a_\phi \phi$$
where $I'=I\cap\Irr(b)$.

Suppose first that~(2) holds. Since $\ell $ is good for $\bL$,
by \cite[Thm. 1.7]{CE99} the restriction of the right hand side of the above
equation to the $\ell'$-elements of $\bL^F$ is non-zero. On the other hand,
since $b$ is nilpotent, $\{d^{1, \bL^F} (\la)\} $ is an $\ell$-basic set
for $b$. Hence $d^{1, \bL^F}(b.{}^*\RLG (\chi)) = m d^{1,\bL^F}(\la)$
for some non-zero $m$. The result follows since $\la(1)\ne 0$.

Now suppose~(4) holds. Then again since
$I'\subseteq\Irr(b)\cap\cE(\bL^F,\ell') \setminus \{\lambda\}$,
the hypothesis implies that $I' = \emptyset$.
The result follows since $\la(1) \ne 0$.
\end{proof}

The previous results combine to give the following criterion which will be
crucial for the proof of Theorem~\ref{thm:mainblocks}. Here, $(\bL,\la)$ is
said to \emph{lie below $\cE(\bG^F,s)$} if the constituents of $\RLG(\la)$
lie in $\cE(\bG^F,s)$, or equivalently, if $\la\in\cE(\bL^F,s)$.

\begin{prop}   \label{prop:excellent}
 Let $e \geq 1$ and let $s\in{\bG^*}^F$ be a semisimple $\ell'$-element.
 Suppose the following.
 \begin{enumerate}[\rm (1)]
  \item The assertions of Theorem~\ref{thm:BMM} hold for the set of
   $e$-cuspidal pairs below $\cE(\bG^F,s)$.
  \item For any $e$-cuspidal pair $(\bL,\la)$ below $\cE(\bG^F,s)$ we have
   $\bL=C_\bG^\circ(Z(\bL)_\ell^F)$, and $\bL^F=C_{\bG^F}(Z(\bL)_\ell^F)$,
   and $\bL$ is $(e,\ell)$-adapted.
 \end{enumerate}
 Then for any  $e$-split Levi subgroup $\bH$ of $\bG$ such that
 $\bH^F=C_{\bG^F}(Z(\bH)_\ell^F)$ the following holds:

 For any $\ell$-block $b$ of $\bH^F$ such that
 $\Irr(b)\cap \cE(\bH^F,s) \ne \emptyset$, there exists a unique $\ell$-block
 $c$ of $\bG^F$ such that for any $\chi\in\Irr (b)\cap \cE(\bH^F, s)$ all
 constituents of $\RHG(\chi)$ lie in $c$.
 Moreover, $(Z(\bH)_\ell^F,b)$ is a $c$-Brauer pair.
\end{prop}

\begin{proof}
Let $\bH$ be as in the statement and let $(\bL,\la)$ be an $e$-cuspidal pair
of $\bG$ such that $\bL \leq\bH$. We claim that
$\bL=C_\bH^\circ(Z(\bL)_\ell^F)$, $\bL^F=C_{\bH^F}(Z(\bL)_\ell^F)$,
$\bL$ is $ (e, \ell)$-adapted in $\bH$, and $\RLH$ satisfies an
$e$-Harish-Chandra theory above $\la$.
The first two assertions of our claim follow from
$$\bL \leq C_\bH^\circ(Z(\bL)_\ell^F) \leq
  C_\bG^\circ(Z(\bL)_\ell^F) \cap \bH =\bL\cap \bH= \bL.$$
Next, we show that $\bL$ is $ (e,\ell)$-adapted in $\bH$. Let
$Z(\bL)_\ell^F=\langle z_1,\ldots,z_m\rangle$ be a system of generators such
that $\bL_i:=C_\bG^\circ(z_1,\ldots,z_i)$ is $e$-split, and
for $1\le i\le m$ let $\bT_i$ be the Sylow $e$-torus of $Z(\bL_i)$ and
$\bT$ the Sylow $e$-torus of $Z(\bH)$, so that $\bL_i = C_\bG(\bT_i)$ and
$\bH = C_\bG(\bT)$. Since $\bT$ is central in $\bH$, and $\bL\le\bH$ is a Levi
subgroup, we have $\bT\le Z(\bL_i)$, so $\bT\le\bT_i$ and
$$C_\bH^\circ(z_1,\ldots,z_i) \leq C_\bG (\bT_i) \cap C_\bH (\bT) =
  C_\bH (\bT_i) =\bH\cap\bL_i= C_\bH^\circ(z_1,\ldots,z_i)$$
for $1\le i\le m$. Finally, since any $e$-split Levi subgroup of $\bH$ is an
$e$-split Levi subgroup of $\bG$, $\RLH$ satisfies an $e$-Harish-Chandra
theory over $\la$ by condition~(1), proving the claim.

Now let $b$ be as in the statement and let $\chi\in\Irr(b)\cap \cE(\bH^F,s)$.
Let $(\bL,\la)$ be an $e$-cuspidal pair of $\bG$ such that $\bL \leq \bH$, and
$\chi$ is a constituent of $\RLH(\la)$. By Proposition~\ref{prop:HC-coroll}(b)
every constituent $\psi$ of $\RHG(\chi)$ is a constituent of $\RLG(\la)$.
As $\RLG$ satisfies an $e$-Harish-Chandra theory above $\la$, condition~(1) of
Proposition~\ref{prop:Cabanescrux} holds for $\{z_1,\ldots,z_m\}$ by
Proposition~\ref{prop:HC-theorycrux}. Further, condition~(2) holds by
hypothesis and by Proposition~\ref{prop:Cabanessingle}.
Hence we have an inclusion of $\bG^F$-Brauer pairs
$$\left(\{1\},b_{\bG^F}(\psi)\right)\subseteq (Z(\bL)_\ell^F,b_{\bL_F}(\la)).$$

On the other hand, by using the claim one sees that the arguments in the
preceding section all apply to $\bH$ also, hence we have an inclusion of
$\bH^F$-Brauer pairs
$$\left(\{1\}, b \right) \subseteq (Z(\bL)_\ell^F , b_{\bL_F}(\la)). $$
Since by hypothesis $\bH^F=C_{\bG^F}(Z(\bH)_\ell^F)$,
this also yields an inclusion of $\bG^F$-Brauer pairs
$$\left(Z(\bH)_\ell^F,b\right) \subseteq (Z(\bL)_\ell^F,b_{\bL_F}(\la)).$$
Let $c$ be the unique block of $\bG^F$ such that we have an inclusion of
$\bG^F$-Brauer pairs
$$\left(\{1\}, c \right) \subseteq \left(Z(\bH)_\ell^F, b \right). $$
By transitivity and uniqueness of inclusion of Brauer pairs, we get that
$c=b_{\bG^F}(\psi)$. This proves the result.
\end{proof}

\section{The quasi-isolated blocks in $F_4(q)$}   \label{sec:F4}

In this section we prove Theorems~\ref{thm:mainblocks} and \ref{thm:BMM} on
quasi-isolated blocks of simple groups of type $F_4$. 

For this recall that a semisimple element $s$ of a connected reductive group
$\bG$ is called \emph{quasi-isolated} if its centraliser $C_\bG(s)$ is not
contained in any proper Levi subgroup of $\bG$. Correspondingly, a
\emph{quasi-isolated $\ell$-block} is a block lying in the Lusztig series
parametrised by a quasi-isolated $\ell'$-elements of the dual group.

By earlier results on blocks (see Remark~\ref{rem:history}) the decomposition
of $\cE_\ell(\bG^F,s)$ into $\ell$-blocks of $\bG^F$ is known except when
$\ell$ is a bad prime for $\bG$ and $s\ne1$ is a quasi-isolated
$\ell'$-element of $\bG^*$, an exceptional group of adjoint type. The various
$(\ell,s)$ will be treated case-by-case in Sections~\ref{sec:F4}--\ref{sec:E8},
so to start we need to recall the classification of quasi-isolated elements
in exceptional groups of adjoint type from \cite[Prop.~4.9 and Table~3]{B05}.

\begin{prop}[Bonnaf\'e]   \label{prop:Bonn}
 Let $\bG$ be a simple exceptional algebraic group of adjoint type and of
 rank at least~4. Then the conjugacy classes of quasi-isolated
 elements $s$ whose order is not divisible by all bad primes for $\bG$, the
 root system of their centraliser $C_\bG(s)$, the group of
 components $A(s):=C_\bG(s)/C_\bG^\circ(s)$ and the automiser
 $A(C):=N_\bG(C_\bG(s))/C_\bG^\circ(s)$ are given in Table~\ref{tab:quasi-el}.
\end{prop}

\begin{table}[htbp]
\caption{Quasi-isolated elements in exceptional groups}   \label{tab:quasi-el}
\[\begin{array}{|c|c|l|c|c|}
\hline
 \bG& o(s)& C_\bG^\circ(s)& A(s)& A(C)\\
\hline
 F_4& 2& C_3\pl A_1,\ B_4& 1& 1\\
    & 3& A_2\pl \tilde A_2& 1& 2\\
    & 4& A_3\pl \tilde A_1& 1& 2\\
 E_6& 2& A_5\pl A_1& 1& 1\\
    & 3& A_2\pl A_2\pl A_2,\ D_4& 3& \fS_3\\
 E_7& 2& D_6\pl A_1& 1& 1\\
    & 2& A_7,\ E_6& 2& 2\\
    & 3& A_5\pl A_2& 1& 2\\
    & 4& A_3\pl A_3\pl A_1,\ D_4\pl A_1\pl A_1& 2& 4\\
 E_8& 2& D_8,\ E_7\pl A_1& 1& 1\\
    & 4& D_5\pl A_3, A_7\pl A_1& 1& 2\\
    & 3& A_8,\ E_6\pl A_2& 1& 2\\
    & 5& A_4\pl A_4& 1& 4\\
\hline
\end{array}\]
\end{table}

In Table~\ref{tab:quasi-el}, in the last two columns, $n$ stands for a cyclic
group of order~$n$. Furthermore, $\tilde A_k$ denotes a component of
$C_\bG^\circ(s)$ of type $A_k$ generated by short root subgroups.

From now on let $\bG$ be simple of type $F_4$, with Steinberg endomorphism
$F:\bG\rightarrow\bG$, so $\bG^F=F_4(q)$, and let $\ell\in\{2,3\}$ be one
of the two bad primes for $\bG$.
According to Proposition~\ref{prop:Bonn} there exist four different types of
centralisers of quasi-isolated elements $1\ne s\in \bG^{*F}$. In
Table~\ref{tab:quasi-F4} we have collected various information on their
centralisers and the corresponding Lusztig series in $\Irr(\bG^F)$ as follows.
Firstly, in the second column we list the possible rational structures
of centralisers of quasi-isolated elements. Here, a quasi-isolated element of
order~4 with centraliser structure $A_3(q)\tilde A_1(q)$ exists when
$q\equiv1\pmod4$, while there is one with centraliser structure
$\tw2A_3(q)\tilde A_1(q)$ when $q\equiv3\pmod4$. Similarly, a quasi-isolated
3-element with centraliser structure $A_2(q)\tilde A_2(q)$ exists when
$q\equiv1\pmod3$, while there is one with centraliser structure
$\tw2A_2(q){}\tw2\tilde A_2(q)$ when $q\equiv2\pmod3$.
\par
In each case there is a unique bad prime $\ell$ not dividing $o(s)$. The third
column contains one of the two possibilities for
$$e=e_\ell(q):=\text{order of $q$ modulo}
  \begin{cases}\ell& \text{if $\ell>2$,}\\
                  4& \text{if $\ell=2$.}\end{cases}$$
More precisely, in order to avoid duplication of arguments, we assume that
$e=1$, that is, $q\equiv1\pmod4$ when $\ell=2$, and $q\equiv1\pmod3$ when
$\ell=3$, respectively. The cases where $e=2$, that is, where $q\equiv1\pmod4$
for $\ell=2$, respectively $q\equiv2\pmod3$ for $\ell=3$, can be obtained from
the former by formally exchanging $q$ by $-q$ in all arguments to come
(the operation of Ennola duality, see \cite[(3A)]{BMM}).
Note that $\bG^F$ itself is its own Ennola dual.

\begin{table}[htbp]
\caption{Quasi-isolated blocks in $F_4(q)$}   \label{tab:quasi-F4}
\[\begin{array}{|r|r|l|llll|}
\hline
 \text{No.}& C_{\bG^*}(s)^F& (\ell,e)& \bL^F& C_{\bL^*}(s)^F& \la& W_{\bG^F}(\bL,\la)\\
\hline\hline
 1&     A_2(q)\,\tilde A_2(q)& (2,1)& \Ph1^4& \bL^{*F}& 1& A_2\ti A_2\\
\hline
 2& \tw2A_2(q)\,\tw2\tilde A_2(q)& (2,1)& \Ph1^2.A_1(q)^2& \Ph1^2\Ph2^2& 1& A_1\ti A_1\\
  &             &      & \Ph1.B_3(q)&  \Ph1\Ph2.\tw2A_2(q)& \phi_{21}& A_1\\
  &             &      & \Ph1.C_3(q)&  \Ph1\Ph2.\tw2\tilde A_2(q)& \tilde\phi_{21}& A_1\\
  &             &      &       \bG^F&  C_{\bG^*}(s)^F& \phi_{21}\otimes\tilde\phi_{21}& 1\\
\hline\hline
 3&       B_4(q)& (3,1)& \Ph1^4& \bL^{*F}& 1& B_4\\
 4&             &      & \Ph1^2.B_2(q)& \bL^{*F}& B_2[1]& B_2\\
\hline
 5& C_3(q)\,A_1(q)& (3,1)& \Ph1^4& \bL^{*F}& 1& C_3\ti A_1\\
 6&             &      & \Ph1^2.B_2(q)& \bL^{*F}& B_2[1]& A_1\ti A_1\\
\hline
 7& A_3(q)\,\tilde A_1(q)& (3,1)& \Ph1^4& \bL^{*F}& 1& A_3\ti A_1\\
\hline
 8& \tw2A_3(q)\,\tilde A_1(q)& (3,1)& \Ph1^3.\tilde A_1(q)& \Ph1^3\Ph2& 1& C_2\ti A_1\\
\hline\hline
 2b&            & (2,2)& \Ph2^4& \bL^{*F}& 1& A_2\ti A_2\\
\hline
\end{array}\]
\end{table}

For each type of centraliser occurring in the table we have also listed in
Table~\ref{tab:quasi-F4} all $e$-cuspidal pairs $(\bL,\la)$ in $\bG$ (up to
$\bG^F$-conjugacy) such that $\la\in\cE(\bL^F,s)$, together with their relative
Weyl groups. More precisely, we denote $\la$ by the standard name of its
unipotent correspondent under Lusztig's Jordan decomposition of characters;
for example $\phi_{21}$ denotes the unipotent character of $\SL_3(q)$
parametrised by the partition $21$ of~3. Thus, in particular, if
$\la\in\cE(\bL,s)$ corresponds to $\rho\in\cE(C_{\bL^*}(s),1)$, then
$\la(1)=|\bL^*:C_{\bL^*}(s)|_{p'}\,\rho(1)$. 

The relative Weyl groups $W_{\bG^F}(\bL,\la)=N_{\bG^F}(\bL,\la)/\bL^F$ can
be computed using the {\sf GAP}-package {\sf Chevie} \cite{MChev}, see also the
paper of Howlett \cite{How}; they are Coxeter groups of the indicated type.

The last line 2b will be needed in one of the arguments below.

\begin{prop}   \label{prop:F4RLG}
 Let $s\ne1$ be a quasi-isolated $\ell'$-element of $\bG^{*F}=F_4(q)$, and
 assume that $e=e_\ell(q)=1$. Then we have:
 \begin{enumerate}[\rm(a)]
  \item $\cE(\bG^F,s)$ is the disjoint union of the $e$-Harish-Chandra
   series listed in the rows of Table~\ref{tab:quasi-F4}.
  \item The assertion of Theorem~\ref{thm:BMM} holds for $\bG$ of type $F_4$.
 \end{enumerate}
\end{prop}

\begin{proof}
We first determine the decomposition of $\RLG$ in the relevant cases.
If $\bL$ is 1-split, this is given by the usual 1-Harish-Chandra theory.
Secondly, if $\bL$ is a maximal torus, or if $\la$ is uniform, this was
determined by Lusztig
\cite[Thm.~4.23]{Lu84}. Thus, the decomposition of $\RLG$ is known in
all cases listed in Table~\ref{tab:quasi-F4}, and also for their Ennola duals
unless $\ell=2$, $e=2$, and $\bL$ is the Ennnola twist of lines 2 or~3 in
case~2, or $\ell=3$, $e=2$ and $\bL^F=\Ph2^2.B_2(q)$ is the Ennola twist of
case~4 or~6. In the second situation, by the Mackey formula in
Theorem~\ref{thm:RLG}(d), $\RLH(\la)$, with $\bH\ge\bL$ an $e$-split Levi
subgroup of type $B_3$ or $C_3$ has norm~2, while $\RHG(\mu)$, for $\mu$ a
constituent of $\RLH(\la)$, has norm~3. So in both cases the decomposition
can be recovered from the uniform projections, for which the decomposition
is known by Lusztig's work. Similarly, in the case that $\ell=2$ we use that
$\RLG(\la)$ has norm~3 to determine its decomposition.
\par
It turns out that all decompositions are independent of $q$.
Both~(a) and~(b) can now be checked from these decompositions.
\end{proof}

We now verify the assumptions for Proposition~\ref{prop:excellent}.

\begin{lem}   \label{lem:F4}
 Let $\bL$ and $\ell$ be as in Table~\ref{tab:quasi-F4}, with $e=e_\ell(q)=1$.
 Then:
 \begin{enumerate}[\rm(a)]
  \item in Cases 1--8, $\bL=C_\bG(Z(\bL)_\ell^F)$ and $\bL$ is
   $(e,\ell)$-adapted;
  \item $\la$ is of quasi-central $\ell$-defect precisely in the numbered
   lines of the table; and
  \item in Case~2b, there is $z\in Z(\bL)_2^F$ with $C_\bG(z)$ of type $B_4$.
 \end{enumerate}
\end{lem}

\begin{proof}
This is easy to check using {\sf Chevie} or by hand calculations in the root
system of type $F_4$.
\end{proof}

In fact, in all numbered lines except~2, $\la$ is even of central $\ell$-defect.

\begin{cor}   \label{cor:F4sub}
 For each quasi-isolated $\ell'$-element $1\ne s\in \bG^{*F}$ the
 $e$-Harish-Chandra series above any $e$-cuspidal pair $(\bL,\la)$ below
 $\cE(\bG^F,s)$ is contained in a unique $\ell$-block of $\bG^F$.
\end{cor}

\begin{proof}
By Proposition~\ref{prop:F4RLG}(b) and Lemma~\ref{lem:F4} the assumptions of
Proposition~\ref{prop:excellent} are satisfied, so each
$e$-Harish-Chandra series in $\cE(\bG^F,s)$ lies in a unique $\ell$-block.
\end{proof}

We're now ready to determine the quasi-isolated $\ell$-blocks and their
defect groups:

\begin{prop}   \label{prop:F4-defgrp}
 Assume that $e_\ell(q)=1$.
 For any quasi-isolated $\ell'$-element $1\ne s\in \bG^{*F}=F_4(q)$ the block
 distribution of $\cE(\bG^F,s)$ is as indicated by the horizontal lines in
 Table~\ref{tab:quasi-F4}.
 \par
 For each $\ell$-block corresponding to one of the cases 1--8 in the table,
 there is a defect group $P\leq N_{\bG^F}(\bL,\la)$ with the structure
 described in Theorem~\ref{thm:mainblocks}.
 In particular, the defect groups are abelian precisely in cases~4, 6 and~8.
\end{prop}

\begin{proof}
In cases~1, 7 and~8 in particular, $\cE(\bG^F,s)$ is a single 1-Harish-Chandra
series. Then $\cE_\ell(\bG^F,s)$ must be an $\ell$-block by
Theorem~\ref{thm:BrMiHi}. In case~2b, by Lemma~\ref{lem:F4}(c) there is
$z\in Z(\bL)_2^F$ with centraliser $\bC$ of type $B_4$. But by
\cite[Prop.~1.5]{En08} each $\cE_2(\bC^F,s)$ is a single 2-block. So by
Proposition~\ref{prop:Cabanescrux} all constituents of $\RLG(\la)$ lie in a
unique 2-block. Since this 2-Harish-Chandra series contains all four
1-Harish-Chandra series under line~2, they all must lie in the same $2$-block.
In order to complete the  proof of the first assertion, it remains to
show that the blocks in lines 3 and 4 correspond to distinct blocks as well
as the blocks in lines 5 and 6. We will do this after
determining the defect groups.

By Lemma~\ref{lem:F4}, the assumptions on $(\bL,\la)$ of
Proposition~\ref{prop:defect:re} are satisfied. Let $P$ be as in
Proposition~\ref{prop:defect:re}. We show that $P$ is a defect group of the
corresponding block~$B$.
In lines 1,  2,  3, 5, 7, and~8 one checks the equality
$$|C_{\bG^*}(s)^F|_\ell = |Z^\circ(\bL)_\ell^F|\cdot |Z([\bL,\bL])_\ell^F|\cdot
  | W_{\bG^F}(\bL,\la)|_\ell,$$
whence by Proposition~\ref{prop:defect:re}(c), $P$ is a defect group of $B$.
Further, in cases 1, 2, 3, 5, and~7, $W_{\bG^F}(\bL,\la)$ is not an
$\ell'$-group, so by Proposition~\ref{prop:defect:re}(e), $P$ is not abelian.

In cases~4, 6 and~8, $Z^\circ(\bL)^F \cap[\bL,\bL]^F$ and
$W_{\bG^F}(\bL,\la)$ are both $\ell'$-groups, hence by
Proposition~\ref{prop:defect:re}(f),(g),
$ Z(\bL)_\ell^F =D =P$ is a defect group of $B$.

Finally, since the block corresponding to line 3 has non-abelian defect
groups whereas the one  corresponding to line 4 has abelian defect groups,
these lines correspond to different blocks.
Similarly, lines 5 and 6 correspond to different blocks.
\end{proof}

This completes the proof of Theorem~\ref{thm:mainblocks} for type $F_4$.

\section{The quasi-isolated blocks in $E_6(q)$ and $\tw2E_6(q)$} \label{sec:E6}

Here we prove Theorems~\ref{thm:mainblocks} and~\ref{thm:BMM} for $\bG$ a
simple simply connected group of type $E_6$. Let's first assume that
$\bG^F=E_6(q)_\SC$. The situation here is more complicated than for type $F_4$
since the dual group $\bG^*$ of adjoint type contains semisimple elements with
disconnected centralisers. In Table~\ref{tab:quasi-E6} we have collected the
six possible types of quasi-isolated elements $1\ne s\in \bG^{*F}$
and their centralisers according to Proposition~\ref{prop:Bonn}. Note that,
whether $\ell=2$ or $\ell=3$, we may have $e=e_\ell(q)=1$ or~2, which explains
the fact that each centraliser occurs twice in the table. \par
Again, for each element $s$ we have listed all $e$-cuspidal pairs
$(\bL,\la)$ below $\cE(\bG^F,s)$ up to $\bG^F$-conjugation.
(If $\bL$ is a proper Levi subgroup of $\bG$, the $e$-cuspidality of the given
character $\la$ is known by induction; when $\bL=\bG$ it will be a consequence
of the explicit decomposition of Lusztig induction.) We denote the characters
$\la$ as explained for $F_4$. Moreover, $\phi,\phi',\phi''$ denote the three
extensions of the unique 2-cuspidal unipotent character of $D_4(q)$ to its
extension by the graph automorphism of order~3.

\begin{table}[htbp]
\caption{Quasi-isolated blocks in $E_6(q)$}   \label{tab:quasi-E6}
\[\begin{array}{|r|r|l|llll|}
\hline
 \text{No.}& C_{\bG^*}(s)^F& (\ell,e)& \bL^F& C_{\bL^*}(s)^F& \la& W_{\bG^F}(\bL,\la)\\
\hline\hline
 1&        A_2(q)^3.3& (2,1)& \Ph1^6& \bL^{*F}& 1& A_2\wr3\\
\hline
 2&        A_2(q^3).3& (2,1)& \Ph1^2.A_2(q)^2& \Ph1^2\Ph3^2.3& 1& A_2\\
\hline
 3&   \Ph1^2.D_4(q).3& (2,1)& \Ph1^6& \bL^{*F}& 1& D_4.3\\
  &                  &      & \Ph1^2.D_4(q)& \bL^{*F}& D_4[1]& 3\\
\hline
 4&   \Ph1\Ph2.\tw2D_4(q)& (2,1)& \Ph1^4.A_1(q)^2& \Ph1^4\Ph2^2& 1& B_3\\
\hline
 5& \Ph3.\tw3D_4(q).3& (2,1)& \Ph1^2.A_2(q)^2& \Ph1^2\Ph3^2.3& 1& G_2\\
  &                  &      & \bG^F& C_{\bG^*}(s)^F& \tw3D_4[\pm1]& 1\\
\hline
 6& A_2(q^2).\tw2A_2(q)& (2,1)& \Ph1^3.A_1(q)^3& \Ph1^3\Ph2^3& 1& A_2\ti A_1\\
  &                  &      &  \Ph1^2.D_4(q)&  \Ph1^2\Ph2^2.\tw2A_2(q)& \phi_{21} & A_2\\
\hline\hline
 7&        A_2(q)^3.3& (2,2)& \Ph1^2\Ph2^3.A_1(q)& \Ph1^3\Ph2^3& 1& A_1\wr3\\
  &                  & & \Ph1\Ph2^2.A_3(q)& \Ph1^2\Ph2^2.A_2(q)& \phi_{21}& A_1\ti A_1\\
  &                  & & \Ph2.A_5(q)& \Ph1\Ph2.A_2(q)^2& \phi_{21}\otimes\phi_{21}& A_1\\
  &                  & & \bG^F& C_{\bG^*}(s)^F& \phi_{21}^{\otimes3}& 1\\
\hline
 8&        A_2(q^3).3& (2,2)& \Ph2.A_2(q^2)A_1(q)& \Ph1\Ph2\Ph3\Ph6.3& 1& A_1\\
  &                  & & \bG^F& C_{\bG^*}(s)^F& \phi_{21}& 1\\
\hline
 9&   \Ph1^2.D_4(q).3& (2,2)& \Ph1^2\Ph2^4& \bL^{*F}& 1& D_4.3\\
  &                  & & \bG^F& C_{\bG^*}(s)^F& \phi,\phi',\phi''& 1\\
\hline
10&   \Ph1\Ph2.\tw2D_4(q)& (2,2)& \Ph1^2\Ph2^4& \bL^{*F}& 1& B_3\\
\hline
11&  \Ph3.\tw3D_4(q).3& (2,2)& \Ph2^2.A_2(q^2)& \Ph2^2\Ph3\Ph6.3& 1& G_2\\
  &                   &      & \bG^F& C_{\bG^*}(s)^F& \phi_{2,1},\phi_{2,2}& 1\\
\hline
 12&  A_2(q^2).\tw2A_2(q)& (2,2)& \Ph1^2\Phi_2^4& \bL^{*F}& 1& A_2\ti A_2\\
\hline\hline
 13&      A_5(q)A_1(q)& (3,1)& \Ph1^6& \bL^{*F}& 1& A_5\ti A_1\\
\hline\hline
 14&      A_5(q)A_1(q)& (3,2)& \Ph1^2\Ph2^4& \bL^{*F}& 1& C_3\ti A_1\\
 15&                  &      & \Ph2.A_5(q)& \bL^{*F}& \phi_{321}& A_1\\
\hline
\end{array}\]
\end{table}

The column headed $W_{\bG^F}(\bL,\la)$ describes the relative Weyl group for
the given $e$-cuspidal pairs as a Coxeter group, possibly extended by a
cyclic group of order~3 if $C_{\bG^*}(s)$ is disconnected.

We now proceed as in the case of $F_4$ and first discuss the decomposition of
$\RLG$ for each line in Table~\ref{tab:quasi-E6}:

\begin{prop}   \label{prop:E6RLG}
 Let $1\ne s\in \bG^{*F}=E_6(q)_\ad$ be a quasi-isolated $\ell'$-element, and
 $e=e_\ell(q)$. Then we have:
 \begin{enumerate}[\rm(a)]
  \item $\cE(\bG^F,s)$ is the disjoint union of the $e$-Harish-Chandra
   series listed in Table~\ref{tab:quasi-E6}.
  \item The assertion of Theorem~\ref{thm:BMM} holds for $\bG$ of type $E_6$.
 \end{enumerate}
\end{prop}

\begin{proof}
The characters of all proper Levi subgroups in cases~1, 4, 6, 7, 10, and
12--15 are
uniform, so the decomposition of Lusztig induction can be reduced to the known
decomposition of $\RTG(\theta)$ for suitable maximal tori $\bT$. The same
is true for the first line in cases~3 and~9. Whenever $\bL=\bG$, there is
nothing to do. For each of the two Levi subgroups $\bL$ of type $A_2^2$
(cases~2, 5 and~11) there are three $N_{\bG^F}(\bL)$-orbits of characters
of degree
$\thrd\Ph1^4\Ph2^2$, their sum being uniform. Since $\bL$ only involves
factors of type $A$, Lusztig induction of this sum can be decomposed. In the
second line in case~3, $\RLG(\la)$ has norm~$3$ by Theorem~\ref{thm:RLG}(d),
and from its known degree one concludes that it equals the sum of the three
remaining characters of $\cE(\bG^F,s)$ not occurring in the $e$-Harish-Chandra
series in line~3 of the table. The same considerations apply to case~8.
\par
It follows from the explicit decompositions that both~(a) and~(b) hold.
\end{proof}

The following is easily checked by explicit computation:

\begin{lem}   \label{lem:E6}
 Let $\bL$ and $\ell$ be as in Table~\ref{tab:quasi-E6}, with $e=e_\ell(q)$.
 Then:
 \begin{enumerate}[\rm(a)]
  \item $\bL=C_\bG(Z(\bL)_\ell^F)$ and $\bL$ is $(e,\ell)$-adapted; and
  \item in the table, $\la$ is of quasi-central $\ell$-defect precisely in the
   numbered lines.
 \end{enumerate}
\end{lem}

In fact, in all numbered lines except~6--8, $\la$ is even of central
$\ell$-defect.

By Proposition~\ref{prop:E6RLG}(b) and Lemma~\ref{lem:E6}, the assumptions of
Proposition~\ref{prop:excellent} are satisfied, so again each
$e$-Harish-Chandra series in Table~\ref{tab:quasi-E6} is contained in a
unique $\ell$-block of $\bG^F$.

\begin{prop}   \label{prop:E6-defgrp}
 Let $e=e_\ell(q)$. For any quasi-isolated $\ell'$-element
 $1\ne s\in \bG^{*F}=E_6(q)_\ad$ the block distribution of $\cE(\bG^F,s)$ is as
 indicated by the horizontal lines in Table~\ref{tab:quasi-E6}.

 For each $\ell$-block corresponding to one of the cases~1--15 in the table
 there is a defect group $P\leq N_{\bG^F}(\bL,\la)$ with the structure
 described in Theorem~\ref{thm:mainblocks}.
 In particular, the defect groups are abelian  precisely in case~15.
\end{prop}

\begin{proof}
In cases~1, 2, 4, 10, 12 and~13, $\cE(\bG^F,s)$ is a single
$e$-Harish-Chandra series, hence $\cE_\ell(\bG^F,s)$ is an $\ell$-block by
Theorem~\ref{thm:BrMiHi}. The Levi subgroup in the second line of case~3
contains the one in the first line of case~3. In the second line of case~3, 
the irreducible characters in
$\cE(\bL^F,s)$ are products of a fixed linear character of $Z(\bL)^F$ of
order~2 with unipotent characters of the derived group $[\bL,\bL]^F$ of type
$D_4$. Now by \cite[Prop.~1.5]{En08} all unipotent characters of this derived
group are contained in a single 2-block, hence all elements  of $\cE(\bL^F,s)$ 
are in the same  $2$-block so the two Harish-Chandra series of $\bG^F$ lie
above a single 2-block of $\bL^F$, and hence lie in a single 2-block of
$\bG^F$ by Proposition~\ref{prop:excellent}. The same argument applies to
case~6, using again that $\cE_2(\bL^F,s)$ forms a single 2-block and that
the Levi subgroup corresponding to the second line contains the 
one corresponding to the first.
\par
In case~7 we also use 1-Harish-Chandra theory from the 1-cuspidal pair
$(\bL,\la)$ in line~1. It turns out that all assertions of Lemma~\ref{lem:E6}(a)
are also satisfied there when $q\equiv3\pmod4$. Then, by
Proposition~\ref{prop:excellent}, all constituents of $\RLG(\lambda)$ lie
in a single 2-block. Since this 1-Harish-Chandra series contains all
2-Harish-Chandra series below~7, the latter must form a single
2-block. The same argument applies to line~9, using line~3. \par
For case~8, we verify that the 1-cuspidal pair $(\bL,\la)$ in line~2, with
$q\equiv3\pmod4$, satisfies $\bL=C_\bG^\circ(Z(\bL)_\ell^F)$ and that 
$\la $ is of central $\ell$-defect. We may conclude by 
Proposition~\ref{prop:decomp} that the
1-Harish-Chandra series in~2 lies in a unique 2-block. Since this contains
both 2-Harish-Chandra series below~8, these lie in a single 2-block.
\par
In case~5 all character of $\cE(\bG^F,s)$ but three (corresponding to the
cuspidal unipotent character of $D_4$) lie in the same 1-Harish-Chandra series,
hence in the same 2-block. Now we also consider Lusztig induction from the
2-split Levi subgroup with fixed point group $\Ph2^2.A_2(q^2)$ (line~11 with
$q\equiv1\pmod4$).
Then again all characters in $\cE(\bG^F,s)$ but three different ones lie in the
same 2-Harish-Chandra series, hence in the same 2-block.
The same argument applies to case~11, using line~5.
\par
In order to complete the proof of the distribution of blocks,
it remains only to show that lines~14 and~15 correspond to different blocks,
and this will be done after the determination of defect groups.

By Lemma~\ref{lem:E6}, the assumptions on $(\bL,\la)$ of
Proposition~\ref{prop:defect:re} are satisfied. Let $P$ be as in
Proposition~\ref{prop:defect:re}. In lines~1--14 one checks the assumption
of Proposition~\ref{prop:defect:re}(c), whence $P$ is a defect group of $B$.
Further, in all these cases $W_{\bG^F}(\bL,\la)$ is not an $\ell'$-group,
so  by Proposition~\ref{prop:defect:re}(e), $P$ is not abelian.
In case~15, $Z^\circ(\bL)^F \cap[\bL, \bL]^F$ and $W_{\bG^F}(\bL,\la)$ are both
$\ell'$-groups, hence by Proposition~\ref{prop:defect:re}(f) and~(g),
$Z(\bL)_\ell^F =D =P$ is a defect group of $B$.

Finally, since the block corresponding to line~14 has non-abelian defect
groups whereas the one corresponding to line~15 has abelian defect groups,
these lines correspond to different blocks.
\end{proof}

This completes the proof of Theorem~\ref{thm:mainblocks} for $G=E_6(q)$.

\medskip
The Lusztig series to consider in $\tw2E_6(q)$ are Ennola duals of those in
$E_6(q)$, and thus precisely the same arguments as for the latter case apply.
We obtain $\ell$-blocks as in Table~\ref{tab:quasi-E6}, with the cases
$(\ell,1)$ and $(\ell,2)$ interchanged, and the Levi subgroups replaced by
their Ennola-duals.

\section{The quasi-isolated blocks in $E_7(q)$} \label{sec:E7}

We now prove Theorems~\ref{thm:mainblocks} and~\ref{thm:BMM} for $\bG$ a
simple simply connected group of type $E_7$, so $\bG^F=E_7(q)_\SC$. The
relevant non-central quasi-isolated elements $s\in \bG^{*F}$ and their
centralisers when $q\equiv1\pmod4$ (for the first two entries) respectively
$q\equiv1\pmod3$ (for the remaining entries) are given in
Table~\ref{tab:quasi-E7} according to Proposition~\ref{prop:Bonn}.
Thus, we have $e=e_\ell(q)=1$ for the cases listed in the table, and hence
$\ell|(q-1)$. The cases where $q\equiv3\pmod4$ and $\ell=2$ (respectively
$q\equiv2\pmod3$ and $\ell=3$) are obtained from these by Ennola duality.
Note that cases 12, 15, 16 and~19 only occur for $q\equiv1\pmod4$, and cases
13, 17, 18 and~20 only for $q\equiv3\pmod4$.
\par
As for $F_4$ and $E_6$, in each case we give all relevant
1-cuspidal pairs $(\bL,\la)$ (up to $\bG^F$-conjugation) lying below characters
from $\cE(\bG^F,s)$ and their relative Weyl groups. Case~2b, with $e=2$,
case~10b, with $e=3$, will be used to further investigate the $\ell$-blocks
in cases~2 and~10.
\par
In order to fit the table on the page, we've adopted the following notation
for the Levi subgroups $\bL$, except in lines 2b and 10b: we just give the
Dynkin type of the derived subgroup $[\bL,\bL]$, with the understanding that
$\bL$ contains a maximally split torus (since $e=1$).

\begin{table}[htbp]
\caption{Quasi-isolated blocks in $E_7(q)$}   \label{tab:quasi-E7}
\[\begin{array}{|r|r|l|llll|}
\hline
 \text{No.}& C_{\bG^*}(s)^F& (\ell,e)& \bL^F& C_{\bL^*}(s)^F& \la& W_{\bG^F}(\bL,\la)\\
\hline\hline
 1&         A_5(q)A_2(q)& (2,1)& \emptyset& \bL^{*F}& 1& A_5\ti A_2\\
\hline
 2& \tw2A_5(q)\tw2A_2(q)& (2,1)& A_1^3& \Ph1^4\Ph2^3& 1& C_3\ti A_1\\
  &                     &      & D_4& \Ph1^3\Ph2^2.\tw2A_2(q)& \phi_{21}& C_3\\
  &                     &      & D_6& \Ph1\Ph2.\tw2A_5(q)& \phi_{321}& A_1\\
  &                     &      & E_7& C_{\bG^*}(s)^F& \phi_{321}\otimes\phi_{21}& 1\\
\hline\hline
 3&  D_6(q)A_1(q)& (3,1)& \emptyset& \bL^{*F}& 1& D_6\ti A_1\\
 4&              &      & D_4& \bL^{*F}& D_4[1]& B_2\ti A_1\\
\hline
 5&      A_7(q).2& (3,1)& \emptyset& \bL^{*F}& 1& A_7.2\\
\hline
 6&  \tw2A_7(q).2& (3,1)& (A_1^3)'& \Ph1^4\Ph2^3.2& 1& C_4\\
 7&              &      & D_6& \Ph1\Ph2.\tw2A_5(q).2& \phi_{321}& A_1\\
\hline
 8& \Ph1.E_6(q).2& (3,1)& \emptyset& \bL^{*F}& 1& E_6.2\\
 9&              &      & D_4& \bL^{*F}& D_4[1]& A_2.2\\
  &              &      & E_6& \bL^{*F}& E_6[\theta^{\pm1}]& 2\\
\hline
10& \Ph2.\tw2E_6(q).2& (3,1)& (A_1^3)'& \Ph1^4\Ph2^3.2& 1& F_4\\
  &                  &      & E_7& C_{\bG^*}(s)^F& \tw2E_6[\theta^{\pm1}],\tw2E_6[1]\!\!& 1\\
11&                  &      & D_6& \Ph1\Ph2.\tw2A_5(q).2& \phi_{321}& A_1\\
\hline
12& A_3(q)^2A_1(q).2& (3,1)& \emptyset& \bL^{*F}& 1& A_3\wr2\ti A_1\\
\hline
13& \tw2A_3(q)^2A_1(q).2& (3,1)& A_1^2& \Ph1^5\Ph2^2& 1& B_2\wr2\ti A_1\\
\hline
14& A_3(q^2)A_1(q).2& (3,1)& (A_1^3)'& \Ph1^4\Ph2^3.2& 1& A_3\ti A_1\\
\hline
15& \Ph1.D_4(q)A_1(q)^2.2& (3,1)& \emptyset& \bL^{*F}& 1& (D_4\ti A_1^2).2\\
16&                 &      & D_4& \bL^{*F}& D_4[1]& A_1\wr2\\
\hline
17& \Ph2.D_4(q)A_1(q)^2.2& (3,1)& A_1& \Ph1^6\Ph2& 1& (D_4\ti A_1^2).2\\
18&                      &      &D_4\cd A_1& \Ph1^2\Ph2D_4(q)& D_4[1]& A_1\wr2\\
\hline
19& \Ph1.\tw2D_4(q)A_1(q^2).2& (3,1)& A_1^2& \Ph1^5\Ph2^2& 1& (B_3\ti A_1).2\\
\hline
20& \Ph2.\tw2D_4(q)A_1(q^2).2& (3,1)& (A_1^3)'& \Ph1^4\Ph2^3.2& 1& B_3\ti A_1\\
\hline\hline
 2b&                         & (2,2)& \Ph2^7& \bL^{*F}& 1& A_5\ti A_2\\
\hline\hline
 10b&                        & (3,3)& \Ph3^2.A_1(q^3)& \Ph2\Ph3^2\Ph6.2& 1& G_5\\
\hline
\end{array}\]
\end{table}

(We remark that the conjugacy class of parabolic subgroups of type $A_1^3$ of
$W(E_7)$ with normaliser quotient $F_4$, denoted by $(A_1^3)'$ in the above
table, seems to have been overseen in \cite{How}.)

\begin{prop}   \label{prop:E7RLG}
 Let $1\ne s\in \bG^{*F}=E_7(q)_\ad$ be a quasi-isolated $\ell'$-element, and
 assume that $e=e_\ell(q)=1$. Then we have:
 \begin{enumerate}[\rm(a)]
  \item $\cE(\bG^F,s)$ is the disjoint union of the $e$-Harish-Chandra
   series listed in the upper part of Table~\ref{tab:quasi-E7}.
  \item The assertion of Theorem~\ref{thm:BMM} holds for $\bG$ of type $E_7$.
 \end{enumerate}
\end{prop}

\begin{proof}
Whenever $\la$ is uniform, the decomposition of $\RLG$ is obtained from the
known decomposition of $\RTG$ for various maximal tori $\bT$ of $\bG$.
Secondly, whenever the relative Weyl group is of order~2,
$\RLG(\la)$ is of norm~2 by the Mackey formula, and its constituents are
easily determined from the uniform projection. Furthermore, in all cases the
induction to a Levi subgroup of type $E_6$ respectively $\tw2E_6$ is known by
the results of the previous section. The norm of characters induced from these
Levi subgroups is small enough to again determine them uniquely from their
uniform projections.
\end{proof}

The conditions on $\bL$ and on $\la$ can be checked as in the previous cases:

\begin{lem}   \label{lem:E7}
 Let $\bL$ and $\ell$ be as in cases 1--20 in Table~\ref{tab:quasi-E7} (and
 recall that $e=e_\ell(q)=1$). Then:
 \begin{enumerate}[\rm(a)]
  \item $\bL=C_\bG(Z(\bL)_\ell^F)$, and $\bL$ is $(e,\ell)$-adapted; and
  \item in the table, $\la$ is of quasi-central $\ell$-defect precisely in the
   numbered lines.
 \end{enumerate}
 Additionally, in cases 2b and 10b we have $\bL=C_\bG^\circ(Z(\bL)_\ell^F)$.
\end{lem}

In fact, in all numbered lines except~2, $\la$ is even of central $\ell$-defect.

\begin{prop}   \label {prop:E7-defgrp}
 Assume that $e_\ell(q)=1$.
 For any quasi-isolated $\ell'$-element $1\ne s\in\bG^{*F}=E_7(q)_\ad$
 the block distribution of $\cE(\bG^F,s)$ is as indicated by the horizontal
 lines in Table~\ref{tab:quasi-E7}.
 \par
 For each $\ell$-block corresponding to one of the cases 1--20 in the table,
 there is a defect group $P\leq N_{\bG^F}(\bL,\la)$ with the structure
 described in Theorem~\ref{thm:mainblocks}.
 In particular, the defect groups are abelian precisely in
 cases~4, 7, 11, 13, 16 and~18.
\end{prop}

\begin{proof}
By Proposition~\ref{prop:excellent}, each $e$-Harish-Chandra series in
Table~\ref{tab:quasi-E7} is contained in a unique $\ell$-block of $\bG^F$.
In cases~1, 5, 12--14, 19 and~20, $\cE(\bG^F,s)$ is a single $\ell$-block by
Theorem~\ref{thm:BrMiHi}. By \cite[Tables for $E_6(q)$]{En08} all unipotent
characters of positive 3-defect of the Levi subgroup of type $E_6$ lie in the
same 3-block, so by Proposition~\ref{prop:excellent} the Harish-Chandra series
in line~9 and the following line belong to the same 3-block. Here note that 
the Levi subgroup in the second line in each case contains 
the one in the first line.  \par
In case~2, we claim that all four Harish-Chandra series lie in the same 2-block.
For this note that the 2-split Levi subgroup $\bL$ in case~2b satisfies
$\bL=C_\bG^\circ(Z(\bL)_\ell^F)$, and then the claim follows from
Proposition~\ref{prop:decomp} applied to the 2-Harish-Chandra series in case~2b,
which contains all Harish-Chandra series from case~2.
By the same arguments, the 1-cuspidal characters $\la=E_6[\theta^{\pm1}]$ in
the second line of case~10 lie in the same block as line~10, since these lie
in the same 3-Harish-Chandra series as in case~10b. We will show that
different numbered lines corresponding to the same quasi-isolated element
lie in different blocks after the determination of the defect groups.
\par
Now let $B$ be an $\ell$-block in $\cE_\ell(\bG^F,s)$
and let $P$ be as in Proposition~\ref{prop:defect:re}. In all numbered lines
which are in the first line of the part of the table corresponding to $s$ we
have that $P$ is a defect group of $B$ by Proposition~\ref{prop:defect:re}(c).
Further, in all of these cases, except line~13, $W_{\bG^F}(\bL,\la)$ is not an
$\ell'$-group, so by Proposition~\ref{prop:defect:re}(e), $P$ is not abelian.
\par
For lines~4, 7, 11, 13, 16 and~18, $Z^\circ(\bL)^F \cap[\bL, \bL]^F$ and
$W_{\bG^F}(\bL,\la)$ are both $\ell'$-groups, hence by
Proposition~\ref{prop:defect:re}(f),(g), $Z(\bL)_\ell^F=D=P$ is a defect
group of $B$.

For case 9, we note that by Proposition~\ref{prop:defect:re}(b),(g)
there is a subgroup $P$ of a defect group of $B$  of the required type and
with $D=A =Z(\bL)_3^F$. Further, $Z(\bL)_3^F= E^3$, where $E$ is a cyclic
group of order $(q-1)_3$,  $Z(\bL)_3^F$ has index~$3$ in $P$ and if
$\sigma \in  P\setminus Z(\bL)_3^F$, then $\sigma$ acts on $E^3$ by cyclically
permuting the factors.
Thus $C_P(\sigma)$ has order $3(q-1)_3$ whereas  $C_P(\tau)$  for any
$\tau \in Z(\bL)_3^F$ has order at least $3(q-1)_3^3 $.  So,
$Z(\bL)_3^F$ is characteristic in $Q$, and it  follows from
Proposition~\ref{prop:defect:re}(d) that $P$ is a defect group of $B$.
\par
The defect groups in cases~8 and~9 have different orders, hence they
correspond to different blocks. The defect groups in cases~3, 6, 10, 15 and~17
are non-abelian whereas those in cases~4, 7, 11, 16 and~18 are abelian,
hence correspond to different blocks.
\end{proof}

\section{The quasi-isolated blocks in $E_8(q)$} \label{sec:E8}

Throughout this section, $\bG$ is a simple group of type $E_8$, so
$\bG^F=E_8(q)$.
The situation is yet more complicated since now there are three bad primes
$\ell=2,3,5$ to deal with, which we'll do one at a time. Until
Section~\ref{subsec:E8(2)} we assume that $q\ne2$.

\subsection{Quasi-isolated 2-blocks of $E_8(q)$}
We begin by considering the case when $\ell=2$. Table~\ref{tab:quasi-E8-2}
contains the possible rational types of centralisers of quasi-isolated 3-
and 5-elements $1\ne s\in \bG^{*F}$, all $e$-cuspidal pairs $(\bL,\la)$ with
$s\in\bL^{*F}$ and their relative Weyl groups for the case $q\equiv1\pmod4$.
Here, a quasi-isolated 3-element as in cases~1 and~3 occurs when
$q\equiv1\pmod3$, as in cases~2 and~5 when $q\equiv2\pmod3$; and a
quasi-isolated 5-element as in case~7 occurs when $q\equiv1\pmod5$, as in
case~8 when $q\equiv2,3\pmod5$ and as in case~9 when $q\equiv-1\pmod5$.
The notation for Levi subgroups and for the cuspidal characters is as in
Table~\ref{tab:quasi-E7} above.

The cases where $q\equiv3\pmod4$ can be obtained from the former by Ennola
duality.

\begin{table}[htbp]
\caption{Quasi-isolated 2-blocks in $E_8(q)$, $q\equiv1\pmod4$}
   \label{tab:quasi-E8-2}
\[\begin{array}{|c|r|l|llll|}
\hline
 \text{No.}& C_{\bG^*}(s)^F& e& \bL^F& C_{\bL^*}(s)^F& \la& W_{\bG^F}(\bL,\la)\\
\hline\hline
 1&     A_8(q)& 1& \emptyset& \bL^{*F}& 1& A_8\\
\hline
 2&    \tw2A_8(q)& 1& A_1^4& \Ph1^4\Ph2^4& 1& B_4\\
  &              &  & D_4\cd A_1& \Ph1^3\Ph2^3.\tw2A_2(q)& \phi_{21}& B_3\\
\hline
 3& E_6(q).A_2(q)& 1& \emptyset& \bL^{*F}& 1& E_6\ti A_2\\
  &              &  & D_4& \bL^{*F}& D_4[1]& G_2\ti A_2\\
 4&              &  & E_6& \bL^{*F}& E_6[\theta^{\pm1}]& A_2\\
\hline
 5& \tw2E_6(q).\tw2A_2(q)& 1& A_1^3& \Ph1^5\Ph2^3& 1& F_4\ti A_1\\
  &              &  & D_4& \Ph1^4\Ph2^2.\tw2A_2(q)& \phi_{21}& F_4\\
  &              &  & D_6& \Ph1^2\Ph2.\tw2A_5(q)& \phi_{321}& A_1\ti A_1\\
  &              &  & E_7& \Ph1.\tw2A_5(q)\tw2A_2(q)& \phi_{321}\otimes\phi_{21}& A_1\\
  &              &  & E_7& \Ph1\Ph2.\tw2E_6(q)& \tw2E_6[1]& A_1\\
  &              &  & E_8& C_{\bG^*}(s)^F& \tw2E_6[1]\otimes\phi_{21}& 1\\
 6&              &  & E_7& \Ph1\Ph2.\tw2E_6(q)& \tw2E_6[\theta^{\pm1}]& A_1\\
  &              &  & E_8& C_{\bG^*}(s)^F& \tw2E_6[\theta^{\pm1}]\otimes\phi_{21}& 1\\
\hline
 7&      A_4(q)^2& 1& \emptyset& \bL^{*F}& 1& A_4^2\\
\hline
 8&  \tw2A_4(q^2)& 1& A_3^2& \Ph1^2\Ph2^2\Ph4^2& 1& B_2\\
  &              &  & D_7& \Ph1\Ph2\Ph4.\tw2A_2(q^2)& \phi_{21}& A_1\\
\hline
 9& \tw2A_4(q)^2& 1& A_1^4& \Ph1^4\Ph2^4& 1& B_2^2\\
  &             &      & D_4\cd A_1& \Ph1^3\Ph2^3.\tw2A_2(q)& \phi_{21}\quad(2\times)& B_2\ti A_1\\
  &             &      & D_6& \Ph1^2\Ph2^2.\tw2A_2(q)^2& \phi_{21}\otimes\phi_{21}& A_1^2\\
\hline\hline
 5b& \tw2E_6(q).\tw2A_2(q)& 2& \Ph2^8& \bL^{*F}& 1& E_6\ti A_2\\
 6b&                      &  & \Ph2^2.\tw2E_6(q)& \bL^{*F}& \tw2E_6[\theta^{\pm1}]& A_2\\
\hline
\end{array}\]
\end{table}

Lets' point out one particularity here. Since $E_7$ has two non-conjugate
Levi subgroups of type $A_1^3$ (see the remark before
Proposition~\ref{prop:E7RLG}), the quasi-isolated involution in case~5
embeds in two different ways into a 1-split Levi subgroup of type $E_7$,
with non-isomorphic centralizers (see rows 4 and 5 in case~5).

\begin{prop}   \label{prop:E8RLGl2}
 Let $1\ne s\in\bG^{*F}=E_8(q)$ be a quasi-isolated $2'$-element and assume
 that $e=e_\ell(q)=1$. Then we have:
 \begin{enumerate}[\rm(a)]
  \item $\cE(\bG^F,s)$ is the disjoint union of the $e$-Harish-Chandra
   series listed in the upper part of Table~\ref{tab:quasi-E8-2}.
  \item The assertion of Theorem~\ref{thm:BMM} holds for $\bG$ of type
   $E_8$ with $\ell=2$.
 \end{enumerate}
\end{prop}

\begin{proof}
We determine the decomposition of $\RLG(\la)$ for the $e$-Harish-Chandra
series occurring in Table~\ref{tab:quasi-E8-2} as in the previous proofs,
using mainly the Mackey formula and transitivity.
\end{proof}

\begin{lem}   \label{lem:E8-2}
 Let $\bL$ be as in cases 1--9 of Table~\ref{tab:quasi-E8-2}, and recall that
 $q\equiv1\pmod4$. Then $\bL=C_\bG(Z(\bL)_2^F)$, and $\bL$ is $(1,2)$-adapted.
 In each numbered line of the table, and no other, $\la$ is of quasi-central
 $2$-defect. It is of central $\ell$-defect in the lines~1, 3, 4, 7, 5b and~6b.

 Moreover, in Cases~5b and~6b we have $\bL=C_\bG^\circ(Z(\bL)_2^F)$.
\end{lem}

Note that for $q\equiv3\pmod4$ this is no longer true; there are many cases
for which $\bL<C_\bG^\circ(Z(\bL)_2^F)$:

\begin{exmp}
 Assume that $q\equiv3\pmod4$ and let $\bL^F$ be of type $\Phi_1^2.A_3(q)^2$.
 Then $C_\bG^\circ(Z(\bL)_2^F)$ is of type $D_4(q)^2$. Similarly, for $\bL^F$
 of type $\Phi_1^4.A_1(q)^4$ we have $C_\bG^\circ(Z(\bL)_2^F)$ is of
 type $A_1(q)^8$.
\end{exmp}

But as explained above, for that congruence we choose the Ennola duals of
the above Levi subgroups, and for those the analogue of Lemma~\ref{lem:E8-2}
continues to hold.

\begin{prop}   \label{prop:E8-2-defgrp}
 Suppose that $q\equiv1\pmod4$.
 For any quasi-isolated $2'$-element $1\ne s\in \bG^{*F}$ the block
 distribution of $\cE(\bG^F,s)$ is as indicated by the horizontal lines
 in the upper part of Table~\ref{tab:quasi-E8-2}.
 \par
 For each $2$-block corresponding to one of the cases 1--9 in the table,
 there is a defect group $P\leq N_{\bG^F}(\bL,\la)$ with the structure
 described in Theorem~\ref{thm:mainblocks}.
 In particular, the defect groups are non-abelian.
\end{prop}

\begin{proof}
We first prove part of the block distribution. Again, each $e$-Harish-Chandra
series in Table~\ref{tab:quasi-E8-2} is contained in a unique $2$-block of
$\bG^F$. The lines~1 and~7 in
Table~\ref{tab:quasi-E8-2} both correspond to a unique block by
Theorem~\ref{thm:BrMiHi}. The unnumbered Harish-Chandra series below cases~2,
3, 8 and~9 lie in the same $2$-block as the respective numbered line by
Proposition~\ref{prop:excellent} since all  characters of $\cE(\bL^F,s)$ lie
in the same $2$-block by \cite[Prop.~1.5]{En08}. \par
Similarly, all characters in each of the two Lusztig series $\cE(\bL^F,s)$,
for $\bL$ a Levi of type $E_7$ in rows 4 and~5 of case~5, lie in a single
2-block, except for those denoted $\tw2E_6[\theta^{\pm1}]$, hence so do the
characters in $\cE(\bG^F,s)$ above them.
To see that the cuspidal character in the line before case~6 belongs to the
block in case~5, we use the alternative 2-Harish-Chandra series above
$(\bL,\la)$ given in cases~5b, which by Lemma~\ref{lem:E8-2} still satisfies
$\bL=C_\bG^\circ(Z(\bL)_2^F)$. Thus Proposition~\ref{prop:decomp} applies.
The 1-Harish-Chandra series in line~6 and the subsequent line are both
contained in the 2-Harish-Chandra series above $(\bL,\la)$ in line~6b. As
$\la$ is of central defect, an application of Proposition~\ref{prop:decomp}
shows that both 1-Harish-Chandra series lie in the same 2-block.

Again, we defer the question  of  different numbered  lines corresponding to
different blocks to after the discussion on defect groups. 
\par
For any $\ell$-block $B$ in $\cE(\bG^F,s)$, let $ (D, v) \leq (P,w)$ be
$B$-Brauer pairs as in Proposition~\ref{prop:defect:re}.
In all numbered cases, $W_{\bG^F}(\bL,\la)$ is not a $2'$-group,
so by Proposition~\ref{prop:defect:re}(e), $P$ is not abelian.
In all numbered lines which are at the top of the part of the table
corresponding to a particular $s$, we conclude by
Proposition~\ref{prop:defect:re}(c) that $P$ is a defect group of~$B$.
\par
In case~4, by Proposition~\ref{prop:defect:re}(g), $D=A=Z(\bL)_2^F$. Further,
$Z(\bL)_2^F =E^2$, where $E$ is a cyclic group of order $(q-1)_2$. The Levi
subgroup of type $E_6$ is contained in the maximal rank subgroup of type
$E_6+A_2$, and $Z(\bL)_2^F.W_{\bG^F}(\bL,\la)$ is contained in the normaliser
of the maximal torus of the $A_2$-factor. Any 2-element
$\sigma\in P\setminus Z(\bL)_2^F$
interchanges two cyclic subgroups of $Z(\bL)_2^F$ of order at least~4, so
$Z(\bL)_2^F$ is the only abelian subgroup of $P$ properly containing
$Z(P)[P,P]$. Further, since $[P,P] \nsubseteq Z(P)$ and $Z(\bL)_2^F$ is
of index $2$ in $P$, $Z(\bL)_2^F =C_P([P,P]) $. Since any subgroup of index
$2$ of $P$ contains $[P,P]$, it follows that $Z(\bL)_2^F$ is the
unique abelian subgroup of index $2$ of $P$. In particular,
$Z(\bL)_2^F$ is characteristic in $P$ and it follows from
Proposition~\ref{prop:defect:re}(d) that $P$ is a defect group of~$B$.
\par
The Levi subgroup of type $E_7$ in case~6 lies in a maximal rank subgroup of
type $E_7+A_1$, and $\bL$ is a central product $E_7\circ\bT$, where $\bT$
is a split torus of $A_1$ and where the involution of the centre of
$E_7$ is identified with the involution of $\bT$.
Thus $Z(\bL)_2^F = |\bT^F|_2$ is cyclic of order $(q-1)_2$ and by
Proposition~\ref{prop:defect:re}(b), $Z(\bL)_2^F$ has index $2$ in $D$.
By considering the projection of $D$ into $\bT$, one sees that
$D$ is cyclic of order $2(q-1)_2 $. Further, if
$\sigma\in P\setminus D $, then $\sigma$ acts by inversion on $A$.
Since $D$ is cyclic of order at least $8$, and $A$ is of index $2$ in $D$,
it follows that $D$ is the unique cyclic subgroup of index $2$ in $P$.
Thus, $D$ and hence $A$ is characteristic in $P$.
Hence by Proposition~\ref{prop:defect:re}(d), $P$ is a defect group of $B$.
\par
Since the defect groups in cases~3 and~4 have different order as do the
defect groups in cases~5 and ~6, we see that these lines correspond to
distinct blocks. In case~4, as shown above $Z(\bL)_2^F$ is the unique abelian
subgroup of $P$ of index~2. So, if the Brauer pairs corresponding to the
two  choices of $\la$ in case~4 correspond to the same block,
then they are $\bG^F$-conjugate, and hence by Lemma~\ref{lem:E8-2}
the corresponding $e$-cuspidal pairs are $\bG^F$-conjugate,
which is not the case. Thus, the two  entries of case~4 correspond to
different blocks. A similar argument applies in case~6.
The subgroup $D$ is the unique cyclic subgroup of  $P$ of index $2$,
and  the group $Z(\bL)_2^F$ is the unique subgroup of index $2$ in $D$.
\end{proof}

\subsection{Quasi-isolated 3-blocks of $E_8(q)$}\label{subsec:E8-3}
Now let $\ell=3$. In Table~\ref{tab:quasi-E8-3} we present the
centralisers of quasi-isolated 2- and 5-elements together with data for
the relevant cuspidal pairs in the case where $q\equiv1\pmod3$. Again those for
$q\equiv2\pmod3$ are obtained by Ennola duality. As in the case when
$\ell=2$, there occurs just one type of quasi-isolated 5-elements, depending
on $q\pmod5$. The quasi-isolated 4-elements in cases~6 and~9 occur when
$q\equiv1\pmod4$, those in cases~8 and~10 when $q\equiv3\pmod4$.

\begin{table}[htbp]
\caption{Quasi-isolated 3-blocks in $E_8(q)$, $q\equiv1\pmod3$}
   \label{tab:quasi-E8-3}
\[\begin{array}{|c|r|llll|}
\hline
 \text{No.}& C_{\bG^*}(s)^F& \bL& C_{\bL^*}(s)^F& \la& W_{\bG^F}(\bL,\la)\\
\hline\hline
 1&       D_8(q)& \emptyset& \bL^{*F}& 1& D_8\\
 2&             & D_4& \bL^{*F}& D_4[1]& B_4\\
\hline
 3& E_7(q)A_1(q)& \emptyset& \bL^{*F}& 1& E_7\ti A_1\\
 4&             & D_4& \bL^{*F}& D_4[1]& C_3\ti A_1\\
  &             & E_6& \bL^{*F}& E_6[\theta^{\pm1}]& A_1\ti A_1\\
 5&             & E_7& \bL^{*F}& E_7[\pm\xi]& A_1\\
\hline
 6& D_5(q)A_3(q)& \emptyset& \bL^{*F}& 1& D_5\ti A_3\\
 7&             & D_4& \bL^{*F}& D_4[1]& A_3\ti A_1\\
\hline
 8& \tw2D_5(q).\tw2A_3(q)& A_1^2& \Ph1^6\Ph2^2& 1& B_4\ti C_2\\
\hline
 9& A_7(q)A_1(q)& \emptyset& \bL^{*F}& 1& A_7\ti A_1\\
\hline
10& \tw2A_7(q)A_1(q)& A_1^3& \Ph1^5\Ph2^3& 1& C_4\ti A_1\\
11&                 & D_6& \Ph1^2\Ph2.\tw2A_5(q)& \phi_{321}& A_1^2\\
\hline
12&      A_4(q)^2& \emptyset& \bL^{*F}& 1& A_4^2\\
\hline
13&  \tw2A_4(q^2)& A_3^2& \Ph1^2\Ph2^2\Ph4^2& 1& B_2\\
14&              & D_7& \Ph1\Ph2\Ph4.\tw2A_2(q^2)& \phi_{21}& A_1\\
\hline
15& \tw2A_4(q)^2& A_1^4& \Ph1^4\Ph2^4& 1& B_2^2\\
16&             & D_4\cd A_1& \Ph1^3\Ph2^3.\tw2A_2(q)& \phi_{21}\quad(2\times)& B_2\ti A_1\\
17&             & D_6& \Ph1^2\Ph2^2.\tw2A_2(q)^2& \phi_{21}\otimes\phi_{21}& A_1^2\\
\hline
\end{array}\]
\end{table}

\begin{prop}   \label{prop:E8RLGl3}
 Let $1\ne s\in\bG^{*F}=E_8(q)$ be a quasi-isolated $3'$-element and recall
 that $e=e_\ell(q)=1$. Then we have:
 \begin{enumerate}[\rm(a)]
  \item $\cE(\bG^F,s)$ is the disjoint union of the $e$-Harish-Chandra
   series listed in Table~\ref{tab:quasi-E8-3}.
  \item The assertion of Theorem~\ref{thm:BMM} holds for $\bG$ of type
   $E_8$, $q\ne2$, and $\ell=3$.
 \end{enumerate}
\end{prop}

\begin{proof}
The decomposition of $\RLG(\la)$ for the $e$-Harish-Chandra series~12--17 has
already been computed in the proof of Proposition~\ref{prop:E8RLGl2}. For the
remaining Harish-Chandra series, the usual arguments yield the claim.
\end{proof}

\begin{lem}   \label{lem:E8-3}
 Let $\bL$ be as in Table~\ref{tab:quasi-E8-3} and recall that $q\equiv1\pmod3$.
 Then $\bL=C_\bG(Z(\bL)_3^F)$, and $\bL$ is $(1,3)$-adapted. Moreover, in each
 numbered line of the table, and only in those, $\la$ is of central $3$-defect.
\end{lem}

We obtain:

\begin{prop}   \label {prop:E8-3-defgrp}
 Suppose that $q \equiv1\pmod3$. For any quasi-isolated $3'$-element
 $1\ne s\in \bG^{*F}$ the block distribution of $\cE(\bG^F,s)$ as indicated
 by the horizontal lines in Table~\ref{tab:quasi-E8-3}.
 \par
 For each $3$-block $B$ corresponding to one of the cases 1--17 in the table,
 there is a defect group $P\leq N_{\bG^F}(\bL,\la)$ with the structure
 described in Theorem~\ref{thm:mainblocks}.

 In particular, the defect groups of $B$ are abelian precisely in the
 cases~5, 11 and~13--17, and then $Z(\bL)_3^F$ is a defect group of $B$.
\end{prop}

\begin{proof}
Each $e$-Harish-Chandra series in Table~\ref{tab:quasi-E8-3} is contained in
a unique $3$-block of $\bG^F$. Next, note that lines~8, 9 and~12 correspond
to a single 3-block each.
The two 1-cuspidal unipotent characters $E_6[\theta^{\pm1}]$ of the derived
subgroup of the Levi subgroup of type $E_6$ below line~4 lie in the same
3-block of $E_6(q)$ as those above $D_4[1]$ by \cite{En00}, so by
Proposition~\ref{prop:excellent} their Harish-Chandra series are contained
in the 3-block from case~4. All other separations of blocks will be argued once
we've determined defect groups.
\par
Concerning the structure of the defect groups, in all numbered lines which are
at the top of the part of the table corresponding to a particular $s$, we
conclude as usual by Proposition~\ref{prop:defect:re}(c).
Further, for all of these except cases~13 and~15, $W_{\bG^F}(\bL,\la)$ is not a
$3'$-group, so $P$ is non-abelian by Proposition~\ref{prop:defect:re}(e).
\par
In cases~5, 11 and 13--17, $Z^\circ(\bL)^F \cap[\bL,\bL]^F$ and
$W_{\bG^F}(\bL,\la)$
are both $3'$-groups, hence by Proposition~\ref{prop:defect:re}(f),(g),
$Z(\bL)_3^F =D =P$ is a defect group of $B$.

In cases~2, 4, 7, by embedding $\bL$ in a maximal rank subgroup of type
$D_4+D_4$, we see that $\bL$ is a central product of $D_4$ with a split maximal
torus $\bT$ of type $D_4$ and $Z(\bL)_3^F=(\bT)_3^F$.
Since $Z^\circ(\bL)^F \cap[\bL, \bL]^F$ is a $3'$-group, by
Proposition~\ref{prop:defect:re}(g), $D=Z(\bL)_3^F$.
The action of $\sigma\in P\setminus D$ on $D =(\bT)_3^F$
can be determined through the action of the Weyl group of type $D_4$ on $\bT$.
We have $D=\langle z_1, z_2, z_3, z_4\rangle$ with $\sigma$
cyclically permuting the $z_1,z_2,z_3$ and fixing $z_4$. So,
$Z(P)[P,P] = \langle z_1z_2z_3, z_4, z_1z_2^{-1}, z_2z_3^{-1} \rangle$
is a subgroup of index $3$ in $D$ and it follows that
$D$ is the only abelian subgroup of $Q$ properly containing
$Z(P)[P,P]$. Thus, $D$ is characteristic in $P$, and it follows by
Proposition~\ref{prop:defect:re}(d) that $P$ is a defect group of $B$.
\par
In all cases, except the two represented by cases~5, respectively~16, one
sees that different numbered lines correspond to different blocks by comparing
orders of the defect group or noting that one of the lines corresponds to
abelian defect while the other doesn't.
To see that the two blocks represented by case~16 are different, note that
each has a maximal Brauer pair of the form $(Z(\bL)_3^F,\la)$ and that by
Lemma~\ref{lem:E8-3}, $\bL=C_\bG(Z(\bL)_3^F)$. Since the pairs $(\bL,\la)$ are
not $\bG^F$-conjugate, neither are the corresponding maximal Brauer pairs.
Similarly, the two blocks represented by case~5 are different.
\end{proof}

\subsection{Quasi-isolated 5-blocks of $E_8(q)$}\label{subsec:E8-5}

Finally, let $\ell=5$. Here, we distinguish two cases according to whether
$q\equiv\pm1\pmod5$ or $q\equiv\pm2\pmod5$. The cuspidal pairs for the case
$e=1$, are collected in Table~\ref{tab:quasi-E8-5}; here the decomposition of
$\RLG$ was already determined in the previous two subsections. The case $e=2$
is obtained from this by Ennola duality. Table~\ref{tab:quasi-E8-5b} contains
the relevant information in the case $e=4$. Here, the relative Weyl groups are,
in general, no longer true Weyl groups, but various types of complex reflection
groups occur.

\begin{table}[htbp]
\caption{Quasi-isolated 5-blocks in $E_8(q)$, $q\equiv1\pmod5$}   \label{tab:quasi-E8-5}
\[\begin{array}{|r|r|llll|}
\hline
 \text{No.}& C_{\bG^*}(s)^F& \bL& C_{\bL^*}(s)^F& \la& W_{\bG^F}(\bL,\la)\\
\hline\hline
 1&       D_8(q)& \emptyset& \bL^{*F}& 1& D_8\\
 2&             &     D_4& \bL^{*F}& D_4[1]& B_4\\
\hline
 3& E_7(q)A_1(q)& \emptyset& \bL^{*F}& 1& E_7\ti A_1\\
 4&             &     D_4& \bL^{*F}& D_4[1]& C_3\ti A_1\\
 5&             &     E_6& \bL^{*F}& E_6[\theta^{\pm1}]& A_1\ti A_1\\
 6&             &     E_7& \bL^{*F}& E_7[\pm\xi]& A_1\\
\hline
 7& D_5(q)A_3(q)& \emptyset& \bL^{*F}& 1& D_5\ti A_3\\
 8&             & D_4& \bL^{*F}& D_4[1]& A_1\ti A_3\\
\hline
 9& \tw2D_5(q).\tw2A_3(q)& A_1^2& \Ph1^6\Ph2^2& 1& B_4\ti C_2\\
\hline
10& A_7(q)A_1(q)& \emptyset& \bL^{*F}& 1& A_7\ti A_1\\
\hline
11& \tw2A_7(q)A_1(q)& A_1^3& \Ph1^5\Ph2^3& 1& C_4\ti A_1\\
12&                 & D_6& \Ph1^2\Ph2.\tw2A_5(q)& \phi_{321}& A_1^2\\
\hline
13&     A_8(q)& \emptyset& \bL^{*F}& 1& A_8\\
\hline
14&    \tw2A_8(q)& A_1^4& \Ph1^4\Ph2^4& 1& B_4\\
15&              & D_4\cd A_1& \Ph1^3\Ph2^3.\tw2A_2(q)& \phi_{21}& B_3\\
\hline
16& E_6(q).A_2(q)& \emptyset& \bL^{*F}& 1& E_6\ti A_2\\
17&              &     D_4& \bL^{*F}& D_4[1]& G_2\ti A_2\\
18&              &     E_6& \bL^{*F}& E_6[\theta^{\pm1}]& A_2\\
\hline
19& \tw2E_6(q).\tw2A_2(q)& A_1^3& \Ph1^5\Ph2^3& 1& F_4\ti A_1\\
20&                      & D_4& \Ph1^4\Ph2^2.\tw2A_2(q)& \phi_{21}& F_4\\
21&                      & D_6& \Ph1^2\Ph2.\tw2A_5(q)& \phi_{321}& A_1\ti A_1\\
22&                      & E_7& \Ph1\Ph2.\tw2E_6(q)& \tw2E_6[1],\tw2E_6[\theta^{\pm1}]& A_1\\
23&                      & E_7& \Ph1.\tw2A_5(q)\tw2A_2(q)& \phi_{321}\otimes\phi_{21}& A_1\\
24&                      & E_8& C_{\bG^*}(s)^F& \tw2E_6[1]\otimes\phi_{21},& \\
  &                      &  &           & \tw2E_6[\theta^{\pm1}]\otimes\phi_{21}& 1\\
\hline
\end{array}\]
\end{table}

\begin{table}[htbp]
\caption{Quasi-isolated 5-blocks in $E_8(q)$, $q\equiv\pm2\pmod5$}   \label{tab:quasi-E8-5b}
\[\begin{array}{|r|r|llll|}
\hline
 \text{No.}& C_{\bG^*}(s)^F& \bL^F& C_{\bL^*}(s)^F& \la& W_{\bG^F}(\bL,\la)\\
\hline\hline
25&       D_8(q)& \Ph4^4& \bL^{*F}& 1& G(4,2,4)\\               
26&             & \Ph4^2.D_4(q)& \bL^{*F}& 4 \chrs& G(4,1,2)\\  
27&             & \bG^F& C_{\bG^*}(s)^F& 4 \chrs& 1\\               
\hline
28& E_7(q)A_1(q)& \Ph4^2.D_4(q)& \Ph4^2.A_1(q)^4& 4 \chrs& G_8\\   
29&             & \Ph4^2.D_4(q)& \Ph4^2.A_1(q)^4& 4 \chrs& G(4,1,2)\\ 
30&             & \bG^F& C_{\bG^*}(s)^F& 32 \chrs& 1\\           
\hline
31& D_5(q)A_3(q)& \Ph4^3.A_1(q^2)& \Ph1\Ph2\Ph4^3& 1& G(4,1,2)\ti Z_4\\ 
32&             & \Ph4^2.D_4(q)& \Ph1\Ph4^2.A_3(q)& \phi_{22}& G(4,1,2)\\ 
33&             & \Ph4^2.D_4(q)& \Ph1\Ph4^2.\tw2A_3(q)& \phi_{22}& Z_4\ti Z_4\\ 
34&             & \Ph4.\tw2D_6(q)& \Ph1.\tw2A_3(q)A_3(q)& \phi_{22}\otimes\phi_{22}& Z_4\\  
35&             & \bG^F& C_{\bG^*}(s)^F& 4 \chrs& 1\\              
\hline
36& A_7(q)A_1(q)& \Ph4^2.D_4(q)& \Ph1\Ph2^2\Ph4^2.A_1(q)& 1,\phi_{11}& G(4,1,2)\\     
37&             & \Ph4.\tw2D_6(q)& \Ph1\Ph2\Ph4.A_3(q)A_1(q)& \phi_{22}\otimes1,\phi_{11}& Z_4\\   
38&             & \bG^F& C_{\bG^*}(s)^F& 8 \chrs& 1\\   
\hline
39&       A_8(q)& \Ph4^2.A_1(q^2)^2& \Ph1^2\Ph2^2\Ph4^2& 1& G(4,1,2)\\ 
40&             &  \Ph4.\tw2D_6(q)& \Ph1\Ph2\Ph4.A_4(q)& \phi_{41},\phi_{311},\phi_{2111}& Z_4\\   
41&             &  \bG^F& C_{\bG^*}(s)^F& 4 \chrs& 1\\   
\hline
42& E_6(q).A_2(q)& \Ph4^2.D_4(q)& \Ph1^2\Ph4^2.A_2(q)& 3 \chrs& G_8\\ 
43&              & \Ph4.\tw2D_6(q)& \Ph1\Ph4.\tw2A_3(q)A_2(q)& 3 \chrs& Z_4\\ 
44&              & \bG^F& C_{\bG^*}(s)^F& 30 \chrs& 1\\
\hline
\end{array}\]
\end{table}

\begin{prop}   \label{prop:E8RLGl5}
 Let $1\ne s\in\bG^{*F}=E_8(q)$ be a quasi-isolated $5'$-element. Then we
 have:
 \begin{enumerate}[\rm(a)]
  \item $\cE(\bG^F,s)$ is the disjoint union of the $e$-Harish-Chandra
   series listed in Tables~\ref{tab:quasi-E8-5} and~\ref{tab:quasi-E8-5b}.
  \item The assertion of Theorem~\ref{thm:BMM} holds for $\bG$ of type
   $E_8$, $q\ne2$, and $\ell=5$.
 \end{enumerate}
\end{prop}

\begin{proof}
The decomposition of $\RLG(\la)$ for the $e$-cuspidal pairs $(\bL,\la)$ in
Table~\ref{tab:quasi-E8-5} was already determined in
Propositions~\ref{prop:E8RLGl2} and~\ref{prop:E8RLGl3}. As for
Table~\ref{tab:quasi-E8-5b}, $\la$ is always uniform except in case~2, or when
$\bL=\bG$ (in which case $\RLG(\la)=\la$).
\end{proof}

\begin{lem}   \label{lem:E8-5}
 Let $\bL$ be as in Table~\ref{tab:quasi-E8-5} or~\ref{tab:quasi-E8-5b}.
 Then $\bL=C_\bG(Z(\bL)_5^F)$ and $\bL$ is $(e,5)$-adapted. Moreover, each
 character $\la$ in the tables is of  central $5$-defect.
\end{lem}

\begin{prop}   \label {prop:E8-5-defgrp}
 Suppose that $2\ne q \equiv1,2,3\pmod 5$. For any quasi-isolated $5'$-element
 $1\ne s\in \bG^{*F}$ the block distribution of $\cE(\bG^F,s)$ is as given in
 Tables~\ref{tab:quasi-E8-5} and~\ref{tab:quasi-E8-5b} for the respective
 congruences of $q\pmod5$. \par
 For each $5$-block $B$ corresponding to one of the cases in
 Table~\ref{tab:quasi-E8-5} or~\ref{tab:quasi-E8-5b} there is a defect group
 $P\leq N_{\bG^F}(\bL,\la)$ with the structure described in
 Theorem~\ref{thm:mainblocks}.

 In particular, $B$ has abelian defect groups precisely when the order
 $W_{\bG^F}(\bL,\la)$ is not divisible by $5$, in which case $Z(\bL)_5^F$ is
 a defect group of $B$.
\end{prop}

\begin{proof}
Again, each $e$-Harish-Chandra series in the tables is contained in
a unique $5$-block of $\bG^F$. In all numbered lines which are at the top
of the part of the table corresponding to a particular $s$ we conclude by
Proposition~\ref{prop:defect:re}(c).
In all cases, $Z^\circ(\bL)^F \cap[\bL, \bL]^F$ is a $5'$-group, so $D=A$,
and in all lines which are not at the top of the part of the table
corresponding to a particular $s$, $W_{\bG^F}(\bL,\la)$ is a $5'$-group. The
assertion
on the defect groups follows by Proposition~\ref{prop:defect:re}(f),(g).
We see that different numbered lines correspond to different blocks
by comparing orders of defect groups, or differentiating on the basis of
whether the defect groups are abelian or not. For the cases where one numbered
line corresponds to several cuspidal pairs, eg. cases~$5, 6,\ldots$, one notes
that the maximal Brauer pairs of the two blocks are not conjugate
(see the argument for the two blocks represented by line~16 of
Table~\ref{tab:quasi-E8-3}).
\end{proof}

\subsection{The group $E_8(2)$}   \label{subsec:E8(2)}
The general Mackey formula has not (yet) been proved for $E_8(2)$. Since
this group has three bad primes 2, 3, and~5, the 3-blocks for quasi-isolated
5-elements and the 5-blocks for quasi-isolated 3-elements are not covered by
previous results.

\begin{prop}
 The results on $3$-blocks and $5$-blocks of $E_8(q)$, $q>2$, stated in
 Sections~\ref{subsec:E8-3} and~\ref{subsec:E8-5} above continue to hold
 for $q=2$.
\end{prop}

\begin{proof}
Let first $\ell=3$, so we are in the situation of the Ennola dual of
Table~\ref{tab:quasi-E8-3}. Here, only quasi-isolated 5-elements need to be
considered, that is, the Ennola duals of lines~12--17 in that table. Now all
relevant centralisers $C_{\bL^*}(s)$ are of type $A$, so their cuspidal
characters $\lambda$ are uniform. In this case, the decomposition of
$\RLG(\lambda)$ is known by the results of Lusztig, and the Mackey formula
is not needed. We thus obtain the same Harish-Chandra series as in the
case $q>2$, and the results there continue to hold. \par
Similarly for $\ell=5$, since $q=2$, we are in the situation of
Table~\ref{tab:quasi-E8-5b} and we only need to consider Lusztig-series for
quasi-isolated 3-elements. Thus, only cases~39--44 in that table matter. But
note that again either $\lambda$ is uniform (cases~39,40,42,43) in which case
the Mackey formula is not needed for the determination of $\RLG(\lambda)$,
or $\bL=\bG$ and $\lambda$ is of 5-defect zero, so lies in a block of defect
zero.
\end{proof}

\begin{rem}   \label{rem:history}
This completes the parametrization of $\ell$-blocks of the finite quasi-simple
groups. Indeed, the $\ell$-blocks of the covering groups of alternating groups
were found by Brauer and Robinson, and by Cabanes and Humphreys
(see e.g \cite{Ol90}). The case of groups of
Lie type in their defining characteristic was solved a long time ago by
Humphreys; here, the non-trivial block are in bijection with irreducible
characters of the centre, and they all have full defect. \par
This leaves the case of groups of Lie type where $\ell$ is different from
the defining characteristic. The first general results on block distribution
for classical type groups were obtained in the landmark papers of Fong and
Srinivasan \cite{FS82,FS89}, which also introduced some of the fundamental
methods. For exceptional type groups, the case of unipotent blocks, that is,
blocks containing some unipotent character, was first considered by Schewe
\cite{Sch}; complete results for some groups of low rank were obtained by
Hiss, Deriziotis and Michler, Malle in \cite{Hi,DeMi,MaF}.
The parametrisation of all unipotent blocks for large primes $\ell$ was
obtained in \cite{BMM} in terms of $e$-Harish-Chandra series. Cabanes and
Enguehard \cite{CE99} determined all $\ell$-blocks whenever $\ell$ is a good
prime. Bonnaf\'e and Rouquier \cite{BR} showed that $\ell$-blocks parametrised
by semisimple elements of the dual group whose centraliser lies in a proper
Levi subgroup are Morita equivalent via Lusztig induction to unipotent blocks
of smaller groups.
The unipotent blocks for small $\ell$ and the quasi-isolated blocks of
classical groups were found by Enguehard \cite{En00,En08}.
\end{rem}

\subsection{Quasi-isolated blocks for $G_2$ and $\tw3D_4$}   
\label{subsec:G2+3D4}
For later use we also record the following easy observations on quasi-isolated
blocks for small exceptional type groups:

\begin{lem}   \label{lem:G2}
 Let $\bG^F=G_2(q)$ or $\bG^F=\tw3D_4(q)$, $p\ne \ell\in\{2,3\}$ and
 $s\in \bG^{*F}$ a quasi-isolated $\ell'$-element. Then for
 $e=e_\ell(q)$, the $e$-Harish-Chandra series in $\cE(\bG^F,s)$ satisfy
 Theorem~\ref{thm:BMM} and $\cE_\ell(\bG^F,s)$ is a single $\ell$-block. 
 Moreover in each numbered line in Table~\ref{tab:quasi-G2}, $\bL=\bT$ is a
 torus with $\bT=C_\bG(\bT_\ell^F)$.  For each $\ell$-block corresponding to one of the numbered lines  in the table
 there is a defect group $P\leq N_{\bG^F}(\bL,\la)$ with the structure
 described in Theorem~\ref{thm:mainblocks}.
 In particular, the defect groups are abelian  precisely in case~3 when  $\bG^F=G_2(q)$, and in case~1  when $\bG^F=\tw3D_4(q)$.
\end{lem}

\begin{proof}
In Table~\ref{tab:quasi-G2} we give the information
on the $1$-Harish-Chandra series in $\cE(\bG^F,s)$ with the same conventions
as earlier. The decomposition of $\RLG$ was determined by Lusztig, and from
that it is easy to check Theorem~\ref{thm:BMM} in this case. The situation
for $e=2$ is completely analogous. The assertion on the block and defect group
structure can be deduced as previously; it was also already obtained in
\cite{Hi} for $G_2(q)$ and in \cite{DeMi} for $\tw3D_4(q)$.
\end{proof}

\begin{table}[htbp]
\caption{Quasi-isolated blocks in $G_2(q)$ and $\tw3D_4(q)$}   \label{tab:quasi-G2}
\[\begin{array}{|l|r|r|l|lll|r|r|}
\hline
 \bG^F& \text{No.}& C_{\bG^*}(s)^F& (\ell,e)& \bL^F& \la& W_{\bG^F}(\bL,\la)\\
\hline\hline
G_2& 1& A_2(q)&       (2,1)& \Ph1^2& 1& A_2\\
   & 2& \tw2A_2(q)&   (2,1)& \Ph1\Ph2& 1& A_1\\
   &  &           &        & \bG^F& \phi_{21}& 1\\
\hline
   & 3& A_1(q)A_1(q)& (3,1)& \Ph1^2& 1& A_1\times A_1\\
\hline\hline
\tw3D_4& 1& A_1(q)A_1(q^3)& (3,1)& \Ph1^2\Ph3& 1& A_1\times A_1\\
\hline
\end{array}\]
\end{table}

This concludes and completes the proof of Theorem~\ref{thm:mainblocks} on the
parametrization of quasi-isolated blocks for exceptional type groups and bad
primes.

\section{Defect groups and Bonnaf\'e--Rouquier equivalences}   \label{sec:def}

The aim of this section is the proof of Theorem~\ref{thm:all-ab-intro} which
shows that abelian defect groups are preserved under Bonnaf\'e--Rouquier
Morita equivalences. Throughout, $\bG$ will denote a connected reductive
algebraic group over the algebraic closure of a finite field, and
$F:\bG\rightarrow\bG$ a Steinberg endomorphism.

\subsection{Miscellany} \label{subsec:mycomp}

We start by proving some auxiliary statements.

\begin{lem}   \label{lem:product}
 Let $\bG$ be connected reductive with derived subgroup of simply connected
 type, $\bL$ an $F$-stable
 Levi subgroup of $\bG$ and $\ell$ a prime. Let $\bG_1,\ldots,\bG_r$ be a set
 of representatives for the $F$-orbits on the set of simple components of
 $[\bG,\bG]$, and $\bL_i:=\bG_i\cap\bL$. Suppose that
 $$C_{\bG_i}(Z(\bL_i)_\ell^{F^{d_i}})=\bL_i\qquad\text{for } i=1,\ldots,r,$$
 where $d_i$ denotes the length of the $F$-orbit of $\bG_i$. Then
 $C_\bG(Z(\bL)_\ell^F)=\bL$.
\end{lem}

\begin{proof}
Let $\bH_1,\bH_2,\ldots,\bH_{d}$ denote an $F$-orbit on the set of simple
components of $[\bG,\bG]$, $\bH:=\bH_1\cdots\bH_d$, and set
$\bM_j:=\bL\cap\bH_j$ for $1\le j\le d$, $\bM:=\bM_1\cdots\bM_d=\bL\cap\bH$.
Then $\bM^F\cong\bM_1^{F^d}$ (see e.g. \cite[Ex.~30.2]{MT}), and similarly
$C_\bH(Z(\bM)_\ell^F)\cong C_{\bH_1}(Z(\bM_1)_\ell^{F^d})= \bM_1$ by assumption.
\par
Now $\bG'=[\bG,\bG]$ is the direct product of $F$-orbits as before, and
hence
$$C_{\bG'}(Z(\bL\cap\bG')_\ell^F)
  =C_{\bG_1}(Z(\bL_1)_\ell^{F^{d_1}})\cdots C_{\bG_r}(Z(\bL_r)_\ell^{F^{d_r}})
  =\bL_1\cdots\bL_r=\bL\cap\bG'.$$
Finally, $\bG=\bG'\bT$ for a central torus $\bT$, whence the claim follows.
\end{proof}

\begin{lem}   \label{lem:ZL}
 Suppose that $\bG $ has simply connected derived subgroup and $\ell $ is a
 good prime for $\bG$.  For any finite abelian $\ell$-subgroup $A$
 of $\bG$, $C_{\bG}(A)$ is a Levi subgroup of $\bG$.
\end{lem}

\begin{proof}
Let $A=\langle z_1,\ldots,z_r\rangle$ be a generating system for $A$.
Since $[\bG, \bG]$ is simply connected, $\bC:=C_{\bG}(z_1)$ is connected, and
it is a Levi subgroup of $\bG$ by \cite[Prop.~13.16]{CE} since $\ell$ is a
good prime for $\bG$. By \cite[Prop.~12.14]{MT}, $[\bC,\bC]$ is simply
connected. We may now replace $\bG$ by $\bC$ and apply induction to conclude.
\end{proof}

\begin{prop}   \label{prop:2selfcent}
 Assume that $\bG$ has simply connected derived subgroup over
 a field of odd characteristic. Let $\bT\le\bG$ be an $F$-stable maximal torus
 of $\bG$ containing a Sylow $e$-torus, where $e=e_2(q)$, and $A=\bT_2^F$
 the Sylow $2$-subgroup of $\bT^F$. Then:
 \begin{enumerate}[\rm(a)]
  \item $N_\bG(\bT)^F$ contains a Sylow $2$-subgroup of $\bG^F$.
  \item $N_{\bG^F}(\bT)/\bT^F$ acts faithfully on $A$.
  \item $C_\bG(A) = \bT$.
 \end{enumerate}
\end{prop}

\begin{proof}
Since $\bG=Z(\bG)[\bG,\bG]$, we may argue in $[\bG,\bG]$, which is a direct
product of simple groups, with $F$ permuting the factors. So after possible
extension of scalars we are reduced to $\bG$ being simple.

By \cite[Prop.~5.20]{MaH} for example, $N_\bG(\bT)^F$ contains a Sylow
2-subgroup of $\bG^F$. For $\bG$ of exceptional type, or for $\bG^F$ of type
$\tw3D_4$, the assertion in~(b) can be checked using \cite{MChev}. Otherwise
$\bT$ is the centralizer of a Sylow 1- or 2-torus and $W_\bG(\bT)^F$ is a
Coxeter group of type
$A_l$, $B_l$ or $D_l$, with $l$ suitable. Let's then write $W_l:=W_\bG(\bT)^F$.
The cases where $l\le4$ can again be checked by computer, so now assume
$l\ge5$. Then it is easy to see that all non-trivial normal subgroups $N$ of
$W_l$ have non-trivial intersection with its parabolic subgroup $W_{l-1}$,
and thus act non-trivially on $\bT^F$ by induction, except for
$N=\langle w_0\rangle$ generated by the longest element $w_0$ in types $B_l$
and $D_l$. But the longest element acts by inversion on $\bT^F$, hence also
non-trivially as $\bT^F$ contains elements of order~4. So $W_l$ acts
faithfully in all cases.  \par
For~(c) let $\bM:=\langle \bT^g\mid g\in\bG,\,A^g=A\rangle$ be generated by the
maximal tori of $\bG$ containing $A$. Then $\bM$ is connected (see e.g.\
\cite[Prop.~1.16]{MT}), $F$-stable,
$W_{\bG^F}(\bT)$-invariant, and $\bT\le\bM\le C_\bG(A)$. Let $\bX$ denote its
unipotent radical. If $[\bM/\bX,\bM/\bX]\ne1$ then we obtain a non-trivial
2-element in the Weyl group $N_{\bM^F}(\bT)$, centralizing $A$ but not lying in
$A$, contradicting~(b). Thus $\bM$ is solvable. Let $\bB$ denote a Borel
subgroup
of $\bG$ containing $\bM$, with unipotent radical $\bU$, so $\bB=\bU.\bT$ with
$\bX\le\bU$. Let $w_0\in W_{\bG^F}(\bT)$ be the longest element. If $u\in\bM$
is unipotent, then
$$u^{w_0}\in\bX^{w_0}\cap\bU^{w_0}=\bX\cap\bU^{w_0}\subseteq\bU\cap\bU^{w_0}=1$$
(see \cite[Cor.~11.18]{MT}), so $\bX=1$ and $\bM=\bT$. This show that
$C_\bG(A)\le N_\bG(A)\le N_\bG(\bT)$, but $W_{\bG^F}(\bT)$ acts faithfully on
$A$ by~(b), whence $C_\bG(A)=\bT$.
\end{proof}

\begin{lem}   \label{lem:2selfcentA}
 Assume that $\bG$ has simply connected derived subgroup with all simple
 factors of type $A$, over a field of odd characteristic. Let $\bT\le\bG$
 denote an $F$-stable maximal torus such that $N_\bG(\bT)^F$ contains a
 Sylow $2$-subgroup $P$ of $\bG^F$ and $C_\bG(\bT_2^F) = \bT$, and set
 $A=\bT^F_2$. Let $Z\le Z(\bG^F)$ be a central subgroup of order $2$. Then:
 \begin{enumerate}[\rm(a)]
  \item If $P$ centralises $A/Z$, then either $\bG$ is a torus and $P=A$ or
   the components of $[\bG, \bG]$ are of type $A_1$, form a single $F$-orbit
   and the index of $A$ in $P$ is $2$.
  \item Suppose that $Z=Z(P)$ and $P/Z$ is abelian. Then $P$ is quaternion of
   order $8$.
 \end{enumerate}
\end{lem}

\begin{proof}
For~(a), suppose that $\bG$ is not a torus.
Let $I$ be the set of $F$-orbits on the simple components of $[\bG, \bG]$,
and for each $i\in I$, let $\bH_i$ denote the product of the simple
components in $i$. So $[\bG, \bG]$ is a direct product of the $\bH_i$'s
and by the above, $\bT$ is a product
$$Z^\circ(\bG)(\prod_{i \in I} \bT_i),$$
where $\bT_i $ is an $F$-stable maximal torus of $\bH_i$ such that
$N_{\bH_i^F}(\bT_i)$ contains a Sylow $2$-subgroup, say $P_i$ of $\bH_i^F$,
$C_{\bH_i}((\bT_i)^F_2) = \bT_i$ and $P$ contains $\prod_{i\in I} P_i$.
Since $\bT$ is a maximal torus of $\bG$, $\bT^F$ covers
$\bG^F/[\bG,\bG]^F$ and hence $P= A (\prod_{i\in I}P_i)$.

Set $A_i = (\bT_i)^F_2 = A\cap \bT_i$. For each $i$, $P_i$ centralises
$A_iZ/Z$. We claim that $|I|=1$. Else, since $Z$ is cyclic and the product
of the $\bH_i$'s is direct, $\bH_i \cap Z = 1$ for some $i \in I$, whence
$P_i$ centralises $A_i \cong A_iZ/Z$. But this is impossible as the Sylow
$2$-subgroups of $\bH_i ^F$ are non-abelian. So, $|I|= 1$, and either
$[\bG, \bG]^F \cong \SL_n(q^d)$ or
$[\bG, \bG]^F \cong \SU_n(q^d)$. If $n \geq 3$, then
$P_i$ does not centralise $ A_i/U$, for a central subgroup $U$ of order~$2$
of $[\bG, \bG]^F$. So, $n=2$ and $[P_1: A_1]=2$.
Since $P= AP_1$, it follows that $[P:A]=2$.
\par
We prove~(c). Since $P_1$ is a Sylow $2$-subgroup of a special linear or
unitary group of degree $2$, $P_1$ is quaternion. Also, $P_1/Z$ is abelian,
hence $P_1$ has order $8$.

Since $Z=Z(P)$, $Z(\bG)^F_2 \leq Z$, and hence the natural surjection
of $\bG$ onto $\bG/Z(\bG)$ induces an injection of $P/Z$ into
$(\bG/Z(\bG))^F \cong \PGL_2(q^d)$ (or $\PGU_2(q^d)$).
The Sylow $2$-subgroups of $(\bG/Z(\bG))^F$ are non-abelian of order
$8$, hence $ |P/Z | \leq 4 $. So $P=P_1 $ is quaternion of order $8$.
\end{proof}

For $\bM$ an $F$-stable Levi subgroup  $\bM$ of $\bG$ and $s$ a semi-simple
element of $ {\bM^*}^F$, we will be interested in the condition
$C_{\bG^*}(s)\le\bM^*$. The following translates this into the corresponding
condition on $\bG$ and $\bM$.

\begin{lem}\label{lem:dualjordan}
 Let $\bM $ be an $F$-stable-Levi subgroup of $\bG$ and $s\in{\bM^*}^F$
 a semi-simple element. The following are equivalent.
 \begin{enumerate}[\rm(i)]
  \item $C_{\bG^*}(s) \le \bM^* $.
  \item For $i=1,2$, let $\bT_i\le\bM$ be $F$-stable maximal tori of $\bM$
   and $\theta_i\in\Irr(\bT_i^F)$ such that the $\bM$-geometric conjugacy
   class of $(\bT_i,\theta_i)$ both correspond via duality to the class of
   $s$. Then any $g \in \bG$ which geometrically conjugates $(\bT_1,\theta_1)$
   to $(\bT_2,\theta_2)$ is in $\bM$.
 \end{enumerate}
\end{lem}

\begin{proof}
The implication (i) $\Rightarrow$ (ii) is in \cite[Lemma 13.26(i)]{DM91} and
the reverse implication follows from reversing the
argument of \cite[Lemma 13.26(i)]{DM91}, and noting that
for any $F$-stable torus $\bT^*$ of $\bM^*$ containing $s$, if
$N_{\bG^*}(\bT^*)\cap C_{\bG^*}(s)\subseteq \bM^*$, then
$C_{\bG^*}(s) \subseteq \bM^* $
(since $N_{\bG^*}(\bT^*) \cap C_{\bG^*}^\circ(s)$ is a Levi-subgroup of
$C_{\bG^*}^\circ(s)$).
\end{proof}

\begin{lem}   \label{lem:centraliserjordan}
 With the notation of Lemma~\ref{lem:dualjordan}, suppose that
 $C_{\bG^*}(s) \leq \bM^*$.
 Let $\bL$ be an $F$-stable Levi subgroup of $\bM$ and $\la\in\Irr(\bL^F)$
 such that all constituents of $\RLM(\la)$ lie in $\cE(\bM^F, s)$.
 Let $z\in Z(\bL)^F$ and set $\bG_0 = C_\bG^\circ(z)$ and
 $\bM_0 = C_\bM^\circ (z)$. Let $s'\in {\bM_0^*}^F$ be a semisimple element
 such that all constituents of $R_\bL^{\bM_0}(\la)$ lie in $\cE(\bM_0^F, s')$.
 Then $C_{\bG_0^*}(s') \leq \bM_0^*$.
\end{lem}

\begin{proof}
Let $\bT$ be an $F$-stable maximal torus of $\bL$ and
$\theta$ an irreducible character of $\bT^F$ such that $\la$ is
a constituent of $R_{\bT}^\bL(\theta)$. Since Lusztig induction preserves
Lusztig series, the $\bM$-geometric conjugacy class of
$(\bT, \theta)$ corresponds to the $\bM^*$-class of $s$ and
the $\bM_0$-geometric conjugacy class of
$(\bT, \theta)$ corresponds to the $\bG_0^*$-class of $s'$.
The assertion follows from Lemma \ref{lem:dualjordan} --- here, note 
that $\bM_0=\bG_0 \cap \bM $.
\end{proof}

\subsection{Bonnaf\'e--Rouquier correspondents}  \label{subsec:abdef}
In this subsection, $\bM$ will denote an $F$-stable Levi subgroup of $\bG$
and $s\in {\bM^*}^F$ a semi-simple $\ell'$-element. We let
$c$ be an $\ell$-block of $\bM^F$ contained in $\cE_\ell(\bM^F, s)$ and
let $b$ be an $\ell$-block of $\bG^F$ contained in $\cE_\ell(\bG^F, s)$.

Recall that if $C_{\bG^*}(s) \le \bM^*$, then for any semisimple
$\ell$-element $t\in C_{\bG^*}(s)$, $\epsilon_\bM\epsilon_\bG \RMG$ induces
a bijection between $\cE(\bM^F,st)$ and $\cE(\bG^F,st)$. This bijection is
independent of choice of parabolic containing $\bM$ (see
\cite[Rem.~13.28]{DM91}) and it induces a bijection between $\ell$-blocks
in $\cE_\ell(\bM^F, s)$ and in $\cE_\ell(\bG^F, s)$. Further,
by \cite[Thm.~B']{BR} there is a Morita equivalence over $\cO$ between pairs
of corresponding blocks which induces the bijection
$\chi\mapsto\epsilon_\bM\epsilon_\bG \RMG(\chi)$ on ordinary irreducible
characters.

\begin{defn}  \label{dfn:bon-rou-corr}
 We say that blocks $b$ and $c$ are \emph{Bonnaf\'e--Rouquier correspondents}
 if $C_{\bG^*}(s)\subseteq \bM^*$ and for some (and hence any)
 $\chi\in\cE(\bM^F, s) \cap \Irr(c)$ we have
 $\epsilon_\bM\epsilon_\bG \RMG(\chi)\in \Irr(b)$.
\end{defn}

\begin{lem}  \label{lem:decojordan}
 Suppose that $b$ and $c$ are Bonnaf\'e--Rouquier correspondents. Let $\bL$ be
 an $F$-stable Levi subgroup of $\bM$, let $\la\in\cE(\bL^F,\ell')$, and
 suppose that every constituent of $\RLM(\la)$ is contained in $\Irr(c)$.
 Then every constituent of $\RLG(\la)$ is contained in $\Irr(b)$, and for
 every $\chi_0 \in \Irr(c)$, $\langle \chi_0, \RLM(\la) \rangle \ne 0$ if and
 only if $\langle \chi, \RLG(\la) \rangle \ne 0$, where
 $\chi = \epsilon_\bG\epsilon_\bM \RMG(\chi_0)$ denotes the corresponding
 element of $\Irr(b)$. Further, 
 $\langle{}^*\RLM (\chi_0), d^{1,\bM^F} (\la) \rangle \ne 0 $ if and only if 
 $\langle{}^*\RLG (\chi), d^{1,\bG^F}(\la)\rangle \ne 0 $.
\end{lem}

\begin{proof}
The first two assertions follow from transitivity of Lusztig induction
(see Theorem~\ref{thm:RLG}(a)). For the third claim, 
let $\chi_0 \in \Irr(c)$ and  write
$$d^{1,\bM}(\chi_0) = \alpha_1 \phi_0^1+\cdots+ \alpha_t\phi_0^r$$
with $\alpha_i$ non-zero for all $i$ and the $\phi^i_0$ pairwise
distinct irreducible characters in $c$. So,
$$d^{1,\bG}(\chi) = \alpha_1 \phi^1+\cdots+ \alpha_r\phi^r.$$
Then,
$$\langle d^{1,\bM^F}(\chi_0), \RLM(\la)\rangle \ne 0 $$
if and only if
$\phi_0^j = \chi_0^i$ for some $i, j$, if and only if
$\phi^j = \chi^i$ for some $i, j$, if and only if
$\langle d^{1,\bG^F}(\chi),\RLG(\la)\rangle \ne 0$. Now  the result 
follows since
$$\langle{}^*\RLM (\chi_0), d^{1,\bM^F} (\la)\rangle=
  \langle d^{1,\bM^F} (\chi_0), \RLM (\la)\rangle$$
and
$$\langle{}^*\RLG (\chi), d^{1,\bG^F} (\la)\rangle=
  \langle  d^{1,\bG^F} (\chi), \RLG (\la)\rangle.$$
\end{proof}

\begin{prop}  \label{prop:Cabanesjordancrux}
 Suppose that $b$ and $c$ are Bonnaf\'e--Rouquier correspondents. Let $\bL$
 be an $F$-stable Levi subgroup of $\bM$ and let $\la\in\cE(\bL^F,\ell')$ be
 such that $\RLM (\la)$ has a constituent in the block $c$.
 Let $A=\langle z_1,\ldots,z_m\rangle$ be a subgroup of $Z(\bL)_\ell^F$
 and set $\bM_i= C_\bM^\circ(z_1,\ldots,z_i)$,
 $\bG_i= C_\bG^\circ(z_1,\ldots,z_i)$, $1\leq i\leq m$, $\bM_0= \bM$,
 $\bG_0= \bG$. Suppose the following:
 \begin{enumerate}[\rm(1)]
  \item  For any $i$, $1\leq i \leq m $, and any
   character $\chi\in\Irr(\bM_i^F)$ with
   $\langle R_\bL^{\bM_i}(\la),\chi\rangle_{\bM_i^F}\ne0$ we have
   $\langle d^{1,\bL^F}(\la),{}^*R_\bL^{\bM_i}(\chi)\rangle_{\bL^F}\ne0$.
  \item The irreducible constituents of $R_\bL^{\bM_m}(\la)$
   lie in a single block.
 \end{enumerate}
 Then, {\rm(1)} and {\rm(2)} hold with $\bM_i$ replaced by $\bG_i$ for
 $1\leq i\leq m$. Consequently, for any $i$ there exists a unique block,
 say $b_i$ of $\bG_i^F$ containing the constituents of
 $R_\bL^{\bG_i}(\la)$, and a unique block say $c_i$ of $\bM_i^F$
 containing all constituents of $R_\bL^{\bM_i}(\la)$. Further, the
 following holds.
 \begin{enumerate}[\rm (a)]
  \item $b_0 = b$, $c_0 = c$, and for all $i$, $0\leq i \leq m$,
   $b_i$  and  $c_i$ are Bonnaf\'e--Rouquier correspondents.
  \item  Let $\tilde c_m$ denote the unique $\ell$-block of $C_{\bM^F}(A)$
   covering $c_m$ and $\tilde b_m$ denote the unique $\ell$-block of
   $C_{\bG^F}(A)$ covering $b_m$. Then $(A,\tilde c_m)$ is a $c$-Brauer pair
   and $(A,\tilde b_m)$ is a $b$-Brauer pair. Moreover,
  \item  $N_{\bM^F} (A, \tilde c_m) \leq  N_{\bG^F}(A, \tilde b_m)$,
   and hence
   $$N_{\bM^F}(A,\tilde c_m)/C_{\bM^F}(A) \leq
     N_{\bG^F}(A,\tilde b_m)/C_{\bG^F}(A)$$
   under the inclusion of $\bM$ in $\bG$.
 \end{enumerate}
\end{prop}

\begin{proof}
By  Proposition \ref{prop:Cabanescrux}, for all $i$, $0 \leq i \leq m$,
the irreducible constituents of $R_\bL^{\bM_i}(\la)$ lie in a unique
block, $c_{i}$ of $\bM_i^F$. Further, by Lemma~\ref{lem:centraliserjordan},
$c_{i}$ has a Bonnaf\'e--Rouquier correspondent,
say  $b_i$  in  $\bG_i^F$. The first assertion  and (a)  holds  by
Lemma~\ref{lem:decojordan},  applied to  $c_i$ and $b_i$, $0\leq i \leq m$.
(b) follows from Proposition \ref{prop:Cabanescrux} applied to both $\bM$
and $\bG$.
\par
For (c), let $g \in N_{\bM^F}(A, \tilde c_m)$.
Then, $\,^g c_m$ is covered by $\tilde c_m$, whence $\,^g c_m =\,^h c_m$
for some $h \in C_{\bM^F}(A)$. Let $\chi_0 \in \Irr(c_m)$.
Then $\,^{h^{-1}g}\chi_0$ is an irreducible character of $c_m$. Since
$c_m$ and $b_m$ are Bonnaf\'e--Rouquier correspondents,
$R_{\bM_m}^{\bG_m}(\,^{h^{-1}g}\chi_0)$
and $R_{\bM_m}^{\bG_m}(\chi_0)$ are irreducible characters in $b_m$.  So,
noting the  independence  from  the  choice of parabolic  subgroup of $\bG_m$
containing $\bM_m$ as pointed out before Definition~\ref{dfn:bon-rou-corr},
$$R_{\bM_m}^{\bG_m}(\,^{h^{-1}g}\chi_0)
  = \,^{h^{-1}g} R_{\bM_m}^{\bG_m} (\chi_0), $$
and $\,^{h^{-1}g} R_{\bM_m}^{\bG_m} (\chi_0)$ is in $\,^{h^{-1}g}\,b_m$.
Hence, $\,^g b_m =\,^h b_m$ and it follows that $\,^g\tilde b_m = \tilde b_m$.
\end{proof}

\subsection{Good pairs}   \label{subsec:goodpairs}
\begin{defn}   \label{dfn:goodpair}
 Let $\bG$ be connected reductive with Steinberg endomorphism $F:\bG\to\bG$.
 Let $\ell$ be a prime and let $b$ be an $\ell$-block of $\bG^F$.
 A pair $(\bL, \la)$ consisting of an $F$-stable Levi subgroup of $\bG$ and
 $\lambda\in\cE(\bL^F,\ell')$ is called a \emph{good pair for $b$} if the
 following holds:
 \begin{enumerate}
  \item[(1)] $\bL= C_\bG(Z(\bL)_\ell^F)$,
  \item[(2)] $\la$ is of quasi-central $\ell$-defect,
  \item[(3)] $(Z(\bL)_\ell^F, b_{\bL^F}(\la))$ is a $b$-Brauer pair, and
  \item[(4)] there is a maximal $\bG^F$-Brauer pair $(P, f)$ such that
   $(Z(\bL)_\ell^F, b_{\bL^F} (\la)) \unlhd (P, f)$.
 \end{enumerate}
\end{defn}

Note that when $(\bL,\la)$ satisfies (1)--(3)
then by Propositions~\ref{prop:quasicentraldefpro}, \ref{prop:decomp}(4),
\ref{prop:Cabanessingle}, and~\ref{prop:Cabanescrux} all irreducible
constituents of $\RLG(\la)$ lie in $b$.

The notion of good pairs is related to that of $e$-cuspidal pairs in
that many $\ell$-blocks of $\bG^F$ have good pairs which are also
$e$-cuspidal pairs of $\bG$ where $e=e_{\ell}(q)$
(see Theorem~\ref{thm:goodpairs} below). However, the two notions are not
identical, and  it will be easier to track the structure of defect groups
through Bonnaf\'e--Rouquier Morita equivalences using good pairs
(see Proposition~\ref{prop:jordan-nonab} below).

\begin{lem}   \label{lem:quasidefext2}
 Let $\tbG$ be connected reductive with Steinberg endomorphism
 $F:\tbG\to\tbG$, containing $\bG$ as an $F$-stable closed subgroup with
 $[\tbG, \tbG] \le\bG$ and let $\tbZ=Z^{\circ}(\tbG)$.
 Let $b$ be a block of $\bG^F$ and $\tilde b$ a block of $\tbG^F$ covering $b$.
 Let $\bL$ be an $F$-stable Levi subgroup of $\bG$ and
 $\tbL= \tbZ\bL$, a Levi subgroup of $\tbG$.
 Set $A= Z(\bL)_\ell^F$ and $\tilde A =Z (\tbL)_\ell^F$.
 Suppose that $\la \in \cE(\bL^F, \ell')$ is such that $(\bL,\la)$ is a
 good pair for $b$, and let $(P, f)$ be a maximal $b$-Brauer pair such that
 $(A, b_{\bL^F}(\la)) \unlhd (P,f)$. Then:
 \begin{enumerate}[\rm(a)]
  \item There exists $\tilde \la \in \cE(\tbL^F, \ell')$ covering
   $\la$ and a maximal $\tilde b$-Brauer pair $(\tilde P, d)$ such that
   $(\tbL, \tilde \la)$ is a good pair for $\tilde b$,
   $(\tilde A, b_{\tbL^F}(\tilde \la)) \unlhd (\tilde P, d)$ and
   $P\leq \tilde P \leq N_{\tilde\bG^F} (A, b_{\tbL^F}(\la))$.
  \item Further, if $b$ and $\tilde b$ are unipotent and $(\bL,\la)$ is
   an $e$-cuspidal pair for $\bG$, then $(\tbL,\tilde\la)$ is a unipotent
   $e$-cuspidal pair of $\tbG$.
  \item Suppose that $\tbZ\cap \bG$ contains no non-trivial
   $\ell$-element. If $\la $ is of central $\ell$-defect, then so is
   $\tilde \la $.
 \end{enumerate}
\end{lem}

\begin{proof} 
We have $\tbL \leq C_{\tbG}(\tilde A)\leq C_{\tbG}(A) =\tbZ\bL=\tbL$,
hence $\tbL= C_{\tbG}(A) = C_{\tbG}(\tilde A)$. Also, note that the inclusion
of $\tbL$ in $\tbG$ induces an isomorphism between $\tbL^F/\bL^F$
and $\tbG^F/\bG^F$.

By Lemma~\ref{lem:local-cliff}, there exists an $\ell$-block $b'$ of $\tbG^F$
covering $b$ and $b'$-Brauer pairs $(A,v)$ and $(\tilde P, d)$ such that
$(A, v) \unlhd (\tilde P, d)$, $\tilde P \leq N_{\tbG^F}(A, b_{\bL^F}(\la))$,
$v$ covers $b_{\bL^F}(\la)$, $\tilde P \cap \bG^F= P$ and $\tilde P/P$ is
isomorphic to a Sylow $\ell$-subgroup of
$N_{\tbG^F}(A, b_{\bL^F}(\la))/N_{ \bG^F}(A, b_{\bL^F}(\la))$.
Since $\tbG^F/\bG^F$ is abelian, by Lemma~\ref{lem:cliff-ab},
$\tilde b =\theta \otimes b'$, for some linear character $\theta$ of
$\tbG^F/\bG^F $. So, $(A,\theta\otimes v)$ and
$(\tilde P,\theta\otimes d)$ are $\tilde b$-Brauer pairs and since $\theta $
contains $\bL^F$ in its kernel, $\theta \otimes v$ covers
$\theta \otimes b_{\bL^F}(\la) = b_{\bL^F}(\la)$.
Thus, replacing  $\theta\otimes d$ by $d$ and $\theta \otimes v$ by $v$,
we may assume that $b'=\tilde b$.
Since $\tilde A$ is central in $\tbL^F$, we also get that $(\tilde A, v)$
is a $\tilde b$-Brauer pair and $(\tilde A,v)\unlhd (\tilde P, d)$.

We claim that $\tilde P$ is a defect group of $\tilde b $. Indeed, since $P$
is a defect group of $b$, and $\tilde b$ covers $b$, it suffices to prove that
$|\tilde P : P| \geq |\tilde\bG^F: \bG^F| _\ell$. But,
$$\begin{aligned}
  |\tilde P :P|
  &=|N_{\tilde\bG^F}(A,b_{\bL^F}(\la)):N_{\bG^F}(A,b_{\bL^F}(\la))|_\ell\\
  &\geq |N_{\tbL^F}(A,b_{\bL^F}(\la)):\bL^F|_\ell
   =|\tbL^F :\bL^F|_\ell= |\tbG^F: \bG^F| _\ell.
\end{aligned}$$
Here, for the second equality, note that the index of
$N_{\tbL^F}(A,b_{\bL^F}(\la))$ in $\tbL^F$ is prime to $\ell$.
This proves the claim.

Now let $\tilde\la \in \cE(\tbL^F,\ell')\cap\Irr(v)$. Then
$\tilde \la$ covers an element of $\cE(\bL^F,\ell')\cap \Irr( b_{\bL^F}(\la))$
(see \cite[Prop.~11.7(b)]{B06}). Since
$\Irr( b_{\bL^F}(\la))\cap \cE(\bL^F,\ell')=\{\la\}$
by Proposition~\ref{prop:quasicentraldefpro}, $\tilde \la$ covers $\la$.
Further, $\tilde \la$ and $\la$ cover a common character of
$[\tbL, \tbL]^F= [\bL, \bL]^F$, so $\tilde \la$ is of
quasi-central $\ell$-defect.  This proves (a).

(b) follows from (a) and the fact that restriction induces a bijection
between $\cE(\tbG^F, 1)$ and $\cE( \bG^F, 1)$ which commutes with
Lusztig induction and restriction.

It remains to prove (c). Since $ \bL \leq \bG  $,  $\tbZ\cap \bL $
contains no non-trivial $\ell$-element.
So, any Sylow $\ell$-subgroup of $\tbL^F $ is a direct product of the
Sylow $\ell$-subgroup  of $\tbZ^F$ and a Sylow $\ell$-subgroup of
$\bL^F$, and similarly, the Sylow $\ell$-subgroup  of $ Z(\tbL^F) $
is a direct product of the Sylow $\ell$-subgroup  of $\tbZ^F$ and the
Sylow $\ell$-subgroup of $ Z(\bL)^F$. The result follows as
$|\tilde \lambda(1)|_\ell =|\lambda(1)|_\ell $.
\end{proof}

The next result shows in particular that all quasi-isolated blocks
have good pairs.

\begin{thm}   \label{thm:goodpairs}
 Suppose that $[\bG, \bG]$ is simply connected. Let $b$ be an $\ell$-block   of $\bG^F$ with $\Irr(b) \subseteq \cE_\ell(\bG^F,s)$
and let $e =e_\ell(q)$.
 \begin{enumerate}[\rm(a)]
  \item  Suppose that $\ell$ is odd, good for $\bG$ and $\ell\ne3 $ if
   $\tw3D_4(q)$ is involved in $\bG^F$. Then $b$ has a good pair $(\bL,\la)$
   and a maximal $b$-Brauer pair $(P,f)$ with
   $(Z(\bL)_\ell^F, b_{\bL^F}(\la)) \unlhd (P,f)$, such that
   $\la$ is of central $\ell$-defect, the extension
   $$ 1 \to Z(\bL)_\ell^F \to P \to P/ Z(\bL)_\ell^F \to 1$$
   is split and $Z(\bL)_\ell^F$ is the unique maximal normal abelian
   subgroup of $P$. If $s$ is central, then $(\bL,\la)$ can be chosen to be
   $e$-cuspidal, and in that case $(\bL,\la)$ is unique up to
   $\bG^F$-conjugacy.
  \item  Suppose that $\ell=2$ and all components of $\bG$ are of type $A$.
   Then $b$ has a good pair $(\bL, \la)$ and a maximal $b$-Brauer pair
   $(P, f)$ with $(Z(\bL)_2^F, b_{\bL^F}(\la)) \unlhd (P,f)$, such that
   $Z(\bL)^F_2= \bT_2^F$, $\la$ is of central $2$-defect and
   $\Aut_P(\bT_2^F) = \Aut_{P'}(\bT_2^F)$. Here, $\bT$ is an $F$-stable
   maximal torus of $\bG$ such that $C_{\bG_1}(\bT^F_2) =\bT$ for a Levi
   subgroup $\bG_1$ of $\bG$ in duality with $C_{\bG^*}^\circ(s)$, and such
   that $N_{\bG_1^F}(\bT)$ contains a Sylow $2$-subgroup $P'$ of $\bG_1^F$.
  \item  Suppose that $\ell=2$, $\bG$ is simple, of classical type different
   from type $A$ and $s$ is quasi-isolated in $\bG^*$.
   Then $b$ has a good pair $(\bL, \la)$, where $ \bL $ is an $F$-stable
   maximal torus of $\bG$ containing a Sylow $e$-torus of $\bG$.
  \item  Suppose that $s$ is quasi-isolated and either $\bG$ is simple  of
   exceptional type and $\ell$ is bad for $\bG$, or $\bG$ is of rational
   type $\tw3D_4$ and $\ell=2,3$. Then $b$ has a good pair
   $(\bL, \la)$ which is $e$-split cuspidal. Further, if $\ell$ is odd,
   $\la$ is of central $\ell$-defect.
 \end{enumerate}
 In particular, if $b$ is quasi-isolated then $b$ has a good pair.
\end{thm}

\begin{proof}
Suppose the assumptions of~(a) hold. Then there exists a pair
$(\bL,\la)$ (denoted $(\bM,\zeta_M)$ in \cite{CE99}) such that
$C_\bG^\circ (Z(\bL)_\ell^F)= \bL$ (\cite[Lemma~4.8]{CE99}),
$\la$ is of central $\ell$-defect (\cite[Lemma~4.11]{CE99}), and
letting $\hat u$ denote the unique $\ell$-block of $C_{\bG^F}(Z(\bL)_\ell^F)$
covering $b_{\bL^F}(\la)$, $(Z(\bL)_\ell^F, \hat u)$ is a $b$-Brauer pair
(\cite[Lemma~4.10]{CE99}). Further, there exists a maximal $\bG^F$-Brauer pair
$(P, f)$ with $(Z(\bL)_\ell^F,\hat u)\unlhd(P,f)$ and
such that the short exact sequence above has the required properties
(\cite[Lemma~4.16]{CE99}).
Thus, in order to prove that $(\bL,\la)$ is a good pair for $b$,
we need only show that $\bL=C_\bG(Z(\bL)_\ell^F)$. But since
$C_\bG^\circ (Z(\bL)_\ell^F)= \bL$  this follows from Lemma~\ref{lem:ZL}.
The final assertion follows from \cite[Thm.~1.1, Lemma~4.5]{CE94}.

Suppose that the assumptions of~(b) hold and note that $\bT$ as in the
statement exists by Proposition~\ref{lem:2selfcentA} applied to $\bG_1$.
Let $\bG\rightarrow\tbG$ be a regular embedding. For $\tbG_1$ an $F$-stable
Levi subgroup of $\tilde\bG$ with $\tbG_1 \cap \bG= \bG_1$ we let
$\tilde \bT$ be an $F$-stable maximal torus of $\tbG_1$ with
$\tilde\bT \cap\bG_1 =\bT$. Set $A=\bT_2^F$ and $\tilde A= \tilde\bT_2^F$.
By Lemma~\ref{lem:ZL}, $ \tbL:= C_{\tbG} (A)$ is a Levi subgroup of
$\tbG$ and $\bL:= C_\bG (A)$ is a Levi subgroup of $\bG$.

Let $\tilde s\in\tilde\bG^{*F}$ be an element of odd order lifting $s$ such
that $C_{\tbG^*}(s)= \tbG_1^*$, where $\tbG_1^*$ is the dual
of $\tbG_1$ in $\tbG^*$ and let $\theta$ be the linear character
of $\tbG_1^F$ in duality with $\tilde s$ (see \cite[Prop.~13.30]{DM91}).
By \cite[Prop.~1.5]{En08}, $\cE_2(\tilde\bG_1^F, \tilde s)$ is a single
$2$-block, say $\tilde c$ and $\cE_2(\tilde\bG^F, \tilde s)$ is a single
$2$-block, say $\tilde b$. In particular, $\tilde b$ covers $b$.
Moreover, $R_{\tbG_1}^{\tbG}$ induces a Morita equivalence between
$\tilde c$ and $\tilde b$.
Since $C_{\bG_1}(A) =\bT$, $C_{\tbG_1}(A) = \tilde \bT$, and hence
by Proposition~\ref{prop:Cabanesjordancrux}, applied with
$\bM= \tbG_1$, $\bG=\tbG $, and $\bL =\tilde \bT$,
$R_{\tilde\bT}^{\tbL} (\theta)$ is (up to sign) an
irreducible character, say $\tilde \chi$ of $\tbL^F =C_{\tbG}(A)^F$
and $(A, b_{\tbL^F}(\tilde \chi))$ is a $\tilde b$-Brauer pair.

Let $P'\leq N_{\bG_1}(\bT)$ be a Sylow $2$-subgroup of $\bG_1^F$. Then
$P'\leq N_{\tbG_1}(A,\theta)$, so by Proposition~\ref{prop:Cabanesjordancrux},
$P'\leq N_{\tbG} (A,b_{\tbL^F}(\tilde \chi))$. In particular, $P'$ acts
on the blocks of $\bL^F$ covered by $ b_{\tbL^F}(\tilde \chi) $.
There is an odd number of such blocks, so there exists a block $f$ of $\bL^F$
covered by $b_{\tbL^F}(\tilde \chi)$ which is
$P'$-stable. Let $\chi'\in\Irr(f) \cap \cE (\bL^F, \ell')$
be covered by $\tilde \chi$ and let $b'$ be the block of $\bG^F$ such that
$(A, f)$ is a $b'$-Brauer pair. Then $b'$ is covered by $\tilde b$.
Since $\pm R_{\tilde \bT}^{\tbL}(\theta)\in\Irr(\tbL^F,\tilde s)$,
and $\tilde s$ is an odd order element,
$$\chi'(1)_{2} =\tilde \chi (1)_{2} =
  \frac{|\tbL^F|_{2}}{|\tilde \bT^F|_2 } \geq
  \frac{| \bL^F|_{2}}{|\bT^F|_2 } = \frac{| \bL^F|_{2}}{|A|}.$$
Since $A$ is central in $\bL^F$, $A \leq \ker(\chi')$. From the above
displayed equation it follows that $A = Z(\bL)_2^F$ and $ \chi'$ is of
central $2$-defect. Let $(P, d)$ be a $b'$-Brauer pair, maximal with respect
to $(A, f) \unlhd (P, d)$. Then $P\cap\bL^F = A$ and by
Lemma~\ref{lem:local-cliff}(a), $P/A $ is a Sylow $\ell$-subgroup of
$N_{\bG^F}(A, f) /\bL^F$. Since
$P'\leq N_{\bG^F}(A, f)$ and $P'\cap C_{\bG^F}(A) =A$,
it follows that $|P|\geq |P'|$. On the other hand, since $b'$
being covered by $b$ means that $\Irr(b')\subseteq \cE_\ell(\bG^F, s)$,
by \cite[Prop.~1.5]{En08}, any Sylow $2$-subgroup of $\bG_1^F$
is a defect group of $b'$. Hence, $(P, d)$ is a maximal $b'$-pair.
Since $\bL$ is a Levi subgroup of $\bG$, $(A,\chi')$ is a good pair for $b'$.
Now $b$ and $b'$ are both covered by $\tilde b $ hence replacing
$(\bL, \chi')$ by a suitable $\tilde\bG^F$-conjugate gives the desired
result.

Now suppose that the assumptions of~(c) hold. So $\bG$ is
simple of type $B$, $C$ or $D$.
Then $s=1$ is the only odd order quasi-isolated element of $\bG^*$.
By \cite[Prop.~1.5]{En08} $\bG^F$ has a unique unipotent 2-block, hence
by Proposition~\ref{prop:2selfcent}
$(\bL, 1)$ is a good pair for $b$ for any $F$-stable maximal torus $\bL$ of
$\bG$ containing a Sylow $e$-torus of $\bG$.

Suppose the assumptions of~(d) hold. If $s$ is non-central in $\bG^*$, then
the result follows from
Theorem \ref{thm:mainblocks} and its proof. If $s$ is central in $\bG^*$,
then~(d) follows from \cite{En00}. Note that Enguehard does not
state the equality $\bL= C_\bG (Z(\bL)_\ell^F)$ for all unipotent
$e$-cuspidal pairs $(\bL, \la)$   but this can be  checked --- the Levis 
occurring for central quasi-isolated elements  also occur in  our tables, 
except for the 1-split Levis of type $D_4$ in $E_6$, and of type $E_6$ in
$E_7$ and their Ennola duals, and these cases can be easily checked also.

Finally, suppose that $b$ is a quasi-isolated block of $\bG^F$.
Then any block of $[\bG,\bG]^F$ covered by $b$ is quasi-isolated, so by
Lemma~\ref{lem:quasidefext2}, we may assume that $\bG=[\bG,\bG]$. Since
$\bG$ is a direct product of simple simply connected groups and the
component of $b$ in the fixed points of each $F$-orbit is quasi-isolated,
by  Lemma~\ref{lem:product} we may assume that $\bG$ is simple. Now the result
follows from parts~(a)--(d).
\end{proof}

\begin{prop}   \label{prop:jordan-nonab}
 Suppose that $b$ and $c$ are Bonnaf\'e--Rouquier correspondents and that
 $c$ has a good pair $(\bL, \la)$. Let $A= Z(\bL)_\ell^F$, let
 $u=b_{\bL^F}(\la)$ and let $(P,f)$ be a maximal $c$-Brauer pair with
 $(A, u) \unlhd (P,f)$. Then:
 \begin{enumerate}[\rm(a)]
  \item Let $v$ be the $\ell$-block of $C_\bG^\circ(A)^F$ containing
   the constituents of $R_\bL^{C_\bG^\circ(A)}(\la)$ and let $\tilde v$
   be the block of $C_\bG(A)^F$ covering $v$.
   Then there is a maximal $\bG^F$-Brauer pair $(Q, d)$ such that
   $(A,\tilde v)\unlhd (Q,d)$, $C_Q(A)\cong C_P(A)$ and $\Aut_Q(A)=\Aut_P(A)$.
  \item $b$ has abelian defect groups if and only if $c$ has abelian
   defect groups. If this is the case, then the defect groups of $b$ and $c$
   are isomorphic.
  \item Let $Z$ be a central $\ell$-subgroup of $\bG^F$ and let $\bar c$
   (respectively $\bar b$) be the image of $c$ (respectively $b$)
   in $\bM^F/Z$ (respectively $\bG^F/Z$). If either $\bar b$ or $\bar c$ has
   abelian defect groups, then $P$ centralises $A/Z$ and $Q$ centralises $A/Z$.
  \item If $\la$ is of central $\ell$-defect and $c$ has abelian defect groups,
   then $A$ is a defect group of both $b$ and $c$.
 \end{enumerate}
\end{prop}

\begin{proof}
Let $U$ be a subgroup of $A$. Since $C_\bM(A) = \bL$,
$C_{C_\bM(U)}(A) = \bL$. So, since $\la$ is of quasi-central $\ell$-defect,
by Proposition~\ref{prop:decomp}, the conditions of
Proposition~\ref{prop:Cabanesjordancrux} hold for any choice of generators
$\langle z_1,\ldots,z_m \rangle $ of $A$ and the statement of~(a)
makes sense. In particular, $u$ and $v$ are Bonnaf\'e--Rouquier correspondents.

Let $(A, \tilde v) \unlhd (Q,d)$ where $Q$ is maximal with respect to the
property that $Q\leq N_{\bG^F}(A, \tilde v)$. Then $QC_{\bG^F}(A)/C_{\bG^F}(A)$
is a Sylow $\ell$-subgroup of $N_{\bG^F}(A, \tilde v)/C_{\bG^F}(A)$
(see Lemma~\ref{lem:local-cliff}). So by
Proposition~\ref{prop:Cabanesjordancrux}(c), and by replacing if necessary
$(Q, d)$ by an $N_{\bG^F}(A, \tilde v)$-conjugate,
$QC_{\bG^F}(A)/C_{\bG^F}(A)$ contains $PC_{\bM^F}(A)/C_{\bM^F}(A)$.
In particular, $P/C_P(A)$ is isomorphic to a subgroup of $Q/C_Q(A)$.

Now $C_P(A)$ is a defect group of the block $u$,
$u$ is nilpotent and is Morita equivalent to $v$ (over ${\mathcal O}$).
By a result of Puig a Morita equivalence over $\cO$ between a nilpotent block
and a block preserves nilpotency and isomorphism type of defect groups
(see \cite[Thm.~8.4 and Cor.~7.3]{Pu}), so $v$ is nilpotent and a defect group
of $v$ is isomorphic to $C_P(A)$. Since $C_Q(A)$ contains a defect group of
$\tilde v$, $C_Q(A)$ contains a defect group of $v$, and hence
$|C_Q(A)| \geq |C_P(A)|$. We have shown above that $P/C_P(A)$ is isomorphic
to a subgroup of $Q/C_Q(A)$, hence $|Q|\geq|P|$. On the other hand,
$|Q|\leq |P|$ as $Q$ is contained in a defect group of $b$, and
$P$ is a defect group of $c$, and $b$ and $c$ are Morita equivalent.
Thus $Q$ is a defect group of $b$, $C_Q(A)\cong C_P(A)$
is a defect group of $v$ and $PC_{\bG^F}(A)/C_{\bG^F}(A) =
QC_{\bG^F}(A)/C_{\bG^F}(A)$. This proves ~(a).

Part~(b) follows from~(a) as $Q$ is abelian if and only if $Q=C_Q(A)$ and
$C_Q(A)$ is abelian, and similarly $P$ is abelian if and only if
$P=C_P(A)$ and $C_P(A)$ is abelian. Part~(c) follows from~(a) and~(b)
on observing that $P/Z$ is a defect group of $\bar c$ and $Q/Z$ is a defect
group of $\bar b$.
Finally, suppose that $\la$ is of central $\ell$-defect and $c$ has abelian
defect groups. Then $P= A =Q $.
\end{proof}

\subsection{Proof of Theorem~\ref{thm:all-ab-intro}}  \label{subsec:centquot}

In this subsection, $\bM$ will denote an $F$-stable Levi subgroup of $\bG$,
and  $b$ and $c$  will  be  $\ell$-blocks of
$\bG^F$ and $\bM^F$ respectively.
For large $\ell$, Theorem~\ref{thm:all-ab-intro} follows from the work of
Cabanes--Enguehard and Enguehard.

\begin{prop}   \label{prop:abelian-good}
 Suppose that $b$ and $c$ are Bonnaf\'e--Rouquier correspondents.
 Let $Z$ be a central $\ell$-subgroup of $\bG^F$ and let $\bar b$ and
 $\bar c$ be the images of $b$ and $c$ in $\bG^F/Z$ and $\bM^F/Z$ respectively.
 \begin{enumerate} [\rm(a)]
  \item  If $[\bG,\bG]$ is simply connected, $\ell$ is odd, good for $\bG$ and
   $\ell\ne3$ if $\tw3D_4(q)$ is involved in $\bG^F$, then $\bar b$ and
   $\bar c$ have isomorphic defect groups.
  \item If $\ell=2$, and all components of $\bG$ are classical, then $\bar b$
   and $\bar c$ have a common defect group.
 \end{enumerate}
\end{prop}

\begin{proof}
In the situation of~(a), $c$ has a good pair $(\bL,\la)$ by
Theorem~\ref{thm:goodpairs}(a), and $A= Z(\bL)_\ell^F$ has a complement
in $P$ and $C_P(A)= A$, with $(P,f)$ a maximal $c$-Brauer pair with
$(A, u)\unlhd (P,f)$. If $(Q,d)$ is a maximal $b$-Brauer pair as in
Proposition~\ref{prop:jordan-nonab}(a), then $C_Q(A)= A$, whence
$A$ is a maximal normal abelian subgroup of $Q$. But by the structure of
the defect groups of $b$ as given in \cite[Lemma~4.16]{CE99}, $Q$ has a
unique maximal normal abelian subgroup and this subgroup has a complement
in $Q$. So, $A$ has a complement in $Q$. The result follows from
Proposition~\ref{prop:jordan-nonab}(a).

In~(b) let $\bG_1\leq\bM$ be a Levi subgroup of $\bG$ in duality with
$C_{\bG^*}(s)\leq \bM^*$. Then, by \cite[Prop.~1.5]{En08}, any Sylow
$2$-subgroup of $\bG_1^F$ is a defect group of both $b$ and $c$.
\end{proof}

In fact, as pointed out to us by Marc Cabanes, it can be deduced from
\cite[Lemma~4.16]{CE99} that the two blocks in the situation of
Proposition~\ref{prop:abelian-good}(a) have a common defect group.
The next result will be needed to deal with $E_6$ at $\ell=3$.

\begin{lem}   \label{lem:3A}
 Suppose that $\bG$ is simply connected in characteristic not~$3$ and all
 components of $\bG$ are of type $A$. Let $b$ be a unipotent $3$-block of
 $\bG^F$, $(\bL,\la)$ an $e_3(q)$-cuspidal unipotent pair which is a good
 pair for $b$
 as in Theorem~\ref{thm:goodpairs}(a), $A= Z(\bL)_3^F$ and let $(P,u)$ be a
 maximal $b$-Brauer pair with $(A, b_{\bL^F}(\la))\unlhd (P, u)$. Let $Z$
 be a central subgroup of order~$3$ of $\bG^F$. Suppose that $P$ is
 non-abelian and that $P$ acts trivially on $A/Z$. Then:
 \begin{enumerate}[\rm(a)]
  \item There is an $F$-orbit of  irreducible components of $[\bG, \bG]$ of
   type $A_2$ whose group of $F$-fixed points contains $Z$, and this is the
   only $F$-orbit of irreducible components of $[\bG,\bG] $ whose fixed points
   contain a central subgroup of order~$3$. Further, $P/A$ is cyclic.
  \item  If moreover $Z(P)=Z$ and $P/Z$ is abelian, then $P$ is extra-special
   of order $3^3$.
 \end{enumerate}
\end{lem}

\begin{proof}
Since restriction induces a  bijection between the unipotent characters of
$\bG^F$ and $[\bG, \bG]^F$, there is a unique block, say $b_0$ of
$[\bG,\bG]^F$ covered by $b$ and it is unipotent. Let $I$ be the  set of $F$-orbits on the simple components of $[\bG, \bG]$,
and for each $i\in I$, let $\bH_i$ denote the product of the simple
components in $i$. So $[\bG, \bG]$ is a direct product of the $\bH_i$'s and
$[\bG, \bG]^F$ is a direct product of the $\bH_i^F$'s. For $i\in I$ let
$b_i$ be the block of $\bH_i^F$ covered by $b_0$, let $(\bL_i,\la_i)$
be an $e$-cuspidal unipotent good pair for $b_i$ and let
$(\bL_0,\la_0)=(\prod_{i\in I} \bL_i,\prod_{i\in I}\la_i)$ be the
corresponding pair for $b_0$. Further, let $P_i$ be the first component of a
maximal $b_i$-Brauer pair normalising $(A, b_{\bL_i}^F(\la_i))$.
Then, up to replacing $(\bL, \la)$ and then $(P, u)$ by a $\bG^F$-conjugate,
$(\bL,\la)$ is an extension of $(\bL_0,\la_0)$ as in
Lemma~\ref{lem:quasidefext2}, $A \cap \bL_i ^F = Z(\bL_i)^F_3$ and
$P_i =P\cap \bL_i^F$ for all $i\in I$.

For each $i \in I$, $[P_i, A_i] \leq [P, A]\leq Z$, and $Z$ is cyclic, hence
there exists at most one $i \in I$ with $[P_i, A_i] \ne 1$, say $i=j$.
Since, by Theorem~\ref{thm:goodpairs}(a), $\la_i$ is of central $3$-defect,
we have $A_i=P_i$ for all $ i\ne j$.

Let $i\in I$ and suppose that the rational type of $\bH_i^F$ is
$(A_n,\epsilon q^m)$.
The group $\bH_i^F$ contains a central element of order $3$ if and only if
$3| d_i:=\gcd(q^m-\epsilon,n+1)$. Further, if $3|d_i$, then by
\cite[Prop.~3.3]{CE94}, $b_i$ is the principal block of $\bH_i^F$,
and $P_i$ is a Sylow $3$-subgroup of $\bH_i^F$. Consequently,
if $3|d_i$, then $P_i $ is non-abelian.
So for all $ j\ne i \in I$, $3{\not|}d_i$ and in particular $\bH_i^F$ does
not contain a central element of order~$3$.

By Theorem~\ref{thm:goodpairs}(a), $\la$ is of central defect, hence
$C_P(A)=A$. Since $P$ is non-abelian, $Z \leq [P,P] \leq [\bG, \bG]^F$.
Hence, $\bH_j^F$ contains a central element of order $3$, thus $3|d_j$.
Suppose the rational type of $\bH_j^F$ is $(A_n,\epsilon q^m)$. If $n \geq 5$,
or $3^2|(q^m-\epsilon)$ then $[P_j, A_j] \leq Z$ has order at least $9$. Thus,
$n=2$, $3|| (q^m-\epsilon)$  and $P_j$ is extra-special of order $3^3$.

Let $\bH'= Z(\bG) \bH_j$, let $b'$ be the (unique) block of $\bH'^F$ covered
by $b$ and $P' = P\cap\bH'^F$, a defect group of $b'$. Then $P'$ is a
Sylow $3$-subgroup of $\bH'^F$ and $P= P'\times\prod_{i\in I, j\ne i} A_i$.
Since $Z(\bG) \cap \bH_j^F$ has order at most~$3$, $Z(\bG)_3^F \leq A$,
and $A_j$ has index~$3$ in $P_j$, it follows that $A$ has index $3$ in $P$.
This proves~(a).
\par

Now suppose that the hypothesis of~(b) hold. We have shown above that $P_j$
has order $3^3$. Since $Z=Z(P)$, $P=P'$ and $Z(\bH')^F_3 \leq Z$.
Thus the surjection of $\bH'$ onto $(\bH'/Z(\bH'))$ induces an injection
of $P/Z$ into $(\bH'/Z(\bH'))^F$ . But $(\bH'/Z(\bH'))^F$ has non-abelian
Sylow $3$-subgroups of order $|P_j| $, hence $|P/Z| < |P_j|$ which means
that $P=P_j$ is extra-special of order $3^3$.
\end{proof}

\begin{thm}   \label{thm:all-ab}
 Suppose that $\bG$ is simple, simply connected and that $b$ and $c$ are 
 Bonnaf\'e--Rouquier correspondents.
 Let $Z$ be a central $\ell$-subgroup of $\bG^F$ and let $\bar b$ and
 $\bar c$ be the images of $b$ and $c$ in $\bG^F/Z$ and $\bM^F/Z$ respectively.
 If either $\bar b$ or $\bar c$ has abelian defect groups, then the
 defect groups of $\bar b $ and $\bar c$ are isomorphic.
\end{thm}

\begin{proof}
By Proposition~\ref{prop:abelian-good}, we may assume that either $\bG$ is of
exceptional type and $\ell$ is a bad prime for $\bG$, or $\ell=3$ and
$\bG^F =\tw3D_4(q)$. If $Z =1 $, the statement is immediate from
Proposition~\ref{prop:jordan-nonab}(b). So, we may assume that $Z\ne 1$,
whence either $\ell=3$ and $\bG$ is of type $E_6$, or $\ell=2$ and $\bG$ is of
type $E_7$.

We first consider the case that $c$ is quasi-isolated.
By Theorem~\ref{thm:goodpairs}, $c$ has a good pair, say $(\bL,\la)$.
Set $A= Z(\bL)_\ell^F$, let $(P,f)$ be a maximal $c$-Brauer pair with
$(A,b_{\bL^F}(\la))\unlhd (P,f)$ and let $(Q,d)$ be a maximal $b$-Brauer pair
as in Proposition~\ref{prop:jordan-nonab}(b).
By Proposition~\ref{prop:jordan-nonab}(c), $[P,A]\leq Z$.

We make some reductions. Suppose that $\la$ is of central defect. Then
$C_P(A)=A$, hence either $P=A$ or $[P, A]=Z$ and $Z(P)\leq A$.
If $P=A$, then $P/Z=Q/Z=A/Z$ and there is nothing to prove.
Also, $[P,P]$ is contained in $[C_\bG(z_1), C_\bG(z_1)]$ and in
$[C_\bM(z_1), C_\bM(z_1)]$ for $z_1 \in Z(P)$. Thus,
we may assume the following: If $\la $ is of central defect, then
$Z=[P, A]$, $Z(P)\leq A$ and for any $z_1\in Z(P)$, $[C_\bG(z_1),C_\bG(z_1)]$
and $[C_\bM(z_1), C_\bM(z_1)] $ contain non-trivial central $\ell$-elements.

Next, let $z_1 \in A$, $\bM_1=C_\bM(z_1)$, $\bG_1 =C_\bG(z_1)$.
By Proposition~\ref{prop:Cabanesjordancrux}, there exist blocks $c_1$ and
$b_1$ of $\bM_1^F$ and $\bG_1^F$ respectively, which are Bonnaf\'e--Rouquier
correspondents, and such that
$(\langle z_1\rangle, c_1)$ is a $c$-Brauer pair, and
$(\langle z_1\rangle, b_1)$ is a $b$-Brauer pair. Note that since
$\bG$ and $\bM$ are simply connected, $\bG_1$ and $\bM_1$ are connected.
If $z_1 \in Z(P)$, then $P \leq \bM_1^F $ is a defect group of $c_1 $,
and also $Q \leq \bG_1 ^F$ is a defect group of $b_1$.
Thus, by Proposition~\ref{prop:abelian-good}, applied to the blocks $b_1$
and $c_1$ we may assume the following. If $\bG$ is of type $E_6$, $\ell=3$
and $z_1\in (Z(P) \cap A) \setminus Z$ then $C_\bG(z_1)$ contains a
component of type $D_4$ and if $\bG$ is of type $E_7$, $\ell=2 $ and
$z_1\in (Z(P) \cap A) \setminus Z$, then $C_\bG(z_1)$ contains a
component of type $E_6$.

Suppose that $\bG$ is of type $E_6$ and $\ell=3$.
Then $\bM$ is classical, so $3$ is good for $\bM$
and by Theorem~\ref{thm:goodpairs}, $\la_0$ is of central $3$-defect.
Suppose first that $[\bM, \bM]$ has a component of type $D_4$ or $D_5$.
By rank considerations $[\bM, \bM]$ does not contain a central element
of order $3$, so $\la$ is of central defect by Lemma~\ref{lem:quasidefext2}(c),
whence by the first reduction above
$ Z\leq [P,P] $. But $[P,P] \leq [\bM, \bM]$, a contradiction.

So, we may assume that all components of $\bM$ are of type $A$.
By Theorem~\ref{thm:goodpairs}, $\la $ is of central $3$-defect. By rank
consideration, if $z \in P$ is such that $[C_\bG(z), C_\bG(z)]$ contains
a component of type $D_4$ or $D_5$, then $[C_\bG(z),C_\bG(z)]$ does not
contain a central element of order $3$, hence by the first reduction
$z \notin Z(P)$. By the second reduction, we may assume that $Z(P)= Z$.

Now $C_{\bM^*}(s)/C_{\bM^*}^\circ(s)$ is isomorphic to a subgroup of
$Z(\bM)/Z^\circ(\bM)$, hence to a subgroup of $Z(\bG)/Z^\circ(\bG)$, the
latter being of order~$3$. On the other hand, the exponent of
$C_{\bM^*}(s)/C_{\bM^*}^\circ(s)$ divides the order of $s$, which is prime
to $3$. Thus, $C_{\bM^*}(s)$ is connected, whence $s$ is isolated in $\bM^*$.
But all components of $\bM^*$ are of type $A$, hence $\bM^*$ has no
non-central isolated elements. Thus $s$ is central in $\bM^*$ and
$c=\theta \otimes c'$, where $c'$ is a unipotent block of $\bM^F$
and $\theta$ is a linear character of $\bM^F$ in duality with $s$.
In particular, $P$ is a defect group of a unipotent block of $\bM^F$.
By Lemma~\ref{lem:3A}(a), $P/A$ (and hence $Q/A$) is cyclic of order~$3$, thus
$P/Z$ is abelian if and only if $Q/Z$ is abelian. So, we may assume that
$P/Z$ is abelian. We have shown above that $Z=Z(P)$. By Lemma~\ref{lem:3A}(b),
$P$ is extra-special of order~$3^3$. Thus, $Q$ is extra-special of order~$3^3$,
so $P/Z$ and $Q/Z$ are elementary abelian of order $3^2$, and in particular
isomorphic.

Suppose that $\bG$ is of type $E_7$ and $\ell=2$. Let $Z$ be the
centre of $\bG$ of order~$2$. Suppose first that $\bM$ has a component of
type $E_6$. Then $[\bM, \bM]$ is simple of type $E_6$. Consequently
$[\bM, \bM]$ does not contain a central element of order $2$, and
it follows that $P$ is abelian. By Proposition~\ref{prop:E6-defgrp},
$[\bM, \bM]^F$ does not contain a non-unipotent, quasi-isolated $2$-block
with abelian defect groups, so $c$ covers a unipotent block of $[\bM,\bM]^F$.
By the tables for $E_6(q)$ and $\tw2E_6(q)$ in \cite{En00}, $c_0$ is of
defect~$0$. Since $\bM^F/[\bM, \bM]^F$ has cyclic Sylow $2$-subgroups,
$P$ and hence $Q$ are cyclic and so are $P/Z$ and $Q/Z$.

Thus, we may assume that $ \bM$ is classical. Suppose that $\bM$ has a simple
component, say $\bH_1$ of type $D_n$, $n\ge4$.
The principal $2$-block of $\bH_1^F$ is the only quasi-isolated $2$-block of
$\bH_1^F$, hence $P$ contains a Sylow $2$-subgroup of $\bH_1^F$ and by
Theorem~\ref{thm:goodpairs}(c), we may assume that this Sylow $2$-subgroup
normalises $\bT_2^F$ where $\bT $ is an $F$-stable maximal torus containing
a Sylow $e$-torus and such that the commutator of the Sylow subgroup with
$\bT_2^F$ is contained in a cyclic group of order $2$. But this is not the case.

So, we may assume that all components of $\bM$ are of type $A$.
Then by Theorem~\ref{thm:goodpairs}, $\la$ is of central $2$-defect. By the
same argument as in the $E_6$-case above we conclude that $z \notin Z(P)$
and $Z(P)= Z$.

Let $\bG_1\leq\bM$ be a Levi subgroup of $\bG$ in duality with
$C_{\bG^*}^\circ(s)=C_{\bM^*}^\circ(s)$.
Let $P' \leq N_{\bG_1}(A)$ be a Sylow $2$-subgroup of $\bG_1^F$ as in
Theorem~\ref{thm:goodpairs}(b). Since $[P,A] \leq Z$, $[P', A]\leq Z$.
So, by Lemma~\ref{lem:2selfcentA}(b), the index of $A$ in $P'$ is
$2$. Hence the index of $A$ in $P$ and $Q$ is also $2$ and $P/Z$ is abelian
if and only if $Q/Z$ is abelian. So, we may assume that $P/Z$ is abelian, and
hence that $P'/Z$ is abelian.

Since $Z(P)= Z$, $Z(P')=Z$, and by Lemma~\ref{lem:2selfcentA}(c), $P'$ is
quaternion of order $8$, hence both $P$ and $Q$ are non-abelian of order $8$,
and $P/Z$ and $Q/Z$ are elementary abelian of order $4$.

Now suppose that $c$ is not quasi-isolated in $\bM$.
Then, replacing $\bM $ by an $F$-stable Levi subgroup whose dual contains
$C_{\bM^*}(s)$ and in which $s$ is quasi-isolated, and $\bG$ by $\bM$, the
above argument again gives the desired result (note that above we do not use
that $\bG $ is of type $E_6$ or $E_7$, but only that $Z$ has order $2$ or $3$
and that the rank of $\bG$ is at most~$7$).
\end{proof}

\section{Brauer's height zero conjecture} \label{sec:proof}

In this section we give the arguments which are necessary to combine our
results and those obtained previously by various authors to prove (HZC1),
that is, Theorem~\ref{thm:main}.

\subsection{Groups not of Lie type}

\begin{prop}   \label{prop:spor}
 Let $S$ be a perfect central extension of a sporadic simple group or the Tits
 simple group $\tw2F_4(2)'$. Then (HZC1) holds for $S$.
\end{prop}

\begin{proof}
It is well known that a Sylow $p$-subgroup of a covering group of a sporadic
simple group of order at least $p^3$ is non-abelian unless $S=J_1$ and $p=2$,
or $S=ON$ and $p=3$. Since the block distribution of ordinary characters
as well as the size of the respective defect groups can easily be obtained
using {\sf GAP}, this deals with most blocks in question. For the remaining
blocks (i.e., non-principal blocks with defect group of order at least~$p^3$)
which are only in characteristic~2 or~3, either the structure of the defect
group is given by Landrock \cite{La78}, or it can easily be shown to be of
extra-special type (see M\"uller \cite{Mue10}).
\end{proof}

\begin{prop}   \label{prop:exccov}
 Let $S$ be an exceptional covering group of a finite simple group of Lie
 type, or of $\fA_7$. Then (HZC1) holds for $S$.
\end{prop}

\begin{proof}
From the ordinary character tables in \cite{Atl} it follows that all $p$-blocks
of the groups in question fall into three categories: either all characters
in the block are of height zero, or the block is principal and the Sylow
$p$-subgroups are non-abelian, or $p=2$, $S=3.\OO_7(3)$ or $6.\OO_7(3)$.
\par
Let $S = 6.\OO_7(3)$, let $b$ be a $2$-block of $S$ and denote by $\bar b$ the
corresponding $2$-block of $\bar S:=3.\OO_7(3)$. By the modular atlas
\cite{MAtlwww}, the defect groups of $b$ have order $2^{10}$, $16$, $4$ or $2$.
In the first case, the defect groups of $b$ (respectively $\bar b$) are
Sylow $2$-subgroups of $S$ (respectively $\bar S$) and hence non-abelian.
If the defect groups of $b$ are cyclic or Klein $4$-groups, then all
characters in $b$ and $\bar b$ are of height zero. So assume that the defect
groups of $b$ have order $16$, and hence that the defect groups of $\bar b$
have order $8$. From ordinary character tables it follows that there exists
an irreducible character in $\bar b$ which does not vanish on an element
of $\bar S$ of order $4$. Thus, the defect groups of $\bar b$ are not
elementary abelian. On the other hand, by \cite{MAtlwww}, $\bar b$ has two
modular irreducible characters. Since blocks with defect groups isomorphic
to $C_4 \times C_2$ or to $C_8$ have a unique modular irreducible character
and since blocks with defect groups isomorphic to $Q_8$ have either one or
three modular irreducible characters, it follows that the defect groups of
$\bar b$ are dihedral. In particular, the defect groups of $\bar b$ and of
$b$ are non-abelian.
\end{proof}

\subsection{Bad primes for exceptional type groups}
We will need the following result of Enguehard \cite[\S3.2]{En00},
respectively Ward \cite{Wa66} and Malle \cite{MaF}.
Let $\bG$ be connected reductive with a Steinberg endomorphism
$F:\bG\rightarrow\bG$ and let $\ell$ be a prime number different from the
defining characteristic of $\bG$. By \cite[Thm. A]{En00}, the assertions of
Theorem~\ref{thm:mainblocks} hold for $\bG^F$ and $\ell$ for the case $s=1$, 
and with the ``quasi-central $\ell$-defect'' condition in~(a3) of 
Theorem~\ref{thm:mainblocks} replaced by  ``central $\ell$-defect''.
For a unipotent $e$-cuspidal pair $(\bL,\la)$ of $\bG$ such that 
$\la$ is of central $\ell$-defect and $S=\bG^F/Z$
for some central subgroup $Z$ of $\bG^F$ we denote by $b_S(\bL,\la)$ the
$\ell$-block of $S$ corresponding to  $(\bL, \la)$, respectively
its image in $\bG^F/Z$, and by $W_{\bG^F}(\bL,\la)$ the relative Weyl group
$N_{\bG^F}(\bL,\la)/\bL^F$.

\begin{prop}   \label{prop:unipabdef}
 Suppose that $\bG$ is simple, simply connected of exceptional type and that
 $\ell$ is a bad prime for $\bG$. Let $S =\bG^F/Z$ for some central subgroup
 $Z$ of $\bG^F$ and $B=b_S(\bL,\la)$ a unipotent $\ell$-block of $S$ with
 non-trivial abelian defect groups.
 Then $B$ is as in Table~\ref{tab:unipabdef} or Ennola dual to an entry there.
 Moreover, $W_{\bG^F}(\bL,\la)$ is an $\ell'$-group, except for the listed
 entries for $E_6(q)$ and $E_7(q)$ and their Ennola duals.
\end{prop}

\begin{table}[htbp]
\caption{Unipotent $\ell$-blocks of quasi-simple exceptional groups with
   non-trivial abelian defect group, $\ell$ bad}   \label{tab:unipabdef}
\[\begin{array}{|r|c|lll|}
\hline
 S& (\ell,e)& \bL^F& \la& \text{conditions}\\
\hline\hline
 ^2G_2(q^2)& (2,1)& \Ph1\Ph2& 1& \\
\hline
  F_4(q)& (3,1)& \Ph1^2.B_2(q)& B_2[1]& \\
\hline
  E_6(q)& (3,1)& \Ph1^2.D_4(q)& D_4[1]& 3||q-1,\ Z(S)=1\\
\hline
 \tw2E_6(q)& (3,1)& \Ph1.\tw2A_5(q)& \phi_{321}& \\
\hline
  E_7(q)& (2,1)& \Ph1.E_6(q)& E_6[\theta^{\pm1}]& 4||q-1,\ Z(S)=1\\
 \hline
  E_8(q)& (3,1)& \Ph1.E_7(q)& E_7[\pm\xi]& \\
        & (5,1)& \Ph1^4.D_4(q)& D_4[1]& \\
        & (5,1)& \Ph1^2.E_6(q)& E_6[\theta^{\pm1}]& \\
        & (5,1)& \Ph1.E_7(q)& E_7[\pm\xi]& \\
        & (5,4)& \Ph4^2.D_4(q)& \xi_1,\ldots,\xi_4& \\
\hline
\end{array}\]
\end{table}

For a prime $\ell$, by a minimal counterexample to (HZC1) we will mean
a pair $(\chi, S)$, such that $S$ is a finite group, $\chi \in \Irr(S)$ is
an irreducible character lying in an $\ell$-block of $S$ with abelian defect
such that $\chi $ has positive height and such that $(\chi(1),|S|)$ is
minimal with respect to the lexicographical ordering on such pairs.

\begin{prop}   \label{prop:excunibad}
 Suppose that $\bG$ is simple, simply connected of exceptional type and that
 $\ell$ is a bad prime for $\bG$. Let $S =\bG^F/Z$ for some central subgroup
 $Z$ of $\bG^F$ and $B=b_S(\bL,\la)$ a unipotent $\ell$-block of $S$.
 Then $(\chi, S)$ is not a minimal counterexample to (HZC1) for any
 $\chi \in \Irr(B)$.
\end{prop}

\begin{proof}
For ${}^2G_2(q^2)$ the validity of (HZC1) follows from the results in
\cite{Wa66}. Note that it is also known to hold for blocks with cyclic defect
group. So now let $B=b_S(\bL,\la)$, for $(\bL,\la)$ a unipotent $e$-cuspidal
pair of central $\ell$-defect and suppose that $B$ has non-cyclic abelian
defect groups.

Assume that $\bG$ is of type $F_4$ or $E_8$. Then $S =\bG^F$. Further,
by Proposition~\ref{prop:unipabdef}, $W_{\bG^F}(\bL,\la)$ is an $\ell'$-group.
We claim that $Z(\bL)_\ell^F$ is a Sylow $\ell$-subgroup of
$C^\circ_\bG([\bL,\bL ])^F$. Indeed, since $\bL=C_\bG(Z^\circ(\bL)_{\Phi_e})$
and since by \cite[Prop.~1.7(ii)]{CE94}, $Z^\circ(\bL)_{\Phi_e}$ is a Sylow
$e$-torus of $C_\bG^\circ([\bL, \bL])$, by the argument before
\cite[Lemma~4.5]{CE94}, $N_{\bG^F}(\bL)$ contains a Sylow $\ell$-subgroup,
say $D$ of $C^\circ_\bG([\bL,\bL])^F$.
Since $D$ centralises $[\bL,\bL]^F $, and $\la$ is  determined by its
restriction to $[\bL,\bL]^F$, $D \leq N_{\bG^F}(\bL,\la)$. But
$W_{\bG^F}(\bL,\la)$ being an $\ell'$-group means that $D \leq \bL$, and
hence $D \leq Z(\bL)$.

Since $\bG$ is self-dual, we may and will identify $\bG^*$ with $\bG$
in such a way that the resulting  correspondence between unipotent
$e$-cuspidal pairs of $\bG$ and $\bG^*$ is the correspondence of
\cite[Prop.~15]{En00}.
Let $ t\in \bG^*$ be an $\ell$-element such that $\chi\in\cE(\bG^F,t)$.
Let $\bH =C_\bG(t)$ and $\psi\in\cE(\bH^F,1)$
be the Jordan correspondent of $\chi$ in $C_{\bG^*}(t)$.
Since $\bG$ has trivial centre, by \cite[Thm.~B, Prop.~17]{En00},
there is a unipotent $e$-cuspidal pair $(\bL_t,\la_t)$ for
$\bH$, with central $\ell$-defect such that
$([\bL, \bL],\Res^{\bL^F}_{[\bL, \bL]^F}\la)$ and
$([\bL_t, \bL_t], \Res^{\bL_t^F}_{[\bL_t, \bL_t]^F} \la_t)$ are
$\bG^F$-conjugate, and such that $\psi$ is in the block
$b_{\bH^F}(\bL_t,\la_t)$. Further, $(\bL_t,\la_t)$ is uniquely determined
up to $\bH^F$-conjugacy.
\par
Since $t$ commutes  with $[\bL_t,\bL_t]$, some $\bG^F$-conjugate of
$t$ commutes with $[\bL,\bL]$, and so by the claim above,
we may assume that $t \in Z(\bL)_\ell^F$, and hence
that $(\bL_t,\la_t)=(\bL,\la)$. By \cite[Prop.~8, 8.bis]{En00},
$\bL= C_\bG^\circ(Z(\bL)_\ell^F)$ and $\bL^F=C_{\bG^F}(Z(\bL)_\ell^F)$ 
(in fact, $\bL= C_\bG(Z(\bL)_\ell^F)$),
hence, also
$\bL= C_\bH^\circ(Z(\bL)_\ell^F)$ and $\bL^F = C_{\bH^F}(Z(\bL)_\ell^F)$.
Since $W_{\bH^F}(\bL,\la) \leq W_{\bG^F}(\bL,\la)$ are  $\ell'$-groups,
$Z(\bL)_\ell^F$ is a defect group of $b_{\bH^F}(\bL_t, \la_t)$ and
of $B$ by Proposition~\ref{prop:defect:re}. \par
Now by the degree formula for Jordan decomposition, $\psi$ and $\chi$ have
the same defect and thus the same height, whence $(\psi,\bH^F)$
is a counterexample to (HZC1). Since $\psi(1) \leq \chi(1)$,
$(\chi,S)$ is a minimal counterexample only if $t$ is central, so $1$,
and hence only if $\chi$ is a unipotent character. But it is easy to
check  from the decomposition  of Lusztig induction of  the relevant
unipotent $e$-cuspidal  pairs  that all unipotent characters in 
$B$ are of zero height.
\par
Now suppose that $\bG$ is of type $E_6$. If $S$ is as in line~4 of the table, 
then the defect groups of $B$ are cyclic. So, we assume that $S$ is as in 
line~5 of the table. 
By Proposition~\ref{prop:unipabdef}, $\ell=3$, $3||(q-1)$, $Z \ne 1$,
and $S$ may be assumed to be the commutator subgroup of
${\hat \bG}^F$, where $\hat \bG= (E_6)_\ad $. Denote by
$\hat B$ the unipotent block of ${\hat \bG}^F$ covering $B$.
Let $\hat \chi \in \Irr(\hat B)$ cover $\chi$ and let $t\in \hat\bG^{*F}$
be such that $\hat \chi \in \cE(\hat \bG^F, t)$.
Since $\hat \bG$ has connected centre, again by \cite[Thm.~B, Prop.~17]{En00},
(and using the canonical correspondences between unipotent $e$-cuspidal pairs
for groups of the same type), $C_{\hat \bG^*}(t)$ contains a unipotent
$e$-cuspidal pair $(\bL_t,\la_t)$ such that $[\bL_t,\bL_t] \cong  [\bL,\bL]$.
There are only three classes of centralisers of semisimple elements of
${\hat \bG}^{*F}= E_6(q)_\SC$ containing Levi subgroups of
type $D_4$: one of type $\Ph1^2.D_4(q)$, one of type $\Ph1.D_5(q)$, and
$\hat\bG^{*F}$ itself. For $t\in Z({\hat \bG}^*)$, the elements of
$\cE(\hat \bG^F,t)$ have the same restrictions to $S$ as the unipotent
characters. Since~3 divides $q-1$
precisely once, there are no 3-elements with centraliser $\Phi_1.D_5(q)$.
Finally, there is exactly one class of elements of order~3 with centraliser
$H=\Ph1^2.D_4(q)$. The Jordan correspondent  of $\hat \chi $ is therefore
the unique (cuspidal) character $D_4[1]$ in $\cE(H,1)$,  $\hat \chi $
is  the only possible non-unipotent character in $\hat B$, and it has height~1.
Since the image of $t$ in the adjoint type group
${\hat \bG}^*/Z({\hat \bG}^*)$ has disconnected centraliser, the restriction
of $\hat\chi$ to $S$ has three irreducible constituents, which are thus of
height zero. In particular, (HZC1) is satisfied for $\hat B$. Exactly the same
reasoning applies to the Ennola dual case,
and a slight variation is valid for $\bG$ of type $E_7$.
\end{proof}

Next we show that no other quasi-isolated block provides a minimal
counterexample to (HCZ1).

\begin{lem}   \label{lem:goodLevi-split}
 Let $\bG$ be connected reductive such that $\bG$ and $\bG^*$ have connected
 centre. Let $s\in\bG^{*F}$ be a semisimple $\ell'$-element such that $\bG$
 and $s$ satisfy the assertions of Theorem~\ref{thm:BMM} and that
 all $e$-cuspidal pairs $(\bL,\la)$ of $\bG$ below $\cE(\bG^F,s)$ satisfy
 $$C_\bG^\circ (Z (\bL)_\ell^F)  =\bL,\quad
   C_{\bG^F} (Z (\bL)_\ell^F)  =\bL^F,$$
 and $\la$ is of central $\ell$-defect. Let $t\in\bG^{*F}$ be an
 $\ell$-element commuting with $s$ and suppose that there exists a proper
 $F$-stable Levi subgroup $\bM$ of $\bG$ such that the following holds.
 \begin{enumerate} [\rm(1)]
  \item  $C_{\bG^*} (st) \leq  \bM$.
  \item $\bM$ is $e$-split.
  \item One of the following holds.
  \begin{enumerate} [\rm(a)]
   \item For all $e$-cuspidal pairs $(\bL,\la)$ of $\bG$ below $\cE(\bG^F,s)$,
    $W_{\bG^F}(\bL,\la)$ is an $\ell'$-group and there exists an $F$-stable
    Levi subgroup $\bM_0$ of $\bM$ such that $C_{\bM^{*}}(s)\leq \bM_0$ and
    such that $\ell$ is good for $\bM_0$.
   \item  $\ell \geq 5 $  and  $\ell $ is good for $\bM$.
  \end{enumerate}
 \end{enumerate}
 Then $(\chi,\bG^F)$ is not a minimal counterexample to (HZC1) for any
 $\chi \in \cE(\bG^F, st)$.
\end{lem}

\begin{proof}
The hypotheses on $\bG$ (and our results in Section 2) imply that for any
$e$-split Levi subgroup $\bH$ of $\bG$ and for any $\bH^F$-conjugacy
class of $e$-cuspidal pairs $(\bL,\la)$ of $\bG$ below $\cE(\bH^F,s)$
there is an $\ell$-block $b_{\bH^F}(\bL,\la)$ of $\bH^F$ in
$\cE_\ell(\bH^F, s)$ such that all constituents of $R_{\bL}^{\bH}(\la)$ lie in
$\Irr(b_{\bH^F}(\bL,\la)\cap \cE(\bH^F,s)$, and such that
$(Z(\bL)^F, b_{\bL^F}(\la))$ is a centric $b_{\bH^F}(\bL,\la)$-Brauer pair.
Moreover, any $\ell$-block of  $\bH^F$  in
$\cE_\ell(\bH^F, s)$  is of the form $b_{\bH^F}(\bL,\la)$ for some
$e$-cuspidal pair $(\bL,\la)$ of $\bG$ below $\cE(\bH^F,s)$.

Let $\chi\in \cE(\bG^F, st)$. By~(1), there exists
$\phi\in \cE(\bM^F, st)$ such that $\chi = \pm\RMG(\phi)$.
Let $b:=b_{\bM^F}(\bL,\la)$ be the $\ell$-block of $\bM^F$ containing
$\phi$ and set $c:= b_{\bG^F}(\bL,\la)$.
Since $W_{\bM^F}(\bL,\la)$ is a subgroup of  $W_{\bG^F}(\bL,\la)$,
by Proposition~\ref{prop:defect:re}(e),(f) if $c$ has abelian defect groups,
then $b$ has abelian defect groups, and
$Z(\bL)^F$ is a defect group of both $b$ and $c$. So,  since
$\phi(1) < \chi(1)$ and $\phi$ and $\chi$ have the same $\ell$-defect,
it suffices to prove that $c$ contains $\chi$, or equivalently that
$d^{1,\bG^F}(\chi)\in c$.

We claim that for any $\psi \in\Irr(b)\cap \cE(\bM^F, s)$, all constituents
of $\RMG(\psi)$ lie in $c$. Indeed, note that by 
Proposition~\ref{prop:HC-coroll} in order to prove the claim,
it suffices to prove that $\Irr(b)\cap \cE(\bM^F,s)$ is
precisely the set of constituents of $\RLM(\la)$. If~(3b) holds then
this follows from the main theorem of \cite{CE99}.
Suppose that~(3a) holds. Then for any $e$-split Levi subgroup $\bH$ of
$\bG$ and any $e$-cuspidal pair $(\tbL,\tilde\la)$ of $\bG$ below
$\cE(\bH^F,s)$, $(Z(\tbL^F), b_{\tbL^F}(\tilde\la))$ is a maximal
$b_{\bH^F}(\tbL,\tilde\la)$-Brauer pair. Consequently, the map
$$(\tbL,\tilde\la) \to b_{\bH^F}(\tbL,\tilde\la)$$
induces a bijection between the $\bH^F$-conjugacy classes of $e$-cuspidal
pairs of $\bG$ below $\cE(\bH^F,s)$ and the set of $\ell$-blocks
in $\cE(\bH^F, s)$, and
$\Irr(b_{\bH^F}(\tbL,\tilde\la))\cap\cE(\bH^F,s)$ is
precisely the set of constituents of $R_{\bL}^{\bH}(\la)$.

Now,
$$d^{1,\bG^F}(\chi) = \pm d^{1,\bG^F}(\RMG(\phi))= \RMG(d^{1,\bM^F}(\phi)).$$
Hence, by the claim above, it suffices to prove that $\Irr(b)\cap\cE(\bM^F,s)$
is an $\ell$-basic set for $b$. Suppose first that~(3a) holds. 
Since $\bG$ has connected centre, so does $\bM$ and by hypothesis,
$\ell$ is good for $\bM$. So, by \cite[Thm.~A]{Ge93}, $\Irr(b)\cap\cE(\bM^F,s)$
is an $\ell$-basic set for $b$. In case~(3b), let $b_0$ be the
Bonnaf\'e--Rouquier correspondent of $b$ in $\bM_0^F$. By the previous
argument, applied to $\bM_0$ instead of $\bM$, $\Irr(b_0)\cap\cE(\bM_0^F,s)$
is an $\ell$-basic set for $b_0$. The result follows as the
Bonnaf\'e--Rouquier Morita equivalence preserves basic sets.
\end{proof}

The next two results will allow us to verify the conditions of the previous
Lemma for certain situations in $E_8$.

\begin{prop}   \label{prop:in esplit}
 Let $\bH$ be connected reductive with Steinberg endomorphism $F$.
 Let $\ell$ be a prime different from the defining characteristic of $\bH$,
 good for $\bH$, not dividing $|Z(\bH)/Z^\circ(\bH)|$, and not a torsion prime,
 and set $e=e_\ell(q)$. Assume that one of the following holds:
 \begin{enumerate}[\rm(1)]
  \item $e$ is the unique integer such that $\ell|\Phi_e(q)$ and
   $\Phi_e(q)\big\vert |\bH^F|$; or
  \item $e\in\{1,2\}$ and a Sylow $e$-torus of $\bH$ is a maximal torus.
 \end{enumerate}
 Then the centraliser of any $\ell$-element $1\ne t\in\bH^F$ lies in the
 centraliser of a non-trivial $e$-torus.
\end{prop}

\begin{proof}
Clearly it suffices to prove the assertion for $t$ of order~$\ell$. Let $t$
be of order~$\ell$ and set $\bC:=C_\bH^\circ(t)$. Then $\bC$ is a Levi
subgroup of $\bH$, and $t\in\bC$ (see \cite[Prop.~14.1]{MT}). Moreover,
$t\in Z(\bC)$, and as $|Z(\bC)/Z^\circ(\bC)|$ divides $|Z(\bH)/Z^\circ(\bH)|$
by \cite[Prop.~1.1.2(b)]{En08}, we even have $t\in Z^\circ(\bC)$.
Thus $Z^\circ(\bC)$ is a torus with $|Z^\circ(\bC)^F|$ divisible by $\ell$.
Under assumption~(1) this implies that $Z^\circ(\bC)$ contains a non-trivial
$e$-torus $\bT$, and thus $\bC\le C_\bH(\bT)$.
\par
In case~(2), let $\bT$ denote a Sylow $e$-torus of $\bH$. Then $t$ is
$\bH$-conjugate to some element of $\bT$, by \cite[Cor. 6.11]{MT}. As
$\ell|\Phi_e(q)$ we see that all elements of order~$\ell$ of $\bT$ are
$F$-stable, so lie in $\bT^F$. But the centraliser of $t$ is connected by
\cite[Ex.~20.16]{MT}, hence $t$ is even $\bH^F$-conjugate to an element
of $\bT^F$ by \cite[Thm.~26.7]{MT}. We may assume that actually $t\in \bT^F$.
Then in particular $\bC$ contains the maximal
torus $\bT$, whence $Z^\circ(\bC)\le\bT$ is a non-trivial $e$-torus.
\end{proof}

\begin{lem}   \label{lem:3-elts}
 Assume that $\bG^F=E_8(q)$ with $q\equiv1\pmod3$. Let $s\in\bG^F$ be a
 quasi-isolated $5$-element such that $F$ induces a nonsplit Steinberg
 endomorphism on $\bH:=C_\bG(s)$. Then for any non-trivial $3$-element
 $t\in\bH^F$, $C_\bG(st)$ is contained in a Levi subgroup $\bM_0$ of $\bG$
 of classical type, which itself lies in a proper $1$-split Levi subgroup
 $\bM$ of $\bG$.
\end{lem}

\begin{proof}
According to Table~\ref{tab:quasi-el} (or Table~\ref{tab:quasi-E8-3}) we have
$\bH^F$ of type either $\tw2A_4(q^2)$ or $\tw2A_4(q)^2$. It is easy (using
Jordan normal forms) to work out the types of 3-elements in $\bH^F$
and the isomorphism types of their centralisers; the result is given in
Table~\ref{tab:3-elts}.

\begin{table}[htbp]
\caption{Centralisers of $3$-elements in $\bH^F$}  \label{tab:3-elts}
\[\begin{array}{|ll|ll|}
\hline
 \span \bH^F=\tw2A_4(q)^2& \span \bH^F=\tw2A_4(q^2)\\
 C_{\bH^F}(t)& Z^\circ(C_{\bH^F}(t))&   C_{\bH^F}(t)& Z^\circ(C_{\bH^F}(t))\\
\hline\hline
            \GL_2(q^2)^2& \Ph1^2\Ph2^2&           \GL_2(q^4)& \Ph1\Ph2\Ph4\\
 \Ph1.\GL_2(q^2)\GU_3(q)& \Ph1^2\Ph2^2&  \Ph1\Ph2.\GU_3(q^2)& \Ph1\Ph2\Ph4\\
       \Ph1^2.\GU_3(q)^2& \Ph1^2\Ph2^2& &\\
      \GL_2(q^2)\SU_5(q)&       \Ph1^2& &\\
   \Ph1.\GU_3(q)\SU_5(q)&     \Ph1\Ph2& &\\
\hline
\end{array}\]
\end{table}

Clearly, $\bM_0:=C_\bG(Z^\circ(C_\bH(t)))$ and
$\bM:=C_\bG(Z^\circ(C_\bH(t))_{\Ph1})$ are Levi subgroups of $\bG$ containing
$C_\bG(st)$, $\bM_0\le\bM$, and $\bM$ is $1$-split and proper. Further,
$\bM_0$ has semisimple rank at most $8-\dim(Z^\circ(C_\bH(t)))$, hence it is
of classical type unless $\dim(Z^\circ(C_\bH(t))\le2$, which happens precisely
for the last two centralisers for $H=\tw2A_4(q)^2$. But there,
$C_\bH(t)$ is of type $A_4+A_2$, $A_4+2A_1$ respectively, and these do not
embed into a group of type $E_6$, by the Borel--de Siebenthal algorithm
(see \cite[Thm.~13.12]{MT}).
Hence in these cases as well, $\bM_0$ is of classical type.
\end{proof}

\begin{prop}   \label{prop:exc-height0}
 Suppose that $\bG$ is simple, simply connected of exceptional type
 $F_4$, $E_6,E_7$ or $E_8$ and that $\ell$ is a bad prime for $\bG$.
 Let $S =\bG^F/Z$ for $Z$ a central subgroup of $\bG^F$ and let $B$ be an
 $\ell$-block of $S$ such that the block of $\bG^F$ lifting $B$ is a
 quasi-isolated, non-unipotent block of abelian defect as in
 Tables~\ref{tab:quasi-F4}--\ref{tab:quasi-E8-5b} or their Ennola duals.
 Then $(\chi,S)$ is not a minimal counterexample to (HZC1) for
 any $\chi\in\Irr(B)$.
\end{prop}

\begin{proof}
For $F_4(q)$, by Proposition~\ref{prop:F4-defgrp}, $B$ is one of the blocks
numbered 3, 5 or~7 or their Ennola duals.
Now note that in all three cases, the $\ell$-power in the degrees of characters
in $\Irr(B)\cap\cE(\bG^F,s)$ is maximal among all elements of
$\cE_\ell(\bG^F,s)$, while on the other hand, they are of height zero in $B$.
Hence, no character in $\Irr(B)\subseteq\cE_\ell(\bG^F,s)$ can have positive
height.
\par
If $S$ is of type $E_6(q)$, then we note by Table~3 that $B$ has abelian
defect groups only if the block of $\bG^F$ lifting $B$ has abelian defect
groups (note that in Lines~13 and~14 of the table the action of the relative
Weyl group does not become trivial on passing to $\bG^F/Z(\bG^F)$).
Hence, only the block numbered~15 has to be considered.
Here, either we can apply the same argument as for $F_4(q)$, or alternatively
observe that the defect groups are cyclic.
The same arguments apply to the unique quasi-isolated block with abelian
defect group of $\tw2E_6(q)$. \par
According to Proposition~\ref{prop:E7-defgrp} and Table~4, for $E_7(q)$ again
the defect groups of $B$ are abelian if and only if the defect groups of
the block of $\bG^F$ lifting $B$ are abelian and this occurs only
for the blocks~4, 7, 11, 13, 16 and~18. In all cases,
the characters in $\Irr(B)\cap\cE_\ell(\bG^F,s)$ have the maximal possible
$\ell$-part in their degree, so we conclude as before. \par

For $E_8(q)$ the blocks with abelian defect were described in
Propositions~\ref{prop:E8-2-defgrp}, \ref{prop:E8-3-defgrp}
and~\ref{prop:E8-5-defgrp}. In particular, there are no cases when $\ell=2$.
For $\ell=3$, we need to treat blocks~5, 11, and 13--17.
Here, the cases~5 and 11 follow by the standard argument on maximal
$\ell$-power in the degrees. In the remaining cases, that is, when $s$ is a
quasi-isolated 5-element with non-split centraliser, the conditions (1)--(3a)
of Lemma~\ref{lem:goodLevi-split} were shown to hold in Lemma~\ref{lem:3-elts},
whence the claim.
\par
Finally assume that $\ell=5$. We conclude as before when $\cE_5(\bG^F,s)$ is
a single $5$-block. It is straightforward to check that all other
centralisers $\bH$ of quasi-isolated elements $s$ in
Table~\ref{tab:quasi-E8-5b} satisfy condition~(1) of
Proposition~\ref{prop:in esplit}, and those in Table~\ref{tab:quasi-E8-5}
satisfy either~(1) or~(2). Thus there is a non-trivial $e$-torus $\bT$
contained in the centre of $C_{\bG^{*F}}(st)$, and we let $\bM\leq\bG$ be
a Levi subgroup in duality with $C_{\bG^{*F}}(\bT)$. So $\bM$ is proper in
$\bG$ and~(1), (2) and~(3b) of Lemma~\ref{lem:goodLevi-split} hold, which gives
the claim. This completes the proof.
\end{proof}

\subsection{Proof of (HZC1)}

\begin{thm}   \label{thm:BrauerHZC}
 The 'if part' (HCZ1) of Brauer's height zero conjecture holds.
\end{thm}

\begin{proof}
We investigate minimal counterexamples to the assertion. For this let $S$ be
a finite group, $p$ a prime, $B$ a $p$-block of $S$ with abelian defect group
$D$ and $\chi\in\Irr(B)$ an irreducible character of $B$ of positive height,
such that $(\chi(1),|S|)$ is minimal with respect to the lexicographic
ordering on such pairs. \par
Then by the principal result of \cite{BK}, $S$ is quasi-simple, with simple
central factor group $X$, say. We may assume that $X$ is not alternating, by
Olsson's result \cite{Ol90}, respectively by Proposition~\ref{prop:exccov} for
the exceptional covering groups of $\fA_6$ and $\fA_7$. (Note that for the
double cover of $\fA_n$ Olsson only treats the odd primes but the abelian
defect groups of $2$-blocks of alternating groups are of order at most $4$,
and (HCZ1) is known for such blocks by results of Brauer.) Furthermore, $X$ is
not sporadic by Proposition~\ref{prop:spor}, nor a special linear or unitary
group by the theorem of Blau and Ellers \cite[Thm.~5]{BE}, respectively by
Proposition~\ref{prop:exccov} for their exceptional covering groups. Thus by
the classification of finite simple groups, $X$ is a simple group of Lie type
not of type $A_n$ or $\tw2A_n$.
\par
There is a simple algebraic group $\bG$ of simply connected type with a
Steinberg endomorphism $F:\bG\rightarrow\bG$ such that $X\cong \bG^F/Z(\bG)^F$
(recall that we consider $\tw2F_4(2)'$ as a sporadic group). Moreover, by
Proposition~\ref{prop:exccov} we may assume that $S=\bG^F/Z$ for some central
subgroup $Z\le Z(\bG)^F$. Now first assume that $p$ is the defining prime
for $X$. Then $\gcd(p,|Z|)=1$, so any $p$-block of $S$ is also a $p$-block of
$\bG^F$, with the same defect group. By the theorem of Humphreys \cite{Hum}
the $p$-blocks of $\bG^F$ are either of defect zero or of full defect. But
the Sylow $p$-subgroup of $\bG^F$ is non-abelian unless $\bG=\SL_2$, which was
excluded before.
\par
Thus, $p$ is not the defining prime for $\bG$. Let $\bG^*$ be a group in
duality with $\bG$, thus of adjoint type, with compatible Steinberg morphism
$F:\bG^*\rightarrow\bG^*$. Let $B_0$ be the $p$-block of $\bG^F$ containing
the lift, say $\chi_0 $, of $\chi$ to $\bG^F$ and let $s\in \bG^{*F}$ be a
semisimple $p'$-element such that $B_0 \subseteq \cE_p(\bG, s)$.
\par
Now assume first that $p$ is a good prime for $\bG$, and $p\ne3$ if $F$ is a
triality automorphism. Since $\bG$ is not of type $A_n$ and $p$ is good for
$\bG$, $p$ does not divide $|Z(\bG^F)|$, hence $B_0$ and $B$ have isomorphic
defect groups. So, $\chi$ and $\chi_0$ have equal degree and equal height.
In particular, $\chi_0$ is not of zero height. Again since $p$ is good for
$\bG$ and $p\ne3$ if $F$ is a triality automorphism, by the theorem of
Enguehard \cite[Thm.~1.6]{En08}, there is a reductive algebraic group
$\bG(s)$ with a Steinberg endomorphism
$F:\bG(s)\rightarrow\bG(s)$, such that $\bG(s)^\circ$ is in duality with
$C_{\bG^*}^\circ(s)$, a $p$-block $b$ of $\bG(s)^F$ with isomorphic defect
group $D$ and a height preserving bijection $\Xi:\Irr(B_0)\rightarrow\Irr(b)$.
Moreover, $\Xi(\psi)(1)|\psi(1)$ for all $\psi\in \Irr(B_0)$. Thus, if
$|\bG(s)^F|<|S|$, then $(\Xi(\chi),\bG(s)^F)$
cannot be a counterexample, so neither can $(\chi,S)$. Hence, in this case we
must have $|\bG(s)^F| \geq |S|$, so $s\in Z(\bG^*)=1$
(since $\bG^*$ is adjoint). So $B_0$ is unipotent. \par
Since $\bG$ is not of type $A_n$, $p$ is good and $p\ne3$ when $F$ induces
triality, $p$ is even $(\bG,F)$-excellent in the sense of
Brou\'e--Michel \cite[Def.~1.11]{BM93}.
Thus, by \cite[Thm.~3.1]{BM93} there is an isotypie between the unipotent block
$B_0$ and a block $b$ of the normaliser of some non-trivial $p$-subgroup; in
particular Brauer's height zero conjecture holds for $B_0$ and hence for $B$,
contradicting our choice.
\par
Thus, either $p$ is a bad prime for $\bG$, or $p=3$ and $F$ induces triality.
Let $\bM $ be an $F$-stable Levi-subgroup of $\bG$ such that
$C_{\bG^*}(s) \leq \bM^*$ and $s$ is quasi-isolated in $\bM^*$. Let $C_0$ be the
Bonnaf\'e--Rouquier correspondent of  $B_0 $ in   $\bM^F$.  Then,
$\psi \mapsto \epsilon_\bG \epsilon_\bM \RMG(\psi)$ is a
height preserving bijection between $\Irr (C_0)$ and $\Irr(B_0)$.
For $\psi \in \Irr (C_0)$ and $z\in \bG^F$, $z \in \ker(\psi)$
if and only if $z \in \ker(\epsilon_\bG\epsilon_\bM \RMG(\psi))$.
Hence, the character of $C_0$ corresponding to $\chi_0$ via the above
bijection is the lift of a character, say $\tau $, of
$\bM^F/Z$ to $\bM^F$.  Let  $C$ be the $\ell$-block of $\bM^F/Z$
containing $\tau$.
By Theorem~\ref{thm:all-ab} the defect groups of $C$ are abelian and of the
same size as the defect groups of $B$ (note that
Theorem~\ref{thm:all-ab} applies to  the Sylow $p$-subgroup of $Z$ and  defect
groups remain unchanged on passing to quotients by $p'$-groups).
Now $\tau $ and $\chi$ have  equal heights, and $\tau(1) \leq \chi(1)$.
So,  $\bM^F/Z =\bG^F/Z$, $\bM =\bG$ and $s$ is quasi-isolated in $\bG^*$.
\par
Assume first that $p=2$ and $\bG^F$ is of classical type different from
$A_n$ or $\tw2A_n$. The quasi-isolated elements of $\bG^*$ are $2$-elements,
hence $s=1$.
But then by \cite[Prop.~1.5]{En08}, $B$ is the principal block of
$S$, and in particular has non-abelian defect groups.
Thus, $S$ is of exceptional type. For the groups $\tw2B_2(q^2)$ the only bad
prime is the defining one, so this case does not occur here.
The height zero
conjecture for the groups $^2G_2(q^2)$, $G_2(q)$, $\tw3D_4(q)$ and
$\tw2F_4(q^2)$ has been checked by Ward \cite{Wa66}, Hi\ss\ \cite{Hi},
Deriziotis--Michler \cite{DeMi} and Malle \cite{MaF} respectively.
So $s$ is quasi-isolated in a quasi-simple exceptional group
of Lie type of rank at least~4 and $p$ is a bad prime. In this case the
claim is contained in Propositions~\ref{prop:excunibad}
and~\ref{prop:exc-height0}.

This completes the proof of Theorem~\ref {thm:BrauerHZC}.
\end{proof}

\section{Blocks with equal height zero degrees} \label{sec:nilblocks}

Here, we complete the proof of Theorem~\ref{thm:nil}. According to the
result in \cite[Thm.~6.1]{MN} the only blocks left to consider are spin
blocks (i.e., faithful blocks) of the double cover of alternating groups,
and quasi-isolated blocks of exceptional groups of Lie type of rank at
least~4. The validity of Theorem~\ref{thm:nil} for spin blocks of alternating
groups has recently been shown by Gramain \cite[Thm.~4.1 and Cor.~4.2]{Gr10}.
\par
So we may assume that $S$ is quasi-simple of exceptional Lie type in
characteristic~$p$, and $B$ is a quasi-isolated $\ell$-block of $S$ with
$\ell\ne p$. It is immediate from our explicit description of such blocks in
Sections~\ref{sec:F4}--\ref{sec:E8} and the degree formula resulting from
Lusztig's Jordan decomposition of characters that the only quasi-isolated
$\ell$-blocks with all height zero characters of the same degree are those
consisting of a single cuspidal character. In those cases, the defect groups
are central, and in particular abelian. Together with the criterion in
\cite[Thm.~4.1]{MN}, the claim follows.

\vskip 2pc\noindent
\textbf{Acknowledgements:} We thank Marc Cabanes for several hints and
explanations, in particular for showing us how to avoid having to treat the
three groups $\tw2E_6(2),E_7(2)$ and $E_8(2)$ separately in the proof of
Proposition~\ref{prop:Cabanesjordancrux}, Michel Brou\'e for a helpful
conversation on the topic of Section~\ref{sec:back}, and J\"urgen M\"uller for
determining the defect groups of some non-principal blocks of
sporadic simple groups. We thank Michel Enguehard for clarifying 
the structure of the defect groups of $2$-blocks of the classical groups.
Last but not least, we thank Jean Michel for
providing the relevant commands in the {\sf Chevie}-package, without which
the computations necessary for this paper would have been much more cumbersome.


\end{document}


\title{Supplement: Quasi-isolated blocks\\ and Brauer's height zero conjecture}

\date{\today}

\author{Radha Kessar}
\address{Department of Mathematics, City, University of London, Northampton Square, London N88HG, UK}
\email{radha.kessar.1@city.ac.uk}
\author{Gunter Malle}
\address{FB Mathematik, TU Kaiserslautern,
  Postfach 3049, 67653 Kaisers\-lautern, Germany.}
\email{malle@mathematik.uni-kl.de}

\begin{abstract}
We treat a configuration for the $5$-blocks of the finite simple groups
$E_8(q)$ that was inadvertently omitted in \cite{KM} and clarify and correct
the meaning of some entries in Tables~3 and~4 from \cite{KM}. The results stated
in the introductions of \cite{KM}, \cite{KM15} and \cite{KM17} remain unaffected.
\end{abstract}

\maketitle

\pagestyle{myheadings}
\markboth{R. Kessar and G. Malle}{Supplements on quasi-isolated blocks}

\section{Certain 5-blocks in groups of type $E_8$} \label{sec:5-blocks}
We deal with a situation missed in our paper \cite{KM}. 
Let $\bG$ be a simple linear algebraic group of type $E_8$ with a Frobenius
endomorphism $F:\bG\rightarrow\bG$ such that $G=\bG^F=E_8(q)$ for a prime
power~$q$. Let $G^*$ be dual to $G$. If $q\equiv\pm1\pmod6$ there exists an
isolated element $s\in G^*\cong G$ of order six with centraliser
$C_{\bG^*}(s)$ of type $A_5+A_2+A_1$. It was inadvertently left
out of \cite[Tab.~1]{KM} (probably as its order is divisible by two distinct
bad primes; but type $E_8$ has three bad primes, viz.~$2,3$ and~$5$). The centraliser of $s$ in
$\bG^*$ is of rational type $A_5(q).A_2(q).A_1(q)$ if $q\equiv1\pmod6$,
and $\tw2A_5(q).\tw2A_2(q).A_1(q)$ if $q\equiv5\pmod6$. We determine the
$\ell$-blocks in $\cE_\ell(\bG^F,s)$ for the only relevant bad prime $\ell=5$.

\begin{thm}   \label{thm:E8 missing}
 Let $q$ be a prime power with $5{\not|}q$. Theorems~$1.2$, $1.4$ and~$1.5$ of
 \cite{KM} continue to hold for the isolated 5-blocks of $G^F=E_8(q)$ in
 $\cE(\bG^F,s)$ with $s$ as
 described above.
\end{thm}

\begin{proof}
The method is completely analogous to that employed in \cite{KM}. There are
three cases to distinguish, depending on whether the order $e:=e_5(q)$ of
$q$ modulo~5 is~1,2 or~4. First we determine the decomposition of the Lusztig
functor $\RLG$ for the relevant $e$-cuspidal pairs $(\bL,\la)$ below $\cE(G,s)$.
Since the unipotent characters of groups of type $A$ are uniform, and Jordan
decomposition commutes with Deligne--Lusztig induction of uniform functions,
this decomposition follows from the corresponding one for unipotent characters.
This also shows the fact that $\bG^F$ satisfies an $e$-Harish-Chandra theory
above each $e$-cuspidal pair below $\cE(\bG^F,s)$ in the sense of
\cite[Def.~2.9]{KM}, thus proving the assertion of \cite[Thm.~1.4]{KM} in this
case. The decomposition of $\RLG(\la)$ is given in Tables~\ref{tab:q=1(5)}
and~\ref{tab:q=2(5)}. The notation is as in the analogous tables in \cite{KM}.
Here, the case $e=2$ can be obtained by Ennola duality from the one for $e=1$,
and the case of centraliser $\tw2A_5(q).\tw2A_2(q).A_1(q)$ when
$q\equiv\pm2\pmod5$
from the one with centraliser $A_5(q).A_2(q).A_1(q)$ when $q\equiv\mp2\pmod5$.
(Note that the $e$-Harish-Chandra series in Lines~45--49 are exactly as for
the 2-blocks~1 and~2 of $E_7(q)$ in \cite[Tab.~4]{KM}.)

\begin{table}[htbp]
\caption{Isolated 5-blocks in $E_8(q)$, $q\equiv1\pmod5$}   \label{tab:q=1(5)}
\[\begin{array}{|r|r|llll|}
\hline
 \text{No.}& C_{\bG^*}(s)^F& \bL^F& C_{\bL^*}(s)^F& \la& W_{\bG^F}(\bL,\la)\\
\hline\hline
 45& A_5(q)A_2(q)A_1(q)& \emptyset& \bL^{*F}& 1& A_5\ti A_2\ti A_1\\
\hline
 46& \tw2A_5(q).\tw2A_2(q)A_1(q)& A_1^3& \Ph1^5\Ph2^3& 1& C_3\ti A_1\ti A_1 \\
 47&     & D_4& \Ph1^4\Ph2^2.\tw2A_2(q)& \phi_{21}& C_3\ti A_1\\
 48&     & D_6& \Ph1^2\Ph2.\tw2A_5(q)& \phi_{321}& A_1\ti A_1\\
 49&     & E_7& \Ph1.\tw2A_5(q).\tw2A_2(q)& \phi_{321}\otimes\phi_{21}& A_1\\
 \hline
\end{array}\]
\end{table}

\begin{table}[htbp]
\caption{Isolated 5-blocks in $E_8(q)$, $q\equiv\pm2\pmod5$}   \label{tab:q=2(5)}
\[\begin{array}{|r|r|llll|}
\hline
 \text{No.}& C_{\bG^*}(s)^F& \bL^F& C_{\bL^*}(s)^F& \la& W_{\bG^F}(\bL,\la)\\
\hline\hline
 50& A_5(q)A_2(q)A_1(q)& \Ph4.\tw2D_6(q)& \Ph1^3\Ph4.A_2(q)A_1(q)& 6\text{ chars}& Z_4\ti A_1\\
 51&  & \bG^F& C_{\bG^*}(s)^F& 18\text{ chars}& 1\\
\hline
\end{array}\]
\end{table}

We had already checked in \cite[Lemma~6.9]{KM} that all relevant $e$-split
Levi subgroups $\bL$ of $\bG$ satisfy
$$C_\bG(Z(\bL)_5^F)=\bL.$$
(In fact they all already occur in Lines~19--23 of Table~7 respectively in
Line~43 of Table~8 in \cite{KM}.) Furthermore, all relevant $e$-cuspidal
characters $\la$ are readily seen to be of central $5$-defect. But then by
\cite[Prop.~2.13 and~2.15]{KM} the two conditions in \cite[Prop.~2.12]{KM}
are satisfied and so for all relevant $e$-cuspidal pairs $(\bL,\la)$ all
constituents of $\RLG(\la)$ lie in a single $5$-block $b_{\bG^F}(\bL,\la)$ of
$\bG^F$. Moreover, $Z^{\circ}(\bL)^F \cap [\bL,\bL]^F$ is a $5'$-group, hence by
\cite[Prop.~2.7(g)]{KM}, in each case $(Z(\bL^F)_5,b_{\bL^F}(\la))$ is a
centric $b_{\bG^F}(\bL,\la)$-Brauer pair.   \par
If $(\bL,\la)$ corresponds to
Line~45 of Table~\ref{tab:q=1(5)}, then by \cite[Prop.~2.7(c)]{KM}, a defect
group of $b_{\bG^F}(\bL,\la)$ is an extension of $Z(\bL^F)_5$ by a Sylow
$5$-subgroup of $W_{\bG^F}(\bL,\la)$. In all other cases, the relative Weyl
group is a $5'$-group, hence by \cite[Prop.~2.7]{KM}, $(Z(\bL)^F_5,\la)$ is a
maximal $b_{\bG^F}(\bL,\la)$-Brauer pair, and in particular $Z(\bL)^F_5$ is a
defect group of $b_{\bG^F}(\bL, \la)$. Thus the defect groups of the various
blocks are as described in Theorem 1.2 of \cite{KM}.
\par
Since the orders of the Sylow $5$-subgroups of the various $Z(\bL)^F$ in
Lines~46--49 are all distinct, these lines correspond to different blocks. To
see that the blocks represented by the six characters of Line~50 are distinct,
note that since $\bL =C_\bG(Z(\bL)^F_5)$ and since the pairs $(\bL,\la)$
are not $\bG^F$-conjugate, neither are the corresponding maximal Brauer pairs
$(Z(\bL)^F_5,\la)$. The blocks corresponding to Line~51 are all of defect
zero, hence are distinct.
\end{proof}

The above results show that \cite[Thm.~A, Thm.~B]{KM15} remain unchanged
(note that Remark 2.2(4) of \cite{KM15} applies also to the element $s$ above).

We also obtain the following consequences, filling the gaps in the proofs
of Theorem~\cite[Thm.~1.1]{KM} and \cite[Main Theorem]{KM17} caused by the
missing case.

\begin{cor}
 The $5$-blocks in $\cE_5(\bG^F,s)$ with non-abelian defect group contain
 characters of positive height. Further, $(\bG^F,\chi)$ is not a minimal
 counter-example to (HZC1) for any semisimple $5$-element $t\in \bG^{*F}$
 commuting with $s$ and any $\chi\in\cE(\bG^F,st)$.
\end{cor}

\begin{proof}
The block in Line~45 of Table~\ref{tab:q=1(5)} has non-abelian defect groups (as
the relative Weyl group has order divisible by~5) and the character in
$\cE(\bG^F,s)$ corresponding to the unipotent character of $A_5(q)A_2(q)A_1(q)$
labelled by $\phi_{51}\otimes\phi_3\otimes\phi_2$ has positive height. The
blocks in all other lines have abelian defect groups. This proves the
first assertion.
\par
Suppose that $s$ corresponds to Lines~50 or~51. The blocks in Line~51 are all
of defect~0, and all remaining characters in $\cE_5(\bG^F,s)$ have the same
5-part in their degree, so are all of height~0 in their respective blocks.
In particular, the second assertion holds in these cases. Now suppose that $s$
corresponds
to Lines~46--49. Here we may apply \cite[Lemma~8.5(3)(b)]{KM} in conjunction
with \cite[Prop.~8.6(1)]{KM} to conclude that the second assertion holds
(see the argument in the last part of the proof of \cite[Prop.~8.8]{KM}).
\end{proof}

\section{Blocks of $E_6(q)$ and $E_7(q)$.}
We correct the interpretation of Lines 2, 5, 8 and 11 of Table~3 and Lines 6, 7,
10, 11, 14 and 20 of Table~4 of \cite{KM}.

\subsection{Disconnected centralisers} \label{sub:dis}
In the description of blocks via the tables listing $e$-cuspidal pairs in
\cite{KM}, we wrote that each numbered line gives rise to a unique
$e$-cuspidal series, hence to a unique $\ell$-block. The characters $\la$ in the
$e$-cuspidal pairs $(L, \la)$ (in series $s$) in the tables are labelled via a
character of $C_{\bL^*}^\circ(s)^F $ in Jordan correspondence with $\la $. In
the affected lines, $C_{\bL^*}(s)$ is disconnected, the listed character is
$C_{\bL^*}^\circ(s)^F $-stable and corresponds to
$|C_{\bL^*} (s)^F: C_{\bL^*}^\circ(s)^F|$ characters of $\bL$.

We show that in these cases there are in fact
$|C_{\bL^*}(s)^F: C_{\bL^*}^\circ (s)^F|$ blocks corresponding to each of the
lines which are transitively permuted by the group of inner-diagonal
automorphisms of $\bG^F$. We sketch the arguments for Table~3. The proofs for 
Table~4 are similar (see remark before Subsection~\ref{sub:sep}). We note that
the Table~4 cases were also picked up in
Paragraph 9 of the proof of Theorem~3.14 of \cite{KM15}. 

\begin{lem}
 For every line in \cite[Tab.~3]{KM} such that $C_{\bL^*}(s)^F\ne C_{\bL^*}^\circ(s)^F$,
 each entry in the $\la$-column corresponds to $3$ possibilities for $\la$. The
 $e$-cuspidal pairs $(\bL,\la)$ are in pairwise different $\bG^F$-conjugacy
 classes and give rise to pairwise disjoint $e$-Harish-Chandra series of
 $\bG^F$, but are conjugate under inner diagonal automorphisms of $\bG^F$. 
\end{lem}
 
\begin{proof}
It follows from Lusztig's Jordan decomposition of characters that the three
$e$-cuspidal characters in each case are fused by the diagonal outer automorphism
of order~3. Since $R_L^G$ commutes with diagonal automorphisms, this implies
the claim.
\end{proof} 

The proof of \cite[Prop.~4.2]{KM} carries over without change, hence all
irreducible characters in any listed $e$-Harish-Chandra series lie in the same
block. Moreover, the structure of the defect groups of the blocks corresponding
to any numbered line is as stated in \cite[Prop.~4.3]{KM}. Thus, in order to
complete the proof of \cite[Prop.~4.3]{KM} it remains to show: 

\begin{enumerate} 
\item The blocks of the three Harish-Chandra series in each of the numbered
 lines 2, 5, 8 and 11 of \cite[Tab.~3]{KM} are pairwise distinct.
\item Each of the three Harish-Chandra series in the unnumbered lines in
 Sections 5, 8 and 11 of \cite[Tab.~3]{KM} is contained in a (different) block
 corresponding to the numbered line in the same section.
\end{enumerate}
 
Recall from Section~3 of \cite{KM15} that an $\ell$-group $P$ is said to be
\emph{Cabanes} if $P$ contains a unique maximal normal abelian subgroup.

\begin{lem}   \label{l:defect}
 Let $b$ be a block corresponding to one of the numbered Lines $2$, $5$, $8$ or
 $11$ of \cite[Tab.~3]{KM} corresponding to the $e$-cuspidal pair $(\bL,\la)$
 and let $u$ be the block of $\bL^F $ containing~$\la$.
 There exists a maximal $b$-Brauer pair $(P, c)$ such that $(Z(\bL)^F_2,u)$ is
 normal in $(P,c)$ and such that the following holds.
 \begin{enumerate}[\rm(a)]
  \item In Lines $2$ and~$8$, $P$ is Cabanes with $Z(\bL)^F_2$ as unique maximal
   abelian normal subgroup.
  \item In Lines $5$ and~$11$, $ Z(\bL)^F_2 $ is the unique normal abelian
   subgroup of index~$4$ in $P$.
 \end{enumerate}
\end{lem} 

\begin{proof}
In Line 2, $Z(\bL)^F_2 =\langle x \rangle \times \langle y \rangle$ is a
homocyclic $2$-group of exponent $2^n=(q-1)_2$, $n\geq 2$, $|P: Z(\bL)^F_2|= 2$,
$P= Z(\bL)^F_2 \langle \sigma \rangle $ with $\sigma$ acting on $Z(\bL)^F_2$ by
either interchanging $x$ and $y$ or by sending $x$ to $(xy)^{-1}$ and
fixing~$y$. This is because $P$ is an extension of $Z(L)^F_\ell$ by a Sylow
$\ell$-subgroup of the relative Weyl group and this extension can already be
observed in the normaliser of a subsystem group $A_2(q)^3$. Thus $P$ is a Sylow
$2$-subgroup of $A_2(q)$. Since $n\geq 2$, $\sigma$ does not act quadratically
on $Z(L)^F_2$, thus the result holds by \cite[Prop.~2] {C94}. 

In Line 5, $Z(\bL)^F_2$ is as above, $|P: Z(\bL)^F_2|= 4$,
$P= Z(\bL)^F_2\langle \sigma, \tau \rangle$ where $\sigma$ acts on $Z(\bL)^F_2$
as above and $\tau$ acts by inversion. Let $Q$ be a normal abelian
subgroup of index $4$ of $P$ different from $ Z(\bL)^F_2$. Then either
$Z(\bL)^F_2 \cap Q $ has index $2$ in $Z(\bL)^F_2 $ and $Z(\bL)^F_2 \cap Q$ is
centralised by one of $\sigma $, $\tau$ or $\sigma \tau$ or $Z(\bL)^F_2\cap Q$
has index $4$ in $Z(\bL)^F_2$ and $Z(\bL)^F_2\cap Q$ is centralised by both
$\sigma$ and $\tau$. But the centraliser in $ Z(\bL)^F_2$ of any of $\sigma$,
$\tau$ or $\sigma\tau$ has index at least $4$ in $Z(\bL)^F_2$ and the
intersection of the centraliser in $Z(\bL)^F_2$ of $\sigma$ with the centraliser
in $Z(\bL)^F_2$ of $\tau$ has order $2$.

In Line 8, $P$ is semi-dihedral with $Z(\bL)^F_2$ the unique cyclic subgroup of
index $2$, hence the result holds.
The proof for Line 11 is similar to that for Line 5.
\end{proof} 

Now point (1) is a consequence of Lemmas~\ref{l:defect} and \ref{l:cabanes2}.
\medskip

Now for point (2). To prove the required statement for Line 5, we follow the
argument given in Paragraph 4 of the proof of \cite[Prop.~4.2]{KM}.
So we consider Lusztig
induction from the $2$-split Levi subgroup of Line 11 with $q\equiv 1 \pmod 4$.
Here also there are three distinct $2$- Harish-Chandra series transitively
permuted by $\Inndiag(\bG^F)$. The union of these series contains a
character of $\bG^F$ corresponding to $\tw3D_4[1]$ in the unnumbered line of
Section 5 of the table, as well as a character in one of the $2$-Harish-Chandra
series corresponding to Line 5. Since the series corresponding to Lines 11 and
those in both lines of Section 5 are conjugate by the action of
$\Inndiag(\bG^F)$, it follows that given any character $\mu$ of $\bG^F$
corresponding to $\tw3D_4[1]$ in the unnumbered line of Section 5, there is a
character $\nu$ of $\bG^F$ in a $2$-Harish-Chandra series for Line 11 such that
$\mu$ and $\nu$ lie in the same $2$-Harish-Chandra series of Line 11.
Similarly, given any character $\mu'$ of $\bG^F$ corresponding to $\tw3D_4[-1]$
in the unnumbered line of Section 5, there is a character $\nu'$ of $\bG^F$ in
a $2$-Harish-Chandra series for Line 11 such that $\mu'$ and $\nu'$ lie in the
same $2$-Harish-Chandra series of Line 11. Now point (2) follows since all
characters in a $2$-Harish-Chandra series of Line 11 lie in the same $2$-block.
We made this last assertion without proof in \cite{KM} but the argument is one
of the standard ones employed in the paper: one checks using CHEVIE \cite{Chv}
that Prop.~2.13(1) of \cite{KM} holds and by inspection of the table entry one
sees that Prop.~2.16(3) of \cite{KM} holds. The requisite conclusion now follows
by Props.~2.12, 2.13 and~2.12 of \cite{KM}. 
 
Point (2) for Line 8 follows similarly by using the 1-cuspidal pairs $(\bL,\la)$
in Line~2, with $q\equiv3\pmod4$. Here also there are $3$ distinct
$2$-Harish-Chandra series transitively permuted by $\Inndiag(\bG^F)$.
As stated in Paragraph 3 of the proof of \cite[Prop.~4.3]{KM}, we may verify
that $(\bL,\la)$ satisfies $\bL=C_\bG^\circ(Z(\bL)_2^F)$ and that $\la$ is of
central $2$-defect. We may conclude by \cite[Prop.~2.16]{KM}
that each 1-Harish-Chandra series in Line~2 lies in a unique 2-block. Since the
union of these series contains all the 2-Harish-Chandra series below Line~8 we
may conclude using the transitive action of $\Inndiag(\bG^F)$.
\medskip 

\begin{rem}
Via Ennola duality these arguments also apply to the groups of type $\tw2E_6$.
The same type of arguments are used for the Lines 6, 7, 10, 11, 14 and~20 of
\cite[Tab.~4]{KM}. For point (1), one checks that the defect group in all
relevant lines is Cabanes with maximal normal abelian subgroup $Z(\bL)^F_\ell$.
\end{rem}

\subsection{Separation of blocks} \label{sub:sep}
Let $\bG$ be connected reductive with Frobenius morphism $F$.
For $i=1,2,$ let $\bL_i$ be an $F$-stable Levi subgroup of $\bG$ with
$\la_i\in\cE(\bL_i^F,\ell')$, and let $u_i$ denote the $\ell$-block of $\bL_i^F$
containing $\la_i$. Suppose that $C_\bG(Z(\bL_i)^F_\ell) =\bL_i$ and that
$\la_i$ is of quasi-central $\ell$-defect. Then by
\cite[Props.~2.12, 2.13, 2.16]{KM} there exists a block $b_i$ of $\bG^F$ such
that all irreducible characters of $R_{\bL_i}^\bG (\la_i)$ lie in $b_i$
and $(Z(\bL_i)^F_\ell,u_i)$ is a $b_i$-Brauer pair. The assumption (a) case of
the following is Lemma~3.9 of \cite{KM15}.

\begin{lem}   \label{l:cabanes2}
 In the above situation, let $(P_i,c_i)$ be a maximal $b_i$-Brauer pair such
 that $(Z(\bL_i)^F_\ell, u_i)\leq (P_i,c_i)$. Assume that one of the following
 holds.
 \begin{enumerate}[\rm(1)]
  \item $P_i$ is Cabanes with $Z(\bL_i)^F_\ell$ as unique maximal abelian normal
   subgroup.
  \item $Z(\bL_1)^F_\ell$ is the unique subgroup of $P_1 $ of its isomorphism
   type and $Z(\bL_2)^F_\ell\cong Z(\bL_1)^F_\ell$. 
  \item $Z(\bL_1)^F_\ell$ is the unique normal subgroup of $P_1$ of its
   isomorphism type, $Z(\bL_2)^F_\ell\cong Z(\bL_1)^F_\ell$ and
   $Z(\bL_2)^F_\ell$ is normal in $P_2 $.
 \end{enumerate} 
 If $b_1=b_2$, then $(L_1,\la_1)$ and $ (L_2,\la_2)$ are $\bG^F$-conjugate. 
\end{lem} 
 
\begin{proof}
Suppose that $b_1=b_2$. Then, $(P_1,c_1)=\tw{g}(P_2,c_2)$ for some $g\in \bG^F$
and consequently $\tw{g}(Z(\bL_2)^F_\ell, u_2) \leq (P_1, u_1) $. By hypothesis
we have $\tw{g}Z(\bL_2)^F_\ell= Z(\bL_1)^F_\ell $. Hence the uniqueness of
Brauer pair inclusion implies that $\tw{g}(Z(\bL_2)^F_\ell,u_2)
 = (Z(\bL_1)^F_\ell,u_1)$. Now the result follows as in Lemma~3.9 of \cite{KM15}.
\end{proof}

\subsection{On the results in \cite{KM,KM15,KM17}} \label{sub:cons}
We now briefly outline why the new interpretation of the Tables~3 and~4 of
\cite{KM} does not have an impact on the other results of \cite{KM} and on the
results of \cite{KM15} and \cite{KM17}. With respect to the proof of the forward
direction (HZC1) of the height zero conjecture, the only lines corresponding to
abelian defect groups affected by the reinterpretation are Lines~7 and~11 of
\cite[Tab.~4]{KM}. The argument for these lines given in the proof of
\cite[Prop.~8.8]{KM} goes over without change and hence so does the statement of
\cite[Prop.~8.8]{KM}. We note here that there is a typographical error in
Paragraph~3 of the proof of Prop.~8.8 of \cite{KM}: $\cE_\ell(\bG^F, s)$
should be replaced by $\cE(\bG^F,s)$. Thus Theorem~1.1 of \cite{KM} is
unaffected. 
We have checked above that Theorems~1.2 and~1.4 of \cite{KM} are unaffected. It
can easily be checked that the reinterpretation of the tables has no effect on
the other results of the introduction of \cite{KM} nor on any results in
Sections~7, 8 and~9 of \cite{KM}.

The proof of parts (a),(b), and (c) of Theorem A of \cite{KM15} and the proof of
Theorem~B of \cite{KM15} remain unaffected by the reinterpretation of the
Tables~3 and~4 of \cite{KM}, and these tables are irrelevant for Part~(e) of
Theorem A of \cite{KM15}. Table~3 is irrelevant for Part~(d) of Theorem A of
\cite{KM15} and the proof of this part in \cite{KM15} was already given with the
new interpretation of \cite[Tab.~4]{KM}.

We now consider the Main Theorem of \cite{KM17}, that is, the reverse direction
of the height zero conjecture for quasi-simple groups. Proposition~2.19 of
\cite{KM17} continues to hold for the affected lines. For this, note that only
the case $Z=1$ in the statement of Prop.~2.19 of \cite{KM17} is relevant
here. In Lines~5, 8 and~11 of \cite[Tab.~3]{KM}, each of the three blocks
contains the union of at least two $e$-Harish-Chandra series (in $3'$-Lusztig
series) above $e$-cuspidal characters the $\ell$-part of whose degrees is
different. We may thus conclude as in the proof of Prop.~2.19 of
\cite{KM17}. The same argument applies to Line~10 of \cite[Tab.~4]{KM}. In
Line~2 of Table~3 and Lines~6, 14 and~20 of \cite[Tab.~4]{KM}, the relative
Weyl group has an irreducible character of degree divisible by $\ell$ and we
conclude by an appropriate modification of the proof of Prop.~2.19 of
\cite{KM17}: the intersection of $\cE(G^F,s)$ with the union of the blocks 
labelled by the relevant line is in $\ell$-defect preserving bijection with a
unipotent $e$-Harish-Chandra series of $C_{\bG^*}(s)$ whose relative Weyl group
is isomorphic to that of the relevant line, and hence by a straightforward
generalisation of the arguments of \cite[Cor.~6.6]{MaH} the union of the blocks
contains characters the $\ell$-part of whose degree is different. Since the
blocks in question are pairwise conjugate via an automorphism of $\bG^F$, the
same is also true for each of the blocks. Once we have \cite[Prop.~2.19]{KM17},
the proof of Theorem~2.20 and consequently of the Main Theorem of \cite{KM17}
goes over without change.
\medskip

\noindent {\bf Acknowledgements.}
We thank Niamh Farrell for pointing out the omission in the treatment of
$5$-blocks of $E_8(q)$, and Ruwen Hollenbach and Norman Macgregor for pointing
out the issue in \cite[Tab.~3]{KM}.
